\documentclass[a4paper]{article}

\usepackage{amssymb,amsmath,amsthm,nicefrac,bbm,enumerate,mathtools,cite,hyperref}

\newtheorem{defn}{Definition}

\newtheorem{thm}[defn]{Theorem}

\newtheorem{lem}[defn]{Lemma}
\newtheorem{cor}[defn]{Corollary}

\renewcommand{\P}{\mathbb{P}}
\newcommand{\E}{\mathbb E}
\newcommand{\N}{\mathbb N}

\newcommand{\Q}{\mathbb Q}
\newcommand{\R}{\mathbb R}

\begin{document}

\title{On arbitrarily slow convergence rates 
for strong numerical approximations of 
Cox-Ingersoll-Ross processes and
squared Bessel processes}

\author{Mario Hefter and Arnulf Jentzen
\medskip
\\
\small{TU Kaiserslautern (Germany) and ETH Zurich (Switzerland)}}

\maketitle

\begin{abstract}
Cox-Ingersoll-Ross (CIR) processes are extensively used in state-of-the-art models for the approximative pricing of financial derivatives.
In particular, CIR processes are day after day employed to model instantaneous variances (squared volatilities) of foreign exchange rates 
and stock prices in Heston-type models and 
they are also intensively used to model short-rate interest rates. 
The prices of the financial derivatives in the above mentioned models 
are very often approximately computed by means of
explicit or implicit Euler- or Milstein-type discretization methods based on equidistant evaluations of the driving noise processes.
In this article we study the strong convergence speeds of all such discretization methods.
More specifically,
the main result of this article reveals that
each such discretization method achieves at most a strong convergence order of $\delta/2$,
where $0<\delta<2$ is the dimension of the squared Bessel process associated to the considered CIR process.
In particular, we thereby reveal that
discretization methods currently employed in the financial industry 
may converge with arbitrarily slow strong convergence rates to 
the solution of the considered CIR process.
We thereby lay open the need of the development of 
other more sophisticated approximation methods which
are capable to solve CIR processes in the strong sense 
in a reasonable computational time and which thus can not
belong to the class of algorithms which use equidistant evaluations of the driving noise processes.
\end{abstract}

\tableofcontents

\section{Introduction}
Stochastic differential equations (SDEs) are a key ingredient in a number 
of models from economics and the natural sciences.
In particular, SDE based models are day after day used in the 
financial engineering industry to approximately compute prices
of financial derivatives. The SDEs appearing in such models are 
typically highly nonlinear and contain non-Lipschitz nonlinearities in the
drift or diffusion coefficient.
Such SDEs can in almost all cases not be solved explicitly and it has been and still is a very active topic of 
research to approximate SDEs with non-Lipschitz nonlinearities; 
see, e.g.,
Hu~\cite{MR1396331},
Gy{\"o}ngy~\cite{MR1625576},
Higham,
Mao, \&
Stuart~\cite{MR1949404},
Hutzenthaler,
Jentzen, \&
Kloeden~\cite{MR2985171},
Hutzenthaler \&
Jentzen~\cite{MR3364862},
Sabanis~\cite{MR3070913,MR3543890},
and the references 
mentioned therein.
In particular, in about the last five years
several results have been obtained that demonstrate that approximation 
schemes may converge arbitrarily slow,
see
Hairer,
Hutzenthaler, \&
Jentzen~\cite{MR3305998},
Jentzen,
M{\"u}ller-Gronbach, \&
Yaroslavtseva~\cite{MR3538358},
Yaroslavtseva \&
M{\"u}ller-Gronbach~\cite{2016arXiv160308686M},
Yaroslavtseva~\cite{2016arXiv160908073Y},
and
Gerencs{\'e}r,
Jentzen, \&
Salimova\cite{2017arXiv170203229G}.
For example, Theorem~1.2 in \cite{MR3538358} demonstrates that there exists an
SDE that has solutions with all moments bounded
but for which all approximation schemes that use only evaluation points of the driving Brownian motion converge in the strong sense with an arbitrarily slow rate;
see also
\cite[Theorem~1.3]{MR3305998},
\cite[Theorem~3]{2016arXiv160308686M},
\cite[Theorem~1]{2016arXiv160908073Y}, and
\cite[Theorem~1.2]{2017arXiv170203229G}
for related results.
All the SDEs in the above examples are purely academic with no connection to applications.
The key contribution 
of this work is to reveal that such slow convergence phenomena 
also arise in concrete models from applications. 
To be more specific, in this work we reveal that
Cox-Ingersoll-Ross (CIR) processes and squared Bessel processes
can in the strong sense in general not be
solved approximately in a reasonable computational time
by means of schemes using equidistant evaluations of the driving Brownian 
motion.
The precise formulation of our result is the subject of 
the following theorem.

\begin{thm}[Cox-Ingersoll-Ross processes]\label{maintheo}
Let 
$ T, a, \sigma \in (0,\infty) $, 
$ b, x \in [0,\infty) $
satisfy $ 2 a < \sigma^2 $,
let
$ 
  ( \Omega, \mathfrak{F}, \P ) 
$
be a probability space 
with a normal filtration $ ( \mathbb{F}_t )_{ t \in [0,T] } $,
let 
$ 
  W \colon [0,T] \times \Omega \to \R 
$
be a
$( \mathbb{F}_t )_{ t \in [0,T] }$-Brownian motion,
let $ X \colon [0,T] \times \Omega \to [0,\infty)$
be a $ ( \mathbb{F}_t )_{ t \in [0,T] } $-adapted stochastic process 
with continuous sample paths which satisfies 
for all $ t \in [0,T] $ 
$ \P $-a.s.\ that
\begin{equation}\label{besselsdeintro}
  X_t
  =
  x
  + 
  \int_0^t
  \left(
    a - b X_s
  \right)\mathrm{d}s
  +
  \int_0^t
  \sigma
  \sqrt{ X_s } \,\mathrm{d}W_s
  .
\end{equation}
Then there exists a real number $ c\in (0,\infty) $ such that 
for all $ N \in \N $ it holds that
\begin{equation}\label{maininequ}
  \inf_{
    \substack{
      \varphi \colon \R^N \to \R
    \\
      \text{Borel-measurable}
    }
  }
  \E\!\left[ 
    \big|
      X_T
      -
      \varphi( 
        W_{ \frac{ T }{ N } }, W_{ \frac{ 2 T }{ N } }, \dots, 
        W_T
      )
    \big|
  \right]
  \geq 
  c
  \cdot 
  N^{ 
    - 
    ( 2 a ) / \sigma^2 
  }.
\end{equation}
\end{thm}

Theorem~\ref{maintheo} is an immediate consequence of Theorem~\ref{thmgenpar} in Section~\ref{last} below.
Upper error bounds for strong approximation of CIR processes and squared Bessel processes,
i.e., the opposite question of Theorem~\ref{maintheo},
have been intensively studied in the literature;
see, e.g.,
Delstra \& Delbaan~\cite{MR1641781},
Alfonsi~\cite{MR2186814},
Higham \& Mao~\cite{strathprints160},
Berkaoui, Bossy, \& Diop~\cite{MR2367990},
Gy{\"o}ngy \& R{\'a}sonyi~\cite{MR2822773},
Dereich, Neuenkirch, \& Szpruch~\cite{MR2898556},
Alfonsi~\cite{MR3006996},
Hutzenthaler, Jentzen, \& Noll~\cite{2014arXiv1403.6385H},
Neuenkirch \& Szpruch~\cite{MR3248050},
Bossy \& Olivero Quinteros~\cite{2015arXiv150804581B},
Hutzenthaler \& Jentzen~\cite{MR3364862},
Chassagneux, Jacquier, \& Mihaylov~\cite{MR3582409},
Hefter \& Herzwurm~\cite{2016arXiv160101455H}, and
Hefter \& Herzwurm~\cite{2016arXiv160800410H}
(for further approximation results,
see, e.g.,
Milstein \& Schoenmakers~\cite{MR3433299},
Cozma \& Reisinger~\cite{2016arXiv160100919C},
and
Kelly \& Lord~\cite{2016arXiv161004003K}).
In the following we relate our result to these results.

Using the truncated Milstein scheme with the corresponding error bound from
Hefter \& Herzwurm~\cite{2016arXiv160800410H}
we get that the the lower bound obtained in
\eqref{maininequ}
is essentially sharp.
The precise formulation of this observation is the subject of the following corollary.

\begin{cor}[Cox-Ingersoll-Ross processes]
\label{cor:CIR}
Let 
$ T, a, \sigma \in (0,\infty) $, 
$ b, x \in [0,\infty) $
satisfy $4a<\sigma^2$,
let 
$ 
  ( \Omega, \mathfrak{F}, \P ) 
$
be a probability space 
with a normal filtration $ ( \mathbb{F}_t )_{ t \in [0,T] } $,
let 
$ 
  W \colon [0,T] \times \Omega \to \R 
$
be a
$ ( \mathbb{F}_t )_{ t \in [0,T] } $-Brownian motion,
let $ X \colon [0,T] \times \Omega \to [0,\infty)$
be a $ ( \mathbb{F}_t )_{ t \in [0,T] } $-adapted stochastic process 
with continuous sample paths which satisfies 
for all $ t \in [0,T] $ 
$ \P $-a.s.\ that
\begin{equation}
  X_t
  =
  x
  + 
  \int_0^t
  \left(
    a - b X_s
  \right)\mathrm{d}s
  +
  \int_0^t
  \sigma
  \sqrt{ X_s } \,\mathrm{d}W_s
  .
\end{equation}
Then there exist real numbers $ c, C \in (0,\infty) $
such that for all $ N \in \N $
it holds that
\begin{equation}\label{endcor}
c
  \cdot 
  N^{ 
    - 
    \nicefrac{ 2 a }{ \sigma^2 }
  }
  \leq
  \inf_{
    \substack{
      \varphi \colon \R^N \to \R
    \\
      \text{Borel-measurable}
    }
  }
  \E\!\left[ 
    \big|
      X_T
      -
      \varphi( 
        W_{ \frac{ T }{ N } }, W_{ \frac{ 2 T }{ N } }, \dots, 
        W_T
      )
    \big|
  \right]
  \leq
  C
  \cdot 
  N^{ 
    - 
    \nicefrac{ 2 a }{ \sigma^2 }
  }. 
\end{equation}
\end{cor}

The lower bound in \eqref{endcor}
is an immediate consequence of
Theorem~\ref{maintheo}
and the upper bound in \eqref{endcor}
is an immediate consequence of
Hefter \& Herzwurm~\cite[Theorem~2]{2016arXiv160800410H} using the truncated Milstein scheme.
We conjecture that in the full parameter range $a, \sigma \in (0,\infty) $ the convergence order in \eqref{endcor} is equal to $\min\{2a/\sigma^2,1\}$,
since for scalar SDEs with coefficients satisfying standard assumptions a convergence order of one is optimal;
see, e.g.,
Hofmann, M\"uller-Gronbach, \& Ritter~\cite{MR1817611}
and
M\"uller-Gronbach~\cite{MR2099646}.
Upper and lower error bounds for CIR processes are crucial due to the fact that CIR processes are a key ingredient in several models for the approximative pricing of financial derivatives on
stocks
(see, e.g.,
Heston~\cite{heston1993closed}),
interest rates
(see, e.g.,
Cox, Ingersoll, \& Ross~\cite{MR785475}),
and
foreign exchange markets
(see, e.g.,
Cozma \& Reisinger~\cite{2015arXiv150901479C}).

The remainder of this article is organized as follows.
In Section~\ref{basicsection} we review a few elementary properties of CIR processes and squared Bessel processes.
In Section~\ref{Basicsogeneral} we present some basic results for general SDEs.
In Section~\ref{preliminaries} we prove the lower error bound for a specific parameter range, which is then generalized in Section~\ref{last}.

\section{Basics of Cox-Ingersoll-Ross (CIR) processes and squared Bessel processes}
\label{basicsection}

\subsection{Setting}
\label{setbasic}

Let $(\Omega,\mathfrak{F},\P)$ be a complete probability space,
let $W\colon [0,\infty)\times\Omega \to\R$ be a Brownian motion,
and for every $\delta\in (0,\infty)$, $b, z \in [0,\infty)$
and every Brownian motion 
$V\colon [0,\infty)\times \Omega\to\R$ 
let $Z^{z,\delta,b,V}\colon [0,\infty) \times \Omega \to [0,\infty)$ 
be a $(\sigma_{ \Omega }(\{\{V_s\leq a\}\colon a\in\R, s\in [0,t]\}\cup\{A\in\mathfrak{F}\colon \P(A)=0\}))_{t\in [0,\infty)}$-adapted stochastic process with continuous sample paths 
which satisfies that for all $ t \in [0,\infty)$ it holds $ \P $-a.s.\ that
\begin{equation}
\label{BesselSDE}
  Z^{ z, \delta, b, V }_t
  = 
  z + 
  \int_0^t
  \left(
    \delta - b \, Z^{ z, \delta, b, V }_s
  \right)
  \mathrm{d}s
  +
  \int_0^t
    2
    \,
    \sqrt{ 
      Z^{ z, \delta, b, V }_s 
    } 
  \, \mathrm{d}W_s
  .
\end{equation}

\subsection{A comparison principle}

\begin{lem}
\label{lemcomparison}
Assume the setting in Section~\ref{setbasic} 
and
let $\delta\in (0,\infty)$,
$b_1,b_2, z_1,z_2 \in [0,\infty)$ satisfy $z_1\leq z_2$ and $b_1\geq b_2$.
Then
\begin{equation}
\label{eq:comparison}
  \P\!\left(
    \forall \, t \in [0,\infty)
    \colon
    Z_t^{ z_1, \delta, b_1, W }
    \leq 
    Z_t^{ z_2, \delta, b_2, W }
  \right)
  = 1 .
\end{equation}
\end{lem}

\begin{proof}[Proof of Lemma~\ref{lemcomparison}]
Equation~\eqref{eq:comparison} is, e.g., an immediate consequence 
of Karatzas \& Shreve~\cite[Proposition~5.2.18]{MR1121940}.
The proof of Lemma~\ref{lemcomparison} is thus completed.
\end{proof}

\subsection{A priori moment bounds}

\begin{lem}
\label{lemboundedmom}
Assume the setting in Section~\ref{setbasic}
and
let $\delta, T\in (0,\infty)$,
$b,z\in [0,\infty)$,
$p \in [1, \infty)$.
Then
\begin{equation}
\label{eq:exp_bound}
  \E\!\left[
    \sup_{ t \in [0,T] }
    \big|
      Z_t^{ z, \delta, b, W }
    \big|^p 
  \right]
  < \infty
  .
\end{equation}
\end{lem}

\begin{proof}[Proof of Lemma~\ref{lemboundedmom}]
Inequality~\eqref{eq:exp_bound} follows, e.g., 
from 
Mao~\cite[Corollary~2.4.2]{MR1475218}.
The proof of Lemma~\ref{lemboundedmom}
is thus completed.
\end{proof}

\subsection{Lipschitz continuity in the initial value}

In the next result, Lemma~\ref{expcomp}, 
we recall a well-known 
explicit formula for the first moments of CIR processes and squared Bessel processes
(cf., e.g., Cox, Ingersoll, \& Ross~\cite[Equation~(19)]{MR785475}).

\begin{lem}[An explicit formula for the first moment]
\label{expcomp}
Assume the setting in Section~\ref{setbasic}
and
let $\delta \in (0,\infty)$,
$b, z, t \in [0,\infty)$.
Then
\begin{equation}
\begin{split}
&
  \E\big[
    Z_t^{ z, \delta, b, W }
  \big]
\\ &
  =
    z
    \cdot 
    e^{ - b t } 
    + 
    \delta
    \cdot
    \left( 
    \int_0^t 
      e^{ - b s }  
      \, \mathrm{d}s 
    \right)  
  =
  z \cdot e^{ - b t }
  +
  \delta \cdot
\begin{cases}
  ( 1 - e^{ - b t } ) / b 
  &
  \colon b \neq 0
\\
  t 
  & 
  \colon b = 0
\end{cases}
  .
\end{split}
\end{equation}
\end{lem}

\begin{proof}[Proof of Lemma~\ref{expcomp}]
Throughout this proof let $ f \colon [0,\infty) \to \R $ 
be the function which satisfies for all $ r \in [0, \infty) $ that
\begin{equation}
  f(r) = 
  \E\big[
    Z_r^{ z, \delta, b, W }
  \big]
  .
\end{equation}
Observe that Lemma~\ref{lemboundedmom}, 
the fact that 
$ 
  Z_r^{ z, \delta, b, W }
$,
$ r \in [0,\infty) $,
is a stochastic process with continuous sample paths,
and Lebesgue's dominated convergence theorem
ensure that $ f $ 
is a continuous function.
This and \eqref{BesselSDE} 
show that for all $ r \in [0,\infty) $
it holds that
\begin{equation}
  f(r) 
  =  
  z
  +
  \int_0^r
  \E\!\left[
    \delta - b \, Z^{ z, \delta, b, V }_s
  \right]
  \mathrm{d}s
  =
  z
  +
  \int_0^r
  \left(
    \delta - b \, f(s)
  \right)
  \mathrm{d}s
  .
\end{equation}
This demonstrates that $ f $ is continuously differentiable 
and that for all $ r \in [0,\infty) $ it holds that
\begin{equation}
  f'(r) = \delta - b \, f(r)
  .
\end{equation}
Hence, we obtain that for all $ r \in [0,\infty) $
it holds that
\begin{equation}
\begin{split}
  f(r)
&
  = 
    e^{ - b r } z 
    + 
    \int_0^r 
      e^{ - b ( r - s ) } 
      \, \delta 
    \, \mathrm{d}s 
  = 
    e^{ - b r } z 
    + 
    \left( 
    \int_0^r 
      e^{ - b s }  
      \, \mathrm{d}s 
    \right)  
    \delta
  .
\end{split}
\end{equation}
This and the fact that 
\begin{multline}
  \forall \, \beta \in (0,\infty) 
  \colon 
    \int_0^t 
      e^{ - \beta s }  
    \, \mathrm{d}s 
=
    \frac{ 1 }{ - \beta }
    \int_0^t 
      - \beta
      \,
      e^{ - \beta s }  
    \, \mathrm{d}s
  =
    \tfrac{ 1 }{ - \beta }
    \left[
      e^{ - \beta s }  
    \right]^{ s = t }_{ s = 0 }
\\
  =
    \frac{ 
      e^{ - \beta t } 
      - 1
    }{
      - \beta
    }
  =
    \frac{ 
      1 - e^{ - \beta t } 
    }{
      \beta
    }
\end{multline}
complete the proof of Lemma~\ref{expcomp}.
\end{proof}

\begin{lem}[$L^1$-Lipschitz continuity]\label{lipsc}
Assume the setting in Section~\ref{setbasic}
and
let $\delta \in (0,\infty)$,
$b,z_1, z_2,t \in [0,\infty)$.
Then
\begin{equation}
\label{l1norm}
  \E\!\left[
    \big|
      Z_t^{ z_1, \delta, b, W }
      -
      Z_t^{ z_2, \delta, b, W }
    \big|
  \right]
  =
  e^{ - b t }
  \cdot 
  \left| z_1 - z_2 \right|
  .
\end{equation}
\end{lem}

\begin{proof}[Proof of Lemma~\ref{lipsc}]
Throughout this proof assume w.l.o.g.\ that $ z_1 \geq z_2 $. 
Next note that Lemma~\ref{lemcomparison}, Lemma~\ref{lemboundedmom}, 
and Lemma~\ref{expcomp} show that
\begin{equation}
\begin{split}
&
  \E\!\left[
    \big|
      Z_t^{ z_1, \delta, b, W }
      -
      Z_t^{ z_2, \delta, b, W }
    \big|
  \right]
  =
  \E\!\left[
    Z_t^{ z_1, \delta, b, W }
    -
    Z_t^{ z_2, \delta, b, W }
  \right]
\\ &
=
  \E\!\left[
    Z_t^{ z_1, \delta, b, W }
  \right]
  -
  \E\!\left[
    Z_t^{ z_2, \delta, b, W }
  \right]
=
  z_1 \cdot e^{ - b t }
  -
  z_2 \cdot e^{ - b t }
  =
  e^{ - b t }
  \cdot 
  \left| z_1 - z_2 \right|
  .
\end{split}
\end{equation}
The proof of Lemma~\ref{lipsc} is thus completed.
\end{proof}

\subsection{The scaling property}

\begin{lem}\label{lem101}
Assume the setting in Section~\ref{setbasic}
and
let $\delta,c \in (0,\infty)$,
$b, z \in [0,\infty)$.
Then
\begin{equation}
\label{scaling}
  \P\!\left( 
    \forall \, t \in [0,\infty)
    \colon 
    c \cdot 
    Z^{ 
      \nicefrac{ z }{ c }, 
      \delta, 
      c b, 
      ( 
        c^{ - 1 / 2 } W_{ c s }
      )_{
        s \in [0,\infty)
      }
    }_{
      \nicefrac{ t }{ c } 
    }
    =
    Z_t^{ z, \delta, b, W } 
  \right) = 1.
\end{equation}
\end{lem}

\begin{proof}[Proof of Lemma~\ref{lem101}]
Equation~\eqref{scaling} follows directly from 
the corresponding scaling property of Brownian motion 
and the stochastic integral 
(cf., e.g., Revuz \& Yor~\cite[Proposition~XI.1.6]{MR1725357}).
The proof of Lemma~\ref{lem101} is thus completed.
\end{proof}

\subsection{Hitting times}

\begin{lem}[The Feller boundary condition]
\label{lemhitzero}
Assume the setting in Section~\ref{setbasic}
and
let $\delta \in (0,\infty)$, $b, z \in [0,\infty)$.
Then
\begin{equation}
\label{hitzero}
\begin{split}
&
  \P\!\left(
    \forall \, t \in (0,\infty)
    \colon
    Z_t^{ z, \delta, b, W } > 0 
  \right)
\\ &
  =
  \P\!\left(
    \forall \, t \in (0,\infty)
    \colon
    Z_t^{ z, \delta, b, W } \neq 0 
  \right)
  =
  \begin{cases}
    1 
  & 
    \colon \delta \geq 2
  \\
    0 
  &
    \colon \delta < 2 
  \end{cases}
  .
\end{split}
\end{equation}
\end{lem}

\begin{proof}[Proof of Lemma~\ref{lemhitzero}]
First, observe that, e.g.,
G\"oing-Jaeschke \& Yor~\cite[page~315, first paragraph]{MR1997032}
(cf.\ also Revuz \& Yor~\cite[page~442]{MR1725357})
shows that
\begin{equation}
\begin{split}
&
  \P\!\left(
    \forall \, t \in (0,\infty)
    \colon
    Z_t^{ z, \delta, b, W } > 0 
  \right)
  =
  \begin{cases}
    1 
  & 
    \colon \delta \geq 2
  \\
    0 
  &
    \colon \delta < 2 
  \end{cases}
  .
\end{split}
\end{equation}
This and the fact that 
$
  \forall \, t \in [0,\infty), \, \omega \in \Omega  
  \colon
  Z^{ z, \delta, b, W }_t( \omega ) \in [0,\infty)
$
complete the 
proof of Lemma~\ref{lemhitzero}.
\end{proof}

\begin{lem}[Bounds for hitting times]
\label{hitnotzero}
Assume the setting in Section~\ref{setbasic}
and
let $\delta \in (0,2)$,
$b\in [0,\infty)$,
$T\in (0,\infty)$.
Then there exists a real number 
$
  c \in (0,\infty)
$ 
such that 
for every $ \varepsilon \in (0,T] $ 
it holds that
\begin{equation}
\P \! \left(
\inf_{t \in [\varepsilon,T]}
Z_{t}^{0,\delta,b,W}>0
\right)
\leq c\cdot \varepsilon^{1-\delta/2}.
\end{equation}
\end{lem}

\begin{proof}[Proof of Lemma~\ref{hitnotzero}]
Throughout this proof let $\nu \in (0,1)$ be the real number given by 
$
  \nu = 1 - \nicefrac{ \delta }{ 2 } 
$,
let 
$
  \Gamma \colon (0,\infty) \to (0,\infty)
$
be the function which satisfies for all 
$ r \in (0,\infty) $ 
that 
\begin{equation}
  \Gamma(r)
  =
  \int_0^{ \infty } x^{ r - 1 }
  \, e^{ - x }
  \, \mathrm{d}x
\end{equation}
(Gamma function),
and let 
$
  P \colon \R \times (0,\infty) \to [0,\infty)
$ 
be the function which satisfies for all 
$ z \in \R $, $ r \in (0,\infty) $ that
\begin{equation}
P(z,r)
=
\begin{cases}
  \P\big(
    \inf_{ t \in [0, r] }
    Z_t^{ z, \delta, 0, W } 
    > 0
  \big)
&
  \colon z \geq 0
\\
  0
  &
  \colon z < 0 
\end{cases}
  .
\end{equation}
There exists a real number $C\in(0,\infty)$ which satisfies for every $z,r \in (0,\infty)$ that
\begin{equation}
  P(z,r)
  = 
  C
  z^\nu
    \int_r^{ \infty } 
      t^{ - \nu - 1 } 
      \, 
      \exp\!\left( \tfrac{ -z }{ 2t } \right) 
    \mathrm{d}t
  ,
\end{equation}
see, e.g., Borodin \& Salminen~\cite[Part~I, Chapter~IV, Section~6, last equation in 46 on page 79]{MR1912205}
(with $\nu=\nu$, $y=\sqrt{z}$ in the notation of
\cite[Part~I, Chapter~IV, Section~6, last equation in 46 on page 79]{MR1912205}).
This and the fact that 
$
  \forall \, t,z \in (0,\infty) \colon
  \exp( \frac{ - z }{ 2 t } ) \leq 1
$
imply that for every $ z, r \in (0,\infty) $ it holds that
\begin{equation}
\label{eqzero2}
  P(z,r)
  \leq
  C 
  z^\nu
  \int_r^\infty 
  t^{ - ( 1 + \nu ) } \,
  \mathrm{d}t
  .
\end{equation}
In the next step we note that 
for every $ \varepsilon \in (0,\infty) $ 
it holds that the random variable 
$
  \varepsilon^{ - 1 } 
  Z_{ \varepsilon }^{ 0, \delta, 0, W} 
$ 
is 
$ \chi^2 $-distributed with $ \delta $ degrees of freedom
(see, e.g., Revuz \& Yor~\cite[Corollary~XI.1.4]{MR1725357}).
Hence, we obtain that for all $ \varepsilon \in (0,\infty) $
it holds that
\begin{equation}
\begin{split}
&
  \E\bigg[
    \Big|
      \tfrac{
        Z_\varepsilon^{ 0, \delta, 0, W } 
      }{
        \varepsilon
      }
    \Big|^\nu
  \bigg]
  =
  \int_0^{ \infty }
  x^{ \nu } 
  \left[
  \frac{
    \left[ \frac{ 1 }{ 2 } \right]^{ \nicefrac{ \delta }{ 2 } }
  }{
    \Gamma( \nicefrac{ \delta }{ 2 } )
  }
    x^{ 
      \nicefrac{ \delta }{ 2 } - 1
    }
    \exp\!\left(
      \tfrac{ - x }{ 2 } 
    \right)
  \right]
  \mathrm{d}x
\\ &
=
  \frac{
    \left[ \frac{ 1 }{ 2 } \right]^{ \nicefrac{ \delta }{ 2 } }
  }{
    \Gamma( \nicefrac{ \delta }{ 2 } )
  }
  \int_0^{ \infty }
    \exp\!\left(
      \tfrac{ - x }{ 2 } 
    \right)
  \mathrm{d}x
  =
  \frac{
    2 
    \left[ \frac{ 1 }{ 2 } \right]^{ \nicefrac{ \delta }{ 2 } }
  }{
    \Gamma( \nicefrac{ \delta }{ 2 } )
  }
  \int_0^{ \infty }
    e^{ - x }
  \, \mathrm{d}x
=
  \frac{
    2^{ 1 - \nicefrac{ \delta }{ 2 } }
  }{
    \Gamma( \nicefrac{ \delta }{ 2 } )
  }
  .
\end{split}
\end{equation}
This and \eqref{eqzero2} imply that 
for all $ r \in (0, \infty) $, $ \varepsilon \in (0,r) $ 
it holds that
\begin{equation}
\begin{split}
&
  \P\!\left(
    \inf_{ t \in [ \varepsilon, r ] }
    Z_t^{ 0, \delta, 0, W } > 0
  \right)
\\
&
  =
  \E\!\left[
    P\!\left(
      Z_{ \varepsilon }^{ 0, \delta, 0, W } ,
      r - \varepsilon
    \right)
  \right]
\leq 
  C \cdot 
  \E\!\left[
    \left|
      Z_{\varepsilon}^{0,\delta,0,W}
    \right|^\nu
  \right]
  \cdot 
  \int_{ r - \varepsilon }^\infty 
    t^{ - (1 + \nu) }
  \, \mathrm{d}t
\\
&
  = C 
  \cdot
  \int_{ r - \varepsilon }^\infty 
    t^{ - (1 + \nu) }
  \, \mathrm{d}t
  \cdot
  \E\bigg[
    \Big|
      \tfrac{
        Z_\varepsilon^{ 0, \delta, 0, W } 
      }{
        \varepsilon
      }
    \Big|^\nu
  \bigg]
  \cdot 
  \varepsilon^{ \nu }
\\
&
  = 
  C 
  \cdot
  \int_{ r - \varepsilon }^{ \infty } 
    t^{ - (1 + \nu) }
  \, \mathrm{d}t
  \cdot
  \left[
    \frac{
      2^{ 1 - \nicefrac{ \delta }{ 2 } }
    }{
      \Gamma( \nicefrac{ \delta }{ 2 } )
    }
  \right]
  \cdot 
  \varepsilon^{ \nu }
=
  C
  \cdot 
  \left[
    \frac{ 
      ( r - \varepsilon )^{ - \nu }
    }{ 
      \nu
    }
  \right]
  \cdot
  \left[
    \frac{
      2^{ 1 - \nicefrac{ \delta }{ 2 } }
    }{
      \Gamma( \nicefrac{ \delta }{ 2 } )
    }
  \right]
  \cdot 
  \varepsilon^{ \nu }
  .
\end{split}
\end{equation}
Therefore, we obtain 
that for all
$ \varepsilon \in (0,T) $ 
it holds that
\begin{equation}
\begin{split}
  \P\!\left(
    \inf_{ t \in [\varepsilon,T] }
    Z_t^{ 0, \delta, 0, W } > 0
  \right)
& 
\leq
  \left[
    \frac{ 
      ( T - \varepsilon )^{ \delta / 2 - 1 }
    }{ 
      ( 1 - \nicefrac{ \delta }{ 2 } )
    }
  \right]
  \cdot
  \left[
    \frac{
      C \cdot 
      2^{ 1 - \delta / 2 }
    }{
      \Gamma( \nicefrac{ \delta }{ 2 } )
    }
  \right]
  \cdot 
  \varepsilon^{ 1 - \delta / 2 }
\\ & =
  \left[
    \frac{ 
      C \cdot 
      2^{ 1 - \delta / 2 } 
      \cdot
      ( T - \varepsilon )^{ \delta / 2 - 1 }
    }{ 
      \Gamma( \nicefrac{ \delta }{ 2 } )
      \cdot
      ( 1 - \nicefrac{ \delta }{ 2 } )
    }
  \right]
  \cdot 
  \varepsilon^{ 1 - \delta / 2 }.
\end{split}
\end{equation}
This and Lemma~\ref{lemcomparison} 
show that
that for all
$ \varepsilon \in ( 0, \nicefrac{ T }{ 2 } ] $ 
it holds that
\begin{equation}
\begin{split}
&
  \P\!\left(
    \inf_{ t \in [\varepsilon,T] }
    Z_t^{ 0, \delta, b, W } > 0
  \right)
\leq
  \P\!\left(
    \inf_{ t \in [\varepsilon,T] }
    Z_t^{ 0, \delta, 0, W } > 0
  \right)
\\ &
\leq
  \left[
    \frac{ 
      C \cdot 
      2^{ 1 - \delta / 2 } 
      \cdot
      ( \nicefrac{ T }{ 2 } )^{ \delta / 2 - 1 }
    }{ 
      \Gamma( \nicefrac{ \delta }{ 2 } )
      \cdot
      ( 1 - \nicefrac{ \delta }{ 2 } )
    }
  \right]
  \cdot 
  \varepsilon^{ 1 - \delta / 2 }
=
  \left[
    \frac{ 
      C \cdot 
      2^{ ( 2 - \delta ) } 
      \cdot
      T^{ \delta / 2 - 1 }
    }{ 
      \Gamma( \nicefrac{ \delta }{ 2 } )
      \cdot
      ( 1 - \nicefrac{ \delta }{ 2 } )
    }
  \right]
  \cdot 
  \varepsilon^{ 1 - \delta / 2 }
  .
\end{split}
\end{equation}
Hence, we obtain that
\begin{equation}
  \sup_{ \varepsilon \in ( 0, T / 2 ] }
  \left[
  \frac{
    \P\big(
      \inf_{ t \in [\varepsilon,T] }
      Z_t^{ 0, \delta, b, W } > 0
    \big)
  }{
    \varepsilon^{ 1 - \delta / 2 }
  }
  \right]
  < \infty
  .
\end{equation}
This assures that
\begin{equation}
\begin{split}
&
  \sup_{ \varepsilon \in ( 0, T ] }
  \left[
  \frac{
    \P\big(
      \inf_{ t \in [\varepsilon,T] }
      Z_t^{ 0, \delta, b, W } > 0
    \big)
  }{
    \varepsilon^{ 1 - \delta / 2 }
  }
  \right]
\\ &
\leq
  \sup_{ \varepsilon \in ( 0, T / 2 ] }
  \left[
  \frac{
    \P\big(
      \inf_{ t \in [\varepsilon,T] }
      Z_t^{ 0, \delta, b, W } > 0
    \big)
  }{
    \varepsilon^{ 1 - \delta / 2 }
  }
  \right]
  +
  \sup_{ \varepsilon \in ( T/2, T ] }
  \left[
  \frac{
    1
  }{
    \varepsilon^{ 1 - \delta / 2 }
  }
  \right]
\\ & 
=
  \sup_{ \varepsilon \in ( 0, T / 2 ] }
  \left[
  \frac{
    \P\big(
      \inf_{ t \in [\varepsilon,T] }
      Z_t^{ 0, \delta, b, W } > 0
    \big)
  }{
    \varepsilon^{ 1 - \delta / 2 }
  }
  \right]
  +
  \left[
  \frac{
    1
  }{
    ( \nicefrac{ T }{ 2 } )^{ 1 - \delta / 2 }
  }
  \right]
  < \infty
  .
\end{split}
\end{equation}
The proof of 
Lemma~\ref{hitnotzero} 
is thus completed.
\end{proof}

\section{Basics of general SDEs}\label{Basicsogeneral}

\subsection{Setting}
\label{settinggensde}

Let $\mathcal{Z}^{(\cdot),(\cdot)}
=
(\mathcal{Z}^{z,v})_{z \in \R,v \in C( [0,\infty), \R )}
\colon
\R\times C([0,\infty),\R)\to C([0,\infty),\R)$
be a Borel-measurable and 
universally adapted function 
(see Kallenberg~\cite[page~423]{MR1876169} for the notion of an universally adapted function),
let $\alpha,\beta\colon \R\to\R$ be continuous functions,
assume that for every complete probability space $(\Omega, \mathfrak{F},\P)$,
every normal filtration $\mathbb{F}=(\mathbb{F}_t)_{t\in [0,\infty)}$
on $(\Omega, \mathfrak{F},\P)$,
every $\mathbb{F}$-Brownian motion $W\colon [0,\infty)\times \Omega\to\R$,
all sample paths continuous $\mathbb{F}$-adapted stochastic processes $Z^{(1)},Z^{(2)}\colon [0,\infty)\times \Omega\to\R$
with
$\forall\, i\in \{1,2\},\, t\in [0,\infty)\colon \P\big(Z_t^{(i)}
=Z_0^{(1)}
+\int_0^t \alpha(Z_s^{(i)})\,\mathrm{d}s
+\int_0^t \beta(Z_s^{(i)})\,\mathrm{d}W_s\big)=1$,
and every $t\in [0,\infty)$ it holds that
\begin{align}
\P\!\left(Z^{(1)}_t=Z^{(2)}_t\right)=1,
\end{align}
assume that for every complete probability space $(\Omega, \mathfrak{F},\P)$,
every normal filtration $\mathbb{F}=(\mathbb{F}_t)_{t\in [0,\infty)}$
on $(\Omega, \mathfrak{F},\P)$,
every $\mathbb{F}$-Brownian motion $W\colon [0,\infty)\times \Omega\to\R$,
every
$ \mathbb{F}_0 $/$ \mathcal{B}( \R )$-measurable
function
$Z\colon \Omega\to\R$,
and every $t\in [0,\infty)$ it holds that
\begin{align}
\P
\Big(
\mathcal{Z}^{Z,W}_t
=Z
+\smallint_0^t \alpha(\mathcal{Z}^{Z,W}_s)\,\mathrm{d}s
+\smallint_0^t \beta(\mathcal{Z}^{Z,W}_s)\,\mathrm{d}W_s
\Big)
=1,
\end{align}
and let $(\Omega,\mathfrak{F},\P)$ be a complete probability space.

\subsection{Brownian motion shifted by a stopping time}

\begin{lem}
\label{triv1}
Assume the setting in Section~\ref{settinggensde},
let $ \mathbb{F} = ( \mathbb{F}_t )_{ t \in [0,\infty) } $ 
be a normal filtration on $ ( \Omega, \mathfrak{F}, \P ) $,
let 
$
  W \colon [0,\infty) \times \Omega \to \R 
$ 
be a $ \mathbb{F} $-Brownian motion,
let $ \tau\colon \Omega\to [0,\infty)$ be a $\mathbb{F}$-stopping time,
let $ {Z}\colon \Omega \to \R $ 
be a $ \mathbb{F}_0 $/$ \mathcal{B}( \R )$-measurable function,
let $ \tilde{W}\colon [0,\infty)\times\Omega\to\R$ be the stochastic process which satisfies for all $t\in [0,\infty)$ that $\tilde{W}_t=W_{t+\tau}-W_\tau$,
and let $\tilde{Z}\colon\Omega\to\R$
be the random variable given by $\tilde{Z}=\mathcal{Z}^{{Z},W}_\tau$.
Then
\begin{enumerate}[(i)]

\item 
it holds that $ \tilde{W} $ is a Brownian motion,

\item
it holds that $ \tilde{W} $ and $ \tilde{Z} $ are independent,
and 

\item
it holds that
\begin{equation}
\label{claim}
  \P\!\left(
    \forall \, t \in [0,\infty)
    \colon
    \mathcal{Z}^{ \tilde{Z}, \tilde{W} }_t
    =
    \mathcal{Z}^{ {Z}, W }_{ t + \tau }
  \right) = 1 .
\end{equation}
\end{enumerate}
\end{lem}

\begin{proof}[Proof of Lemma~\ref{triv1}]
Throughout this proof let $\tilde {\mathbb{F}}=(\tilde {\mathbb{F}}_t)_{t\in [0,\infty)}$ be the normal filtration on $(\Omega,\mathfrak{F},\P)$ which satisfies for all $t\in [0,\infty)$ that $\tilde {\mathbb{F}}_t=\mathbb{F}_{t+\tau}$.
Observe that the fact that the function $\mathcal{Z}^{{Z},W}_\tau$ is $ {\mathbb{F}}_\tau $/$ \mathcal{B}( \R ) $-measurable
ensures that
$\tilde{Z}$ is
$ \tilde {\mathbb{F}}_0 $/$ \mathcal{B}( \R ) $-measurable.
In addition, note that,
e.g., Kallenberg~\cite[Theorem~13.11]{MR1876169}
demonstrates that
$\tilde{W}$ is a
$
  \tilde {\mathbb{F}}
$-Brownian motion.
This and the fact that
$\tilde{Z}$ is
$ \tilde {\mathbb{F}}_0 $/$ \mathcal{B}( \R ) $-measurable
show that $\tilde{W}$ and $\tilde{Z}$ are independent.
Next observe that the stochastic process $(\mathcal{Z}^{{Z},W}_{t+\tau})_{t\in [0,\infty)}$
has continuous sample paths,
is $\tilde {\mathbb{F}}$-adapted,
and satisfies that for all $t\in [0,\infty)$ it holds $\P$-a.s.\ that
\begin{align}
\begin{split}
\mathcal{Z}^{{Z},W}_{t+\tau}
&=
{Z}
+\int_0^{t+\tau} \alpha(\mathcal{Z}^{{Z},W}_s)\,\mathrm{d}s
+\int_0^{t+\tau} \beta(\mathcal{Z}^{{Z},W}_s)\,\mathrm{d}W_s
\\
&=
{Z}
+\int_0^{\tau} \alpha(\mathcal{Z}^{{Z},W}_s)\,\mathrm{d}s
+\int_0^{\tau} \beta(\mathcal{Z}^{{Z},W}_s)\,\mathrm{d}W_s
\\
&\quad+
\int_\tau^{t+\tau} \alpha(\mathcal{Z}^{{Z},W}_s)\,\mathrm{d}s
+\int_\tau^{t+\tau} \beta(\mathcal{Z}^{{Z},W}_s)\,\mathrm{d}W_s
\\
&=
\tilde{Z}
+\int_0^{t} \alpha(\mathcal{Z}^{{Z},W}_{s+\tau})\,\mathrm{d}s
+\int_0^{t} \beta(\mathcal{Z}^{{Z},W}_{s+\tau})\,\mathrm{d}\tilde{W}_s.
\end{split}
\end{align}
This establishes \eqref{claim}.
The proof of Lemma~\ref{triv1} is thus completed.
\end{proof}

\begin{lem}
\label{triv2}
Assume the setting in Section~\ref{settinggensde},
let $W,\tilde{W}\colon [0,\infty)\times\Omega\to \R$ be Brownian motions,
let $\tau \colon\Omega\to[0,\infty]$ be a random variable,
assume for all $t\in [0,\infty)$ that
$\P(W_{t\wedge\tau}=\tilde{W}_{t\wedge\tau})=1$,
let $Z\colon \Omega\to\R$ be a random variable,
assume that $W$ and $Z$ are independent,
and assume that $\tilde{W}$ and $Z$ are independent.
Then 
\begin{equation}
\label{claim1}
  \P\!\left(
    \forall \, t \in [0,\infty) 
    \colon
    \big[
      \mathcal{Z}^{ Z , W }_t
      -
      \mathcal{Z}^{ Z , \tilde{W} }_t
    \big] 
    \,
    \mathbbm{1}^{ \Omega }_{
      \{ t \leq \tau \}
    }
    = 0
  \right)
  = 1 .
\end{equation}
\end{lem}

\begin{proof}[Proof of Lemma~\ref{triv2}]
Observe that it holds that
\begin{align}
\P\!
\left(
\forall\,t\in[0,\infty)\cap \Q\colon
\big[
t\leq \tau
\implies
W_{t}
=
\tilde{W}_{t}
\big]
\right)
=1.
\end{align}
The fact that
$W$ and $\tilde{W}$
have continuous sample paths hence shows that
\begin{align}\label{prev}
\P\!
\left(
\forall\,t\in[0,\infty)\colon
\big[
t\leq \tau
\implies
W_{t}
=
\tilde{W}_{t}
\big]
\right)
=1.
\end{align}
The assumption that $\mathcal{Z}^{(\cdot),(\cdot)}$ is universally adapted therefore proves that
\begin{align}
  \P\!\left(
    \forall \, t \in [0,\infty) \cap \Q
    \colon
    \big[
      \mathcal{Z}^{ Z , W }_t
      -
      \mathcal{Z}^{ Z , \tilde{W} }_t
    \big] 
    \,
    \mathbbm{1}^{ \Omega }_{
      \{ t \leq \tau \}
    }
    = 0
  \right)
  = 1 .
\end{align}
This and the fact that
the stochastic process
$
    \big[
      \mathcal{Z}^{ Z , W }_t
      -
      \mathcal{Z}^{ Z , \tilde{W} }_t
    \big] 
    \,
    \mathbbm{1}^{ \Omega }_{
      \{ t \leq \tau \}
    }
    \in \R
$,
$
  t \in [0,\infty)
$,
has left-continuous sample paths
establishes \eqref{claim1}.
The proof of Lemma~\ref{triv2} is thus completed.
\end{proof}

\begin{lem}
\label{concatbrownian}
Assume the setting in Section~\ref{settinggensde},
for every $m\in\{0,1\}$ let $\mathbb{F}^{(m)}=(\mathbb{F}^{(m)}_t)_{t\in [0,\infty)}$ be a normal filtration on $(\Omega,\mathfrak{F}, \P)$,
assume that
$(\cup_{t\in[0,\infty)}\mathbb{F}^{(0)}_t)\subseteq \mathbb{F}^{(1)}_0$,
for every $m\in\{0,1\}$ let $W^{(m)}\colon [0,\infty)\times\Omega\to\R$ be a $\mathbb{F}^{(m)}$-Brownian motion,
for every $m\in\{0,1\}$ let $\tau^{(m)}\colon\Omega\to [0,\infty)$ be a $\mathbb{F}^{(m)}$-stopping time,
let $\bar {\mathbb{F}}=(\bar {\mathbb{F}_t})_{t\in [0,\infty)}$ be the normal filtration on $(\Omega,\mathcal A,\P)$ which satisfies for all $t\in [0,\infty)$ that
\begin{align}
\bar {\mathbb{F}_t}
=\Big\{
A\in\mathfrak{F}\colon
A\cap\{t<\tau^{(0)}\}
\in \mathbb{F}^{(0)}_t
\text{ and }
A\cap\{t\geq \tau^{(0)}\}
\in \mathbb{F}^{(1)}_{\max\{t-\tau^{(0)},0\}}
\Big\},
\end{align}
let $\bar\tau\colon\Omega\to[0,\infty)$ be the random variable given by $\bar\tau
=\tau^{(0)}+\tau^{(1)}$,
let $\bar{W}\colon [0,\infty)\times\Omega\to\R$ be the stochastic process which satisfies for all $t\in [0,\infty)$ that
\begin{align}
\bar{W}_t
=
W_t^{(0)}
\mathbbm{1}^{ \Omega }_{
      \{ t \leq \tau^{(0)} \}}
+
(W^{(1)}_{|t-{\tau}^{(0)}|}+W^{(0)}_{\tau^{(0)}})
\mathbbm{1}^{ \Omega }_{
      \{ t > \tau^{(0)} \}},
\end{align}
let ${Z}\colon \Omega\to\R$ be a random variable which is
$ \mathbb{F}^{(0)}_0 $/$ \mathcal{B}( \R ) $-measurable,
and let $\tilde{Z}\colon \Omega\to\R$ be the random variable given by $\tilde{Z}=\mathcal{Z}_{\tau^{(0)}}^{{Z},W^{(0)}}$.
Then
\begin{enumerate}[(i)]
\item\label{item1}
it holds that $\mathbb{F}^{(0)}_0\subseteq \bar{\mathbb{F}_0}$,
\item\label{item2}
it holds that
$(\cup_{t\in[0,\infty)}\bar{\mathbb{F}}_t)\subseteq \sigma_{ \Omega }\big(\cup_{t\in [0,\infty)}\mathbb{F}^{(1)}_t\big)$,
\item\label{item3}
it holds that $\bar\tau$ is a $\bar{\mathbb{F}}$-stopping time,
\item\label{item4}
it holds that $\tau^{(0)}$ is a $\bar{\mathbb{F}}$-stopping time,
\item\label{item5}
it holds that $\bar{W}$ is a $\bar{\mathbb{F}}$-Brownian motion,
\item\label{item6}
it holds that $\bar{W}$ and ${Z}$ are independent,
\item\label{item7}
it holds that $W^{(1)}$ and $\tilde{Z}$ are independent,
and
\item\label{item8}
it holds that
\begin{align}
\P\Big(
\forall\,
t\in [0,\infty)
\colon
\mathcal{Z}^{{Z},\bar{W}}_t
=
\mathcal{Z}^{{Z},W^{(0)}}_t
\mathbbm{1}^{ \Omega }_{
      \{ t \leq \tau^{(0)} \}}
+
\mathcal{Z}^{\tilde{Z},W^{(1)}}_{|t-\tau^{(0)}|}
\mathbbm{1}^{ \Omega }_{
      \{ t > \tau^{(0)} \}}
\Big)
=1.
\end{align}
\end{enumerate}
\end{lem}

\begin{proof}[Proof of Lemma~\ref{concatbrownian}]
Throughout this proof let $F\colon \R\to (0,1)$ be the function which satisfies for all $x\in\R$ that $F(x)=\int_{-\infty}^x\frac{1}{\sqrt{2\pi}}e^{-t^2/2}\,\mathrm dt$ (distribution function of the standard normal distribution)
and for every $t\in [0,\infty)$ let $\rho_t\colon \Omega\to [0,\infty)$ be the random variable given by $\rho_t=\max\{t-\tau^{(0)},0\}$.
Observe that for every $t\in [0,\infty)$ it holds that $\rho_t$ is a $\mathbb{F}^{(1)}$-stopping time.
The fact that $\{0<\tau^{(0)}\}\in \mathbb{F}_0^{(0)}$ ensures that it holds for every $A\in \mathbb{F}_0^{(0)}$ that $A\cap\{0<\tau^{(0)}\}\in \mathbb{F}_0^{(0)}$
and
\begin{align}
A\cap \{0\geq \tau^{(0)}\}\in \mathbb{F}_0^{(0)}\subseteq \mathbb{F}_0^{(1)}=\mathbb{F}_{\max\{0-\tau^{(0)},0\}}^{(1)}.
\end{align}
This proves item~\eqref{item1}.
Next observe that for every
$t\in [0,\infty)$,
$A\in\bar{\mathbb{F}_t}$
it holds that
\begin{align}
A\cap \{t<\tau^{(0)}\}\in \mathbb{F}_t^{(0)}\subseteq \mathbb{F}_0^{(1)}\subseteq \sigma_{ \Omega }(\cup_{s\in [0,\infty)}\mathbb{F}_s^{(1)})
\end{align}
and
\begin{align}
A\cap \{t\geq \tau^{(0)}\}\in \mathbb{F}^{(1)}_{\max\{t-\tau^{(0)},0\}}\subseteq \sigma_{ \Omega }(\cup_{s\in [0,\infty)}\mathbb{F}_s^{(1)}).
\end{align}
Hence, we obtain for every
$t\in [0,\infty)$,
$A\in\bar{\mathbb{F}_t}$
that
\begin{align}
A=(A\cap \{t<\tau^{(0)}\})\cup (A\cap \{t\geq \tau^{(0)}\})\in \sigma_{ \Omega }(\cup_{s\in [0,\infty)}\mathbb{F}_s^{(1)}).
\end{align}
This proves item~\eqref{item2}.
Observe that for every $t\in [0,\infty)$ it holds that
\begin{align}
\{\bar\tau\leq t\}\cap \{t<\tau^{(0)}\}=\emptyset\in \mathbb{F}_t^{(0)}
\end{align}
and
\begin{align}
\begin{split}
&\{\bar\tau\leq t\}\cap \{t\geq \tau^{(0)}\}\\
&\quad=\{\bar\tau\leq t\}\\
&\quad=\{\tau^{(1)}\leq t-\tau^{(0)}\}\\
&\quad=\{\tau^{(1)}\leq \max\{t-\tau^{(0)},0\}\}\cap\{t-\tau^{(0)}\geq 0\}\in \mathbb{F}^{(1)}_{\max\{t-\tau^{(0)},0\}}.
\end{split}
\end{align}
This proves item~\eqref{item3}.
In the next step we note that for every $t\in [0,\infty)$ it holds that
\begin{align}
\{\tau^{(0)}\leq t\}\cap\{t<\tau^{(0)}\}=\emptyset\in \mathbb{F}^{(0)}_t
\end{align}
and
\begin{align}
\{\tau^{(0)}\leq t\}\cap \{t\geq \tau^{(0)}\}=\{\tau^{(0)}\leq t\}\in \mathbb{F}_t^{(0)}\subseteq \mathbb{F}_0^{(1)}\subseteq \mathbb{F}^{(1)}_{\max\{t-\tau^{(0)},0\}}.
\end{align}
This proves item~\eqref{item4}.
The strong Markov property of Brownian motion
(see, e.g.,
Kallenberg~\cite[Theorem~13.11]{MR1876169})
implies that it holds for every
$s\in [0,\infty)$
that $(W_{\rho_s+u}^{(1)}-W_{\rho_s}^{(1)})_{u\in [0,\infty)}$ is a Brownian motion 
independent of $\mathbb{F}^{(1)}_{\rho_s}$.
This and the fact that for every
$s\in [0,\infty)$,
$A\in\bar{\mathbb{F}_s}$
it holds that
$A\cap \{s\geq \tau^{(0)}\}\in \mathbb{F}^{(1)}_{\rho_s}$
demonstrate that for every
$s\in [0,\infty)$,
$t\in (s,\infty)$,
$A\in\bar{\mathbb{F}_s}$,
$a\in\R$
it holds that
\begin{align}\label{need1}
\begin{split}
&\P \!\left(
\{\bar{W}_t-\bar{W}_s\leq a\}
\cap
A
\cap
\{s\geq \tau^{(0)}\}
\right)
\\
&=
\P \!\left(
\{W_{\rho_s+t-s}^{(1)}
-W_{\rho_s}^{(1)}\leq a\}
\cap
A
\cap
\{s\geq \tau^{(0)}\}
\right)
\\
&=
F(a/\sqrt{t-s})\cdot
\P \!\left(
A
\cap
\{s\geq \tau^{(0)}\}
\right).
\end{split}
\end{align}
Observe that for every
$s\in [0,\infty)$,
$t\in (s,\infty)$,
$A\in\bar{\mathbb{F}_s}$
it holds
\begin{enumerate}[(a)]
\item\label{itemextra1} that
$W^{(0)}_{t\wedge\tau^{(0)}\vee s}-W_s^{(0)}$
is
$ \mathbb{F}_{t\wedge\tau^{(0)}\vee s}^{(0)} $/$ \mathcal{B}( \R ) $-measurable,
\item\label{itemextra2} that
$A\cap \{s<\tau^{(0)}\}\in\mathbb{F}_s^{(0)}\subseteq \mathbb{F}_{t\wedge\tau^{(0)}\vee s}^{(0)}$, and
\item\label{itemextra3} that
$t-(t\wedge\tau^{(0)}\vee s)$ is
$ \mathbb{F}_{t\wedge\tau^{(0)}\vee s}^{(0)} $/$ \mathcal{B}( \R ) $-measurable.
\end{enumerate}
Next note that the fact that
for every
$s\in [0,\infty)$,
$t\in (s,\infty)$
it holds that
$\mathbb{F}^{(0)}_{t\wedge\tau^{(0)}\vee s}\subseteq
\sigma_{ \Omega }(\cup_{u\in [0,\infty)}\mathbb{F}^{(0)}_u)\subseteq \mathbb{F}_0^{(1)}$
ensures that
for every
$s\in [0,\infty)$,
$t\in (s,\infty)$
it holds that the Brownian motion $W^{(1)}$ is independent of $\mathbb{F}^{(0)}_{t\wedge\tau^{(0)}\vee s}$.
Moreover, observe that the strong Markov property of Brownian motion (see, e.g.,
Kallenberg~\cite[Theorem~13.11]{MR1876169})
demonstrates that for every
$s\in [0,\infty)$,
$t\in (s,\infty)$
it holds that
$(W^{(0)}_{u+(t\wedge\tau^{(0)}\vee s)}-W^{(0)}_{t\wedge\tau^{(0)}\vee s})_{u\in [0,\infty)}$ is a Brownian motion 
which is independent of $\mathbb{F}^{(0)}_{t\wedge\tau^{(0)}\vee s}$.
Combining items \eqref{itemextra1}-\eqref{itemextra3}
with the fact that
for every
$s\in [0,\infty)$,
$t\in (s,\infty)$
it holds that the Brownian motion $W^{(1)}$ is independent of $\mathbb{F}^{(0)}_{t\wedge\tau^{(0)}\vee s}$
therefore demonstrates that
for every
$s\in [0,\infty)$,
$t\in (s,\infty)$,
$A\in\bar{\mathbb{F}_s}$,
$a\in\R$
it holds that
\begin{align}
\begin{split}
&\P \!\left(
\{\bar{W}_t-\bar{W}_s\leq a\}
\cap
A
\cap
\{s<\tau^{(0)}\}
\right)
\\
&=
\P \!\left(
\{W^{(1)}_{t-(t\wedge\tau^{(0)}\vee s)}
+W^{(0)}_{t\wedge\tau^{(0)}\vee s}
-W^{(0)}_s
\leq a\}
\cap
A
\cap
\{s<\tau^{(0)}\}
\right)
\\
&=
\P \!\left(
\{W^{(0)}_{t}
-W^{(0)}_s
\leq a\}
\cap
A
\cap
\{s<\tau^{(0)}\}
\right).
\end{split}
\end{align}
The fact that
for every
$s\in [0,\infty)$,
$A\in\bar{\mathbb{F}_s}$
it holds that
$A\cap \{s<\tau^{(0)}\}\in \mathbb{F}^{(0)}_{s}$
hence shows that for all
$s\in [0,\infty)$,
$t\in (s,\infty)$,
$A\in\bar{\mathbb{F}_s}$,
$a\in\R$
it holds that
\begin{align}\label{need2}
\begin{split}
&\P \!\left(
\{\bar{W}_t-\bar{W}_s\leq a\}
\cap
A
\cap
\{s<\tau^{(0)}\}
\right)
\\
&=
F(a/\sqrt{t-s})\cdot
\P \!\left(
A
\cap
\{s< \tau^{(0)}\}
\right).
\end{split}
\end{align}
Combining this and \eqref{need1} imply that it holds
for every
$s\in [0,\infty)$,
$t\in (s,\infty)$,
$A\in\bar{\mathbb{F}_s}$,
$a\in\R$
that
\begin{align}
\begin{split}
&\P \!\left(
\{\bar{W}_t-\bar{W}_s \leq a\}
\cap
A
\right)
\\
&=
\P \!\left(
\{\bar{W}_t-\bar{W}_s\leq a\}
\cap
A
\cap
\{s\geq \tau^{(0)}\}
\right)
\\
&\quad
+
\P \!\left(
\{\bar{W}_t-\bar{W}_s\leq a\}
\cap
A
\cap
\{s<\tau^{(0)}\}
\right)
\\
&=
F(a/\sqrt{t-s})
\cdot
\P(A).
\end{split}
\end{align}
This proves item~\eqref{item5}.
Item~\eqref{item1} implies that it holds that ${Z}$ is
$ \bar{\mathbb{F}}_0 $/$ \mathcal{B}( \R ) $-measurable.
This proves item~\eqref{item6}.
Observe that $\tilde{Z}$ is
$ \sigma_{ \Omega }(\cup_{u\in [0,\infty)}\mathbb{F}_u^{(0)}) $/$ \mathcal{B}( \R ) $-measurable.
The fact that $\sigma_{ \Omega }(\cup_{u\in [0,\infty)}\mathbb{F}_u^{(0)})\subseteq \mathbb{F}^{(1)}_0$
hence shows that $W^{(1)}$ and $\tilde{Z}$ are independent.
This proves item~\eqref{item7}.
Lemma~\ref{triv2} implies that
\begin{align}\label{abcprop1}
\P\!
\left(
\forall\,t\in[0,\infty)\colon
t\leq \tau^{(0)}
\implies
\mathcal{Z}^{{Z},\bar{W}}_{t}
=\mathcal{Z}^{{Z},W^{(0)}}_{t}
\right)
=1.
\end{align}
Therefore, we obtain that
\begin{align}\label{need1234}
\P(\tilde{Z}=\mathcal{Z}^{{Z},\bar{W}}_{\tau^{(0)}})
=1.
\end{align}
Next observe that it holds for all $t\in [0,\infty)$ that $\bar{W}_{t+\tau^{(0)}}-\bar{W}_{\tau^{(0)}}=W^{(1)}_t$.
Lemma~\ref{triv1}
and
\eqref{need1234}
hence imply
that
\begin{equation}
\label{abcprop2}
  \P\!\left(
    \forall \, t \in [0,\infty)
    \colon
    \mathcal{Z}^{ {Z}, \bar{W} }_{ t + \tau^{ (0) } }
    =
    \mathcal{Z}^{ \tilde{Z}, W^{ (1) } }_t
  \right)
  = 1 .
\end{equation}
Combining \eqref{abcprop1} and \eqref{abcprop2} establishes item~\eqref{item8}.
The proof of Lemma~\ref{concatbrownian} is thus completed.
\end{proof}

\subsection{A piecewise construction of a Brownian motion}

\begin{lem}
\label{lem1}
Assume the setting in Section~\ref{settinggensde},
for every $m\in\N_0$ let $\mathbb{F}^{(m)}=(\mathbb{F}^{(m)}_t)_{t\in [0,\infty)}$ be a normal filtration on $(\Omega,\mathfrak{F}, \P)$,
assume for all $m\in\N_0$ that $(\cup_{u\in[0,\infty)}\mathbb{F}^{(m)}_u)\subseteq \mathbb{F}^{(m+1)}_0$,
for every $m\in\N_0$ let $W^{(m)}\colon [0,\infty)\times\Omega\to\R$ be a $\mathbb{F}^{(m)}$-Brownian motion,
for every $m\in\N_0$ let $\tau^{(m)}\colon\Omega\to [0,\infty)$ be a $\mathbb{F}^{(m)}$-stopping time,
assume that $\sum_{m=0}^\infty\tau^{(m)}=\infty$,
for every $m\in\N_0$ let $T^{(m)}\colon \Omega\to [0,\infty)$ be the random variable given by $T^{(m)}=\sum_{i=0}^{m-1} \tau^{(i)}$,
let $W\colon [0,\infty)\times\Omega\to\R$ be the stochastic process which satisfies
for all
$m\in\N_0$, $t\in [0,\infty)$
that
$W_{0}=0$
and
\begin{equation}
\big[W_t-W_{T^{(m)}}
-
W^{(m)}_{|t-T^{(m)}|}\big]
\mathbbm{1}^{ \Omega }_{
\{
T^{(m)}
\leq t
\leq T^{(m+1)}
\}
}
=0,
\end{equation}
let $\bar{Z}\colon \Omega \to\R$ be a $ \mathbb{F}_0^{(0)} $/$ \mathcal{B}( \R ) $-measurable function,
let $Z^{(m)}\colon \Omega\to \R$, $m\in\N_0$, be the random variables which satisfy for every $m\in\N_0$ that
$Z^{(m)}$ is
$ \mathbb{F}_0^{(m)} $/$ \mathcal{B}( \R ) $-measurable,
$Z^{(0)}=\bar{Z}$,
and
\begin{equation}
Z^{(m+1)}
=\mathcal{Z}^{Z^{(m)},W^{(m)}}_{\tau^{(m)}},
\end{equation}
and let $\tilde{Z}\colon [0,\infty)\times\Omega\to \R$ be a stochastic process with continuous sample paths which satisfies that
\begin{equation}
\P\!
\left(
\forall\,m\in\N_0,
t\in [0,\infty)\colon
\big[
\tilde{Z}_t
-\mathcal{Z}^{Z^{(m)},W^{(m)}}_{|t-T^{(m)}|}
\big]
\mathbbm{1}^\Omega_{
\{
T^{(m)}
\leq t
\leq T^{(m+1)}
\}
}
=0
\right)
=1.
\end{equation}
Then
\begin{enumerate}[(i)]

\item
it holds that $W$ is a a Brownian motion,

\item
it holds that $W$ and $\bar{Z}$ are independent,
and

\item
it holds that
\begin{align}\label{ghtf}
\P\!
\left(
\forall\, t\in [0,\infty)\colon
\mathcal{Z}^{\bar{Z},W}_t
=\tilde{Z}_t
\right)
=1.
\end{align}
\end{enumerate}
\end{lem}

\begin{proof}[Proof of Lemma~\ref{lem1}]
Throughout this proof let $\mathcal D$ be the set given by
\begin{align}
\begin{split}
\mathcal D
=
\Big\{
&\big(
\mathfrak G^{(0)},\mathfrak G^{(1)},V^{(0)}, V^{(1)}, \upsilon^{(0)},\upsilon^{(1)}
\big)\colon
\\
&
\forall\,
i\in \{1,2\}\colon
\mathfrak G^{(i)}=(\mathfrak G_t^{(i)})_{t\in [0,\infty)}
\text{ is a normal filtration on }
(\Omega,\mathfrak{F},\P),
\\
&
\forall\,
i\in \{1,2\}\colon
V^{(i)}\colon [0,\infty)\times \Omega\to\R
\text{ is a }
\mathfrak G^{(i)}
\text{-Brownian motion},
\\
&
\forall\,
i\in \{1,2\}\colon
\upsilon^{(i)}\colon \Omega\to[0,\infty)
\text{ is a }
\mathfrak G^{(i)}
\text{-stopping time},
\\
&
\text{and }
(\cup_{u\in [0,\infty)}\mathfrak G^{(0)}_u)
\subseteq
\mathfrak G^{(1)}_0
\Big\},
\end{split}
\end{align}
let
\begin{align}
\bar{\mathbb{F}}\colon \mathcal D\to\{\mathfrak G=(\mathfrak G_t)_{t\in [0,\infty)}\text{ is a normal filtration on }(\Omega,\mathfrak{F},\P)\}
\end{align}
be the function which satisfies for all $(
\mathfrak G^{(0)},\mathfrak G^{(1)},V^{(0)}, V^{(1)}, \upsilon^{(0)},\upsilon^{(1)}
)\in\mathcal D$, $t\in [0,\infty)$ that
\begin{multline}
(\bar {\mathbb{F}}(
\mathfrak G^{(0)},\mathfrak G^{(1)},V^{(0)}, V^{(1)}, \upsilon^{(0)},\upsilon^{(1)}))_t
=\Big\{
A\in\mathfrak{F}:\\
\big(
A\cap\{t<\upsilon^{(0)}\}
\in \mathfrak G^{(0)}_t
\big)
\text{ and }
\big(
A\cap\{t\geq \upsilon^{(0)}\}
\in \mathfrak G^{(1)}_{\max\{t-\upsilon^{(0)},0\}}
\big)
\Big\},
\end{multline}
let
\begin{align}
\bar{\tau}
\colon
\mathcal D
\to
\{\upsilon\colon \Omega\to[0,\infty)
\text{ is a random variable}\}
\end{align}
be the function which satisfies for all $(
\mathfrak G^{(0)},\mathfrak G^{(1)},V^{(0)}, V^{(1)}, \upsilon^{(0)},\upsilon^{(1)}
)\in\mathcal D$ that
\begin{equation}
\bar\tau(
\mathfrak G^{(0)},\mathfrak G^{(1)},V^{(0)}, V^{(1)}, \upsilon^{(0)},\upsilon^{(1)}
)
=\upsilon^{(0)}+\upsilon^{(1)},
\end{equation}
let
\begin{align}
\bar{W}
\colon
\mathcal D
\to
\{V\colon [0,\infty)\times\Omega\to\R
\text{ is a stochastic process}\}
\end{align}
be the function which satisfies for all $(
\mathfrak G^{(0)},\mathfrak G^{(1)},V^{(0)}, V^{(1)}, \upsilon^{(0)},\upsilon^{(1)}
)\in\mathcal D$, $t\in [0,\infty)$ that
\begin{equation}
(\bar{W}(
\mathfrak G^{(0)},\mathfrak G^{(1)},V^{(0)}, V^{(1)}, \upsilon^{(0)},\upsilon^{(1)}
))_t
=
\begin{cases}
V^{(0)}_t &\colon t\leq \upsilon^{(0)}\\
V^{(1)}_{t-{\upsilon}^{(0)}}+V^{(0)}_{\upsilon^{(0)}} &\colon t\geq \upsilon^{(0)}
\end{cases},
\end{equation}
and for every $m\in\N_0$ let $\bar{\mathbb{F}}^{(m)}=(\bar{\mathbb{F}}^{(m)})_{t\in [0,\infty)}$ be the normal filtration on $(\Omega,\mathfrak{F},\P)$,
$\bar{W}^{(m)}\colon [0,\infty)\times\Omega\to\R$ be the $\bar{\mathbb{F}}^{(m)}$-Brownian motion,
and $\bar \tau^{(m)}\colon \Omega\to[0,\infty)$
be the $\bar{\mathbb{F}}^{(m)}$-stopping time which satisfy
for all $m\in\N$ that
\begin{align}
(\bar{\mathbb{F}}^{(0)}, \bar{W}^{(0)}, \bar \tau^{(0)})
=({\mathbb{F}}^{(0)},W^{(0)}, \tau^{(0)})
\end{align}
and
\begin{align}
\begin{split}
\bar{\mathbb{F}}^{(m)}
&=\bar{\mathbb{F}}
\!\left(
\bar{\mathbb{F}}^{(m-1)},
{\mathbb{F}}^{(m)},
\bar{W}^{(m-1)},
W^{(m)},
\bar \tau^{(m-1)},
\tau^{(m)}
\right),
\\
\bar{W}^{(m)}
&=
\bar{W}
\!\left(
\bar{\mathbb{F}}^{(m-1)},
{\mathbb{F}}^{(m)},
\bar{W}^{(m-1)},
W^{(m)},
\bar \tau^{(m-1)},
\tau^{(m)}
\right),
\\
\bar \tau^{(m)}
&=
\bar\tau
\!\left(
\bar{\mathbb{F}}^{(m-1)},
{\mathbb{F}}^{(m)},
\bar{W}^{(m-1)},
W^{(m)},
\bar \tau^{(m-1)},
\tau^{(m)}
\right)
\end{split}
\end{align}
(the unique existence of $(\bar{\mathbb{F}}^{(m)}, \bar{W}^{(m)}, \bar \tau^{(m)})$, $m\in\N_0$, follows from Lemma~\ref{concatbrownian}).
Observe that for every $m\in\N_0$, $t\in [0,\infty)$ it holds that $\bar\tau^{(m)}=T^{(m+1)}$
and
\begin{align}\label{propa}
\big[W_t
-
\bar{W}^{(m)}_t\big]
\mathbbm{1}^\Omega_{
\{
t\leq \bar \tau^{(m)}
\}
}
=0.
\end{align}
Hence, we obtain for every $t\in [0,\infty)$ that
\begin{align}\label{conv}
\lim_{m\to\infty}
\bar{W}^{(m)}_t
=W_t.
\end{align}
This shows that $W$ is a Brownian motion.
Lemma~\ref{concatbrownian} implies that for every $m\in\N_0$ it holds that $\bar{Z}$ is
$ \bar{\mathbb{F}}_0^{(m)} $/$ \mathcal{B}( \R ) $-measurable.
Therefore, we obtain for every $m\in\N_0$ that $\bar{Z}$ and $\bar{W}^{(m)}$ are independent.
Combining this with \eqref{conv} implies that $\bar{Z}$ and $W$ are independent.
In the following we show by induction that for all $m\in\N_0$ it holds
\begin{equation}\label{ind}
\P\big(
\forall\, t\in [0,\infty)\colon
\big[
\mathcal{Z}_t^{\bar{Z},\bar{W}^{(m)}}
-
\tilde{Z}_t
\big]
\mathbbm{1}^\Omega_{
\{
t\leq \bar \tau^{(m)}
\}
}
=0
\big)
=1.
\end{equation}
The induction base case $m=0$ is clear.
For the induction step $\N_0\ni m\to m+1\in\N$, assume that \eqref{ind} holds for some $m\in\N_0$.
Note that the induction hypothesis implies
that
\begin{align}
\P\big(
\mathcal{Z}^{\bar{Z},\bar{W}^{(m)}}_{\bar\tau^{(m)}}
=\tilde{Z}_{\bar\tau^{(m)}}
=\mathcal{Z}^{Z^{(m)},W^{(m)}}_{\tau^{(m)}}
=Z^{(m+1)}
\big)
=1.
\end{align}
Lemma~\ref{concatbrownian} hence implies that it holds $\P$-a.s.\ for all $t\in [0,\infty)$ it holds that
\begin{align}
\mathcal{Z}_t^{\bar{Z},\bar{W}^{(m+1)}}
=
\begin{cases}
\mathcal{Z}^{\bar{Z},\bar{W}^{(m)}}_t &\colon t\leq \bar \tau^{(m)}\\
\mathcal{Z}^{Z^{(m+1)},W^{(m+1)}}_{t-\bar\tau^{(m)}} &\colon t\geq \bar\tau^{(m)}
\end{cases}.
\end{align}
This proves \eqref{ind} in the case $m+1$.
Induction thus establishes \eqref{ind}.
Combining
\eqref{propa} and \eqref{ind}
with Lemma~\ref{triv2} demonstrates \eqref{ghtf}.
The proof of Lemma~\ref{lem1} is thus completed.
\end{proof}

\section{Lower error bounds for CIR processes and squared Bessel processes
in the case 
of a special choice of the parameters}
\label{preliminaries}

\subsection{Setting}
\label{setproc}

For every 
$ \delta \in (0,2) $, $ b \in [0,\infty) $
let
$
  \mathcal{Z}^{ ( \cdot ) ,\delta,b, (\cdot ) }
  =
  ( \mathcal{Z}^{ z, \delta, b, v } )_{
    z \in \R,
    v \in C( [0,\infty), \R )
  }
  \colon \R \times C([0,\infty),\R) \to C([0,\infty),\R)
$
be a Borel-measurable and 
universally adapted function 
(see Kallenberg~\cite[page~423]{MR1876169} for the notion of an universally adapted function)
which satisfies that for every complete probability space 
$ 
  ( \Omega, \mathfrak{F}, \P )
$,
every normal filtration 
$ 
  ( \mathbb{F}_t)_{ t \in [0,\infty) } 
$
on 
$
  ( \Omega, \mathfrak{F}, \P )
$,
every 
$ \mathbb{F}_0 $/$ \mathcal{B}( \R ) $-measurable
func\-tion $ Z \colon \Omega \to \R $,
every 
$
	( \mathbb{F}_t )_{ t \in [0,\infty) }
$-Brownian motion 
$ 
  W \colon [0,\infty) \times \Omega\to\R
$,
and every $ t \in [0,\infty) $
it holds $ \P $-a.s.\ that 
\begin{equation}
  \mathcal{Z}_t^{ Z, \delta, b, W }
  = X
  +\int_0^t \left( \delta - b \cdot \mathcal{Z}_s^{ Z, \delta, b, W } \right)\mathrm{d}s
  + \int_0^t 2 \sqrt{ | \mathcal{Z}_s^{ Z, \delta, b, W } | }\,\mathrm{d}W_s,
\end{equation}
let 
$ \delta \in (0,2) $, $ b \in [0,\infty) $,
let 
$ \mathcal{C}_0 $ and $ \mathcal{C}_{ 00 } $ 
be the sets 
given by
$
  \mathcal{C}_0
  =
  \{
    f\in C([0,\infty),\R)\colon f(0)=0
  \}
$
and
$
  \mathcal{C}_{00}
=
  \{
    f \in C([0,1],\R)
  \colon
    f(0)=f(1)=0
  \}
$,
let 
$
  \mathfrak{v}
  \colon 
  \{ \triangle, \Box \}
  \to \N
$
be the function which satisfies 
$
  \mathfrak{v}( \triangle ) = 3
$
and 
$
  \mathfrak{v}( \Box ) = 4
$,
for every 
$ n \in \N $,
$ \ast \in \{ \triangle, \Box \} $ 
let
$ 
  G_n^\ast\colon \mathcal{C}_0\times \mathcal{C}_{00}\to \mathcal{C}_0
$
be the function which satisfies for all
$
  w \in \mathcal{C}_0 
$, 
$ 
  f \in \mathcal{C}_{00} 
$, 
$
  t \in [0,\infty)
$
that
\begin{equation}
\begin{split}
&( G_n^\ast( w, f ) )_t
\\
&=
\begin{cases}
  \big(n \cdot w_{ 1 / n } \cdot t 
  + 
  \frac{ 1 }{ \sqrt{n} } \cdot f_{ n t }\big)\cdot (\mathfrak{v}(\ast)-3)
  +w_t \cdot (4-\mathfrak{v}(\ast))
&
  \colon 0 \leq t \leq \frac{ 1 }{ n }
\\
  w_t 
&
  \colon \frac{ 1 }{ n } \leq t < \infty
\end{cases}
  ,
\end{split}
\end{equation}
for every $ n \in \N $, $ \ast \in \{ \triangle, \Box \} $ let
$
  F^{ \ast }_n
  \colon [0,\infty) \times [\mathcal{C}_0]^3 \times \mathcal{C}_{00} \to \mathcal{C}_0
$
be the function which satisfies for all
$ 
  w^{ (1) } ,
  w^{ \triangle },
  w^{ (2) } \in \mathcal{C}_0
$,
$
  f \in\mathcal{C}_{00}
$, 
$ 
  r, t \in [0,\infty)
$ 
that
\begin{equation}
\begin{split}
&
  ( F_n^{ \ast }( r, w^{(1)}, w^{ \triangle }, w^{(2)},f) )_t
\\
&=
\begin{cases}
  w^{(1)}_t 
&
  \colon 
    t \leq r
\\
  ( G_n^\ast( w^{ \triangle },f) )_{ t - r } 
  +
  w^{(1)}_{ r } 
&
  \colon r \leq t \leq r + \frac{ 1 }{ n }
\\
  w^{(2)}_{ t - ( r + 1 / n ) }
  +
  w^{ \triangle }_{ 1 / n }
  +
  w^{ (1) }_{ r } 
  &
  \colon 
  r + \frac{ 1 }{ n } \leq t
\end{cases}
  ,
\end{split}
\end{equation}
for every $ n \in \N $,
$ k \in \{ 1, 2 \} $ 
let
$
  \mathfrak{T}^k_n \colon [0,\infty)\to [ \nicefrac{ ( k - 1 ) }{ n } , \nicefrac{ k }{ n } )
$
be the function which satisfies for all
$ t \in [0,\infty) $ 
that
\begin{equation}
  \mathfrak{T}^k_n(t)
  =
  \min\!\big( 
    \big\{
      0, \tfrac{ 1 }{ n }, 
      \tfrac{ 2 }{ n }, 
      \tfrac{ 3 }{ n }, 
      \dots 
    \big\}
    \cap 
    [ t, \infty )
  \big) 
  - t 
  +
  \tfrac{ ( k - 1 ) }{ n }
  ,
\end{equation}
for every $ n \in \N $ 
let
$
  \mathcal{S}_n\colon [0,\infty)\times [\mathcal{C}_0]^3\times \mathcal{C}_{00} \to [ \nicefrac{ 1 }{ n } , \infty ]
$,
$
  \mathcal{T}_n
  \colon [0,\infty) \times [\mathcal{C}_0]^3 \times \mathcal{C}_{00} 
  \to 
  [ \nicefrac{ 1 }{ n },\infty)
$,
and 
$
  \Phi_n 
  = 
  ( \Phi_{ n, 1 }, \dots, \Phi_{ n, 6 } )
  \colon
  [0,\infty)\times [\mathcal{C}_0]^3 \times \mathcal{C}_{00}
  \to
  [0,\infty)^2 \times [\mathcal{C}_0]^2 \times [C([0,\infty),\R)]^2
$
be the functions which satisfy
for all 
$ t \in [0,\infty) $, 
$ y \in [\mathcal{C}_0]^3 \times \mathcal{C}_{00} $ 
that
\begin{equation}
  \mathcal{S}_n( t, y )
  =
  \max_{\ast\in \{\triangle,\Box\}}
  \inf\!\left(
    \left\{
      s
      \in [ \mathfrak{T}^2_n(t) , \infty )
      \colon
      \mathcal{Z}^{ 0, \delta, b, F^\ast_n( \mathfrak{T}^1_n( t ) , y ) }_s
      = 0
    \right\}
  \cup 
    \{ \infty \}
  \right) 
  ,
\end{equation}
\begin{equation}
  \mathcal{T}_n(t, y)
  =
\begin{cases}
  \mathcal{S}_n( t, y) 
&
  \colon \mathcal{S}_n(t,y) \neq \infty
\\
  \mathfrak{T}^2_n( t ) 
&
  \colon \mathcal{S}_n( t, y ) = \infty
\end{cases}
  ,
\end{equation}
and
\begin{multline}
\label{eq:defPhi_n}
  \Phi_n( t, y )
  = 
  ( \Phi_{ n, 1 }( t, y ), \dots, \Phi_{ n, 6 }( t, y ) )
  =
  \big(
    t ,
    t + \mathcal{T}_n(t, y) ,
\\
    F^{ \triangle }_n( \mathfrak{T}^1_n(t) , y ),
    F^\Box_n( \mathfrak{T}^1_n(t), y ),
    \mathcal{Z}^{ 
      0, \delta, b, F^{ \triangle }_n( \mathfrak{T}^1_n(t), y ) 
    },
    \mathcal{Z}^{
      0, \delta, b, F^\Box_n( \mathfrak{T}^1_n( t ) , y ) 
    }
  \big)  
  ,
\end{multline}
let $(\Omega,\mathfrak{F},\P)$ be a complete probability space,
let 
$ \tilde{W}, \tilde{W}^{(1)},\tilde{W}^{ \triangle },\tilde{W}^{(2)} \colon \Omega\to\mathcal{C}_{ 0 }$ 
be Brownian motions,
let 
$ 
  B \colon \Omega \to \mathcal{C}_{ 00 }
$ 
be a Brownian bridge,
let 
$ Z \colon \Omega\to [0,\infty) $ 
be a random variable,
let 
$
  Y^{ [n] } 
  \colon \Omega \to 
  [ \mathcal{C}_0 ]^3 \times \mathcal{C}_{00}
$,
$ n \in \N_0 $,
be i.i.d.\ random variables 
with
$
  Y^{ [0] } = 
  ( \tilde{W}^{ (1) } , \tilde{W}^{ \triangle }, \tilde{W}^{ (2) }, B )
$,
let 
$
  X^{ (n), [m] }
  = ( X^{ (n), [m] }_1, \dots, X^{ (n), [m] }_6 )
  \colon 
  \Omega \to [0,\infty)^2 \times [\mathcal{C}_0]^2 \times [C([0,\infty),\R)]^2
$,
$ n \in \N $, 
$ m \in \N_0 $,
be the random variables which satisfy 
for all 
$ n, m \in \N $
that
$
  X^{(n),[m]}
  = \Phi_n(X^{(n),[m-1]}_2,Y^{[m]})
$
and
\begin{equation}
\label{eq:defX}
  X^{ (n), [0] }
  =
  \begin{cases}
    (0,0,\tilde{W},\tilde{W},\mathcal{Z}^{Z,\delta,b,\tilde{W}},\mathcal{Z}^{Z,\delta,b,\tilde{W}}) 
  \\
    \quad
    \colon
    ( \forall \, t \in [0,\infty) \colon \mathcal{Z}_t^{ Z, \delta, b, \tilde{W} } \neq 0 )
  \\
    (
      0,
      \inf\{ 
        t \in [0,\infty) \colon \mathcal{Z}^{ Z, \delta, b, \tilde{W}}_t = 0 
      \} ,
      \tilde{W},
      \tilde{W},
      \mathcal{Z}^{ Z, \delta, b, \tilde{W} },
      \mathcal{Z}^{ Z, \delta, b, \tilde{W} }
    )
  \\
  \quad
  \colon ( \exists \, t \in [0,\infty) \colon \mathcal{Z}_t^{ Z, \delta, b, \tilde{W} } = 0 )
\end{cases},
\end{equation}
for every 
$ n \in \N $,
$ \ast \in \{ \triangle, \Box \} $
let
$
  W^{ (n), \ast }
  \colon \Omega \to \mathcal{C}_{ 0 }
$
be a stochastic process which satisfies
for all
$ m \in \N_0 $,
$ t \in [0,\infty) $
that 
$
  W^{ (n), \ast }_0 
  = 0 
$
and
\begin{equation}
  \Big[
    W^{ (n), \ast }_t - 
    W^{ (n), \ast }_{ X_1^{ (n), [m] } }
    -
    ( X^{ (n), [m] }_{ \mathfrak{v}( \ast ) } )_{ 
      | t - X_1^{ (n), [m] } |
    }
  \Big]
  \mathbbm{1}^{ \Omega }_{
    \{ 
      X_1^{ (n), [m] }
      \leq t
      \leq 
      X_2^{ (n), [m] }
    \}
  }
  = 0
  ,
\end{equation}
for every $ n \in \N $,
$ \ast \in \{ \triangle, \Box \} $
let
$
  Z^{ (n), \ast }
  \colon [0,\infty) \times \Omega \to \R
$
be a stochastic process with continuous sample paths which satisfies
for all
$ m \in \N_0 $
that
\begin{equation}
  \P\Big(
    \forall \,
    t \in [0,\infty) 
    \colon
    \big[
      Z^{ (n), \ast }_t
      -
      ( X^{ (n), [m] }_{ \mathfrak{v}( \ast ) + 2 } )_{ | t - X_1^{ (n), [m] } | }
    \big] 
    \mathbbm{1}^{ \Omega }_{
      \{
        X_1^{ (n), [m] }
        \leq t
        \leq X_2^{ (n), [m] }
      \}
    }
    = 0
  \Big)
  = 1
  ,
\end{equation}
for every $ n \in \N $ let
$ 
  \mathcal{M}_n \colon \Omega \to \N_0
$
and 
$
  \gamma_n \colon \Omega \to [0,1] \cup \{ \infty \} 
$
be the random variables given by
$
  \mathcal{M}_n
  =
  \sup( 
    \{ 0 \} \cup
    \{
      m \in \{ 0, 1, \dots, n+1 \} 
      \colon 
      X_1^{ (n), [m] }
      \leq 1
    \}
  )
$
and
\begin{equation}
  \gamma_n
  =
  \begin{cases}
    X_1^{(n),[\mathcal{M}_n]} 
  &
    \colon \mathcal{M}_n \neq 0
  \\
    \infty 
  &
    \colon \mathcal{M}_n = 0
  \end{cases}
  ,
\end{equation}
and 
assume that 
$
  \tilde{W}
$,
$
  \tilde{W}^{ (1) } 
$, 
$  
  \tilde{W}^{ \triangle }
$, 
$
  \tilde{W}^{(2)}
$, 
$ B $, 
$ Z $, 
$ Y^{ [1] } $,
$ Y^{ [2] } $,
$ \dots $ 
are independent.

\subsection{Properties of the constructed random objects}

\subsubsection{The Feller boundary condition revisited}

\begin{lem}[Hit of the zero boundary]
\label{comptriv}
Assume the setting in Section~\ref{setproc}.
Then
\begin{equation}
\label{c41}
  \P\!\left(
    \exists \, t \in [0,\infty)
    \colon
    \mathcal{Z}_t^{ Z, \delta, b, \tilde{W} }
    = 0
  \right)
  = 1
  .
\end{equation}
\end{lem}

\begin{proof}[Proof of Lemma~\ref{comptriv}]
Note that the assumption that $ \delta \in (0,2) $ 
and Lemma~\ref{lemhitzero} ensure that for all $ z \in [0,\infty) $
it holds that
\begin{equation}
\label{eq:probab_zero}
  \P\!\left(
    \forall \, t \in [0,\infty)
    \colon
    \mathcal{Z}_t^{ z, \delta, b, \tilde{W} }
    \neq 0
  \right)
  = 0 .
\end{equation}
Next observe that the integral transformation theorem,
the fact that $ Z $ and $ \tilde{W} $ are independent,
and Fubini's theorem ensure that
\begin{equation}
\begin{split}
&
  \P\!\left(
    \forall \, t \in [0,\infty)
    \colon
    \mathcal{Z}_t^{ Z, \delta, b, \tilde{W} }
    \neq 0
  \right)
  =
  \E\!\left[ 
    \mathbbm{1}^{ \Omega }_{
      \{  
        \forall \, t \in [0,\infty)
        \colon
        \mathcal{Z}_t^{ Z, \delta, b, \tilde{W} }
        \neq 0
      \}
    }
  \right]
\\ &
=
  \E\!\left[ 
    \mathbbm{1}^{ C( [0,\infty), \R ) }_{
      \left\{  
        v \in C( [0,\infty) , \R )
        \colon
        ( 
          \forall \, t \in [0,\infty) \colon
          v(t) \neq 0
        )
      \right\}
    }(
      \mathcal{Z}^{ Z, \delta, b, \tilde{W} }
    )
  \right]
\\ &
=
  \int_{ 
    [0,\infty) 
    \times
    C( [0,\infty), \R ) 
  }
    \mathbbm{1}^{ C( [0,\infty), \R ) }_{
      \left\{  
        v \in C( [0,\infty) , \R )
        \colon
        ( 
          \forall \, t \in [0,\infty) \colon
          v(t) \neq 0
        )
      \right\}
    }(
      \mathcal{Z}^{ z, \delta, b, w }
    )
\\ &
\quad
  \big(
    ( Z, \tilde{W} )( \P )_{
      \mathcal{B}( [0,\infty) ) 
      \otimes
      \mathcal{B}( C( [0,\infty), \R ) )
    }
  \big)( dz, dw )
\\ &
=
  \int_{ [0,\infty) }
  \int_{ 
    C( [0,\infty),\R ) 
  }
    \mathbbm{1}^{ C( [0,\infty), \R ) }_{
      \left\{  
        v \in C( [0,\infty) , \R )
        \colon
        ( 
          \forall \, t \in [0,\infty) \colon
          v(t) \neq 0
        )
      \right\}
    }(
      \mathcal{Z}^{ z, \delta, b, w }
    )
\\ &
\quad
  \tilde{W}( \P )_{ 
    \mathcal{B}( C( [0,\infty), \R ) )
  }( dw )
  \,
  Z( \P )_{ 
    \mathcal{B}( [0,\infty) ) 
  }( dz ).
\end{split}
\end{equation}
Combining this and \eqref{eq:probab_zero} assures that
\begin{equation}
\begin{split}
&
  \P\!\left(
    \forall \, t \in [0,\infty)
    \colon
    \mathcal{Z}_t^{ Z, \delta, b, \tilde{W} }
    \neq 0
  \right)
\\ &
=
  \int_0^{ \infty }
  \E\!\left[ 
    \mathbbm{1}^{ C( [0,\infty), \R ) }_{
      \left\{  
        v \in C( [0,\infty) , \R )
        \colon
        ( 
          \forall \, t \in [0,\infty) \colon
          v(t) \neq 0
        )
      \right\}
    }(
      \mathcal{Z}^{ z, \delta, b, \tilde{W} }
    )
  \right]
  Z( \P )_{ \mathcal{B}( [0,\infty) ) }( dz )
\\ &
=
  \int_0^{ \infty }
  \E\!\left[ 
    \mathbbm{1}^{ \Omega }_{
      \left\{  
        \forall \, t \in [0,\infty)
        \colon
        \mathcal{Z}_t^{ z, \delta, b, \tilde{W} }
        \neq 0
      \right\}
    }
  \right]
  Z( \P )_{ \mathcal{B}( [0,\infty) ) }( dz )
\\ &
=
  \int_0^{ \infty }
  \P\!\left( 
        \forall \, t \in [0,\infty)
        \colon
        \mathcal{Z}_t^{ z, \delta, b, \tilde{W} }
        \neq 0
  \right)
  Z( \P )_{ \mathcal{B}( [0,\infty) ) }( dz )
  = 0 .
\end{split}
\end{equation}
Hence, we obtain that
\begin{equation}
\begin{split}
  \P\!\left(
    \exists \, t \in [0,\infty)
    \colon
    \mathcal{Z}_t^{ Z, \delta, b, \tilde{W} }
    = 0
  \right)
& = 
  1 
  -
  \P\!\left(
    \forall \, t \in [0,\infty)
    \colon
    \mathcal{Z}_t^{ Z, \delta, b, \tilde{W} }
    \neq 0
  \right)
  = 1 .
\end{split}
\end{equation}
The proof of Lemma~\ref{comptriv}
is thus completed.
\end{proof}

\subsubsection{One step in the construction of the Brownian motions}

In the next well-known lemma 
we briefly recall the covariance matrix associated to a Brownian bridge.

\begin{lem}[Covariance associated to a Brownian bridge]
\label{lem:covariance_bridge}
Let $ ( \Omega, \mathfrak{F}, \P ) $ be a probability space,
let $ T \in (0,\infty) $,
let $ W \colon [0,T] \times \Omega \to \R $
be a Brownian motion, 
and let $ B \colon [0,T] \times \Omega \to \R $ 
be the function which satisfies for all $ t \in [0,T] $
that
\begin{equation}
  B_t = W_t - \tfrac{ t }{ T } W_T
  .
\end{equation}
Then it holds for all $ s, t \in [0,T] $ that
\begin{equation}
  \E\!\left[ 
    B_s B_t
  \right]
  =
  \min\{ s, t \}
  -
  \tfrac{ s t }{ T }
\end{equation}
\end{lem}

\begin{proof}[Proof of Lemma~\ref{lem:covariance_bridge}]
Observe that the fact that
\begin{equation}
  \forall \, s, t \in [0,T] \colon
  \quad
  \E\!\left[ W_s W_t \right] = \min\{ s, t \}
\end{equation}
ensures that
for all $ s, t \in [0,T] $
it holds that
\begin{equation}
\begin{split}
  \E\!\left[ 
    B_s B_t
  \right]
&
  =
  \E\!\left[ 
    ( W_s - \tfrac{ s }{ T } W_T ) 
    ( W_t - \tfrac{ t }{ T } W_T )
  \right]
\\ &
=
  \E\!\left[ 
    W_s W_t
  \right]
  -
  \tfrac{ s }{ T } 
  \E\!\left[ 
    W_T 
    W_t
  \right]
  -
  \tfrac{ t }{ T } 
  \E\!\left[ 
    W_s 
    W_T
  \right]
  +
  \tfrac{ s t }{ T^2 } 
  \E\!\left[ 
    ( W_T )^2
  \right]
\\ &
=
  \min\{ s, t \}
  -
  \tfrac{ s t }{ T } 
  -
  \tfrac{ s t }{ T } 
  +
  \tfrac{ s t }{ T } 
=
  \min\{ s, t \}
  -
  \tfrac{ s t }{ T } 
  .
\end{split}
\end{equation}
The proof of Lemma~\ref{lem:covariance_bridge} is thus completed.
\end{proof}

\begin{lem}[Construction of a Brownian motion]
\label{lem:BW_construction}
Let $ ( \Omega, \mathfrak{F}, \P ) $ be a probability space,
let $ T \in (0,1] $,
let $ W \colon [0,1] \times \Omega \to \R $
be a Brownian motion, 
let $ B \colon [0,1] \times \Omega \to \R $
be a Brownian bridge,
assume that $ W $ and $ B $ are independent,
and let $ \mathcal{W} \colon [0,T] \times \Omega \to \R $ 
be the function which satisfies for all $ t \in [0,T] $
that
\begin{equation}
  \mathcal{W}_t = 
  \tfrac{ t }{ T } \cdot W_{ T }
    +
    \sqrt{ T } \cdot B_{ \frac{ t }{ T } }.
\end{equation}
Then it holds that 
$ \mathcal{W} $ is a Brownian motion.
\end{lem}

\begin{proof}[Proof of Lemma~\ref{lem:BW_construction}]
Note that Lemma~\ref{lem:covariance_bridge} and the assumption that
$ W $ and $ B $ are independent ensure that for all $ s, t \in [0,T] $
it holds that
\begin{equation}
\begin{split}
&
  \E\!\left[ 
    \mathcal{W}_s \mathcal{W}_t
  \right]
=
  \E\!\left[ 
    (
      \tfrac{ s }{ T } \cdot W_{ T }
      +
      \sqrt{ T } \cdot B_{ \frac{ s }{ T } }
    )
    (
      \tfrac{ t }{ T } \cdot W_{ T }
      +
      \sqrt{ T } \cdot B_{ \frac{ t }{ T } }
    )
  \right]
\\
&
  =
  \tfrac{ s t }{ T^2 }
  \cdot
  \E\!\left[  
    ( W_{ T } )^2
  \right]
  +
  \tfrac{ t }{ \sqrt{ T } }
  \cdot
  \E\big[  
    B_{ \frac{ s }{ T } }
    W_{ T } 
  \big]
  +
  \tfrac{ s }{ \sqrt{ T } }
  \cdot
  \E\big[  
    W_{ T } 
    B_{ \frac{ t }{ T } }
  \big]
  +
  T \cdot
  \E\big[ 
    B_{ \frac{ s }{ T } }
    B_{ \frac{ t }{ T } }
  \big]
\\
&
  =
  \tfrac{ s t }{ T }
  +
  \tfrac{ t }{ \sqrt{ T } }
  \cdot
  \E\big[  
    B_{ \frac{ s }{ T } }
  \big]
  \cdot
  \E\big[
    W_{ T } 
  \big]
  +
  \tfrac{ s }{ \sqrt{ T } }
  \cdot
  \E\big[  
    W_{ T } 
  \big]
  \cdot
  \E\big[
    B_{ \frac{ t }{ T } }
  \big]
  +
  T 
  \left(
    \min\{ 
      \tfrac{ s }{ T }
      ,
      \tfrac{ t }{ T }
    \}
    -
    \tfrac{ s t }{ T^2 }
  \right)
\\
&
  =
  \tfrac{ s t }{ T }
  +
    \min\{ 
      s
      ,
      t
    \}
    -
    \tfrac{ s t }{ T }
=
    \min\{ 
      s
      ,
      t
    \}
  .
\end{split}
\end{equation}
The proof of Lemma~\ref{lem:BW_construction} is thus completed.
\end{proof}

\begin{lem}[Construction of Brownian motions]
\label{comptriv1}
Assume the setting in Sec\-tion~\ref{setproc},
let $ n \in \N $,
let $ \tau \colon \Omega \to [0,\infty) $ be a random variable, 
assume that $ Y^{ [0] } $ and $ \tau $ are independent,
and let
$
  \tilde{W}^\Box,
  W^{ \triangle },
  W^\Box
  \colon \Omega \to \mathcal{C}_0
$
be the random variables given by
\begin{equation}
\begin{split}
  \tilde{W}^\Box
  =
  G^\Box_n( \tilde{W}^{ \triangle }, B ) ,
\quad
  W^{ \triangle }
  = F^{ \triangle }_n( \tau, Y^{ [0] } ) ,
\quad
  \text{and}
\quad
  W^\Box
  =
  F^\Box_n( \tau, Y^{ [0] } )
  .
\end{split}
\end{equation}
Then it holds that the stochastic processes
$\tilde{W}^\Box$,
$ W^{ \triangle } $,
and 
$W^\Box$ are Brownian motions.
\end{lem}

\begin{proof}[Proof of Lemma~\ref{comptriv1}]
In the case of a constant random variable $ \tau $ the claim 
follows from 
Lemma~\ref{lem:BW_construction}.
The case of a general $ \tau $ follows from the corresponding claim 
with a constant $ \tau $ by using the 
independence of $ Y^{[0]} $ and $ \tau $.
The proof of Lemma~\ref{comptriv1} is thus completed.
\end{proof}

\begin{lem}[One step in the construction of the Brownian motions]
\label{triv10}
Assume the setting in Section~\ref{setproc},
let $ n \in \N $,
let $ \tau \colon \Omega \to [0,\infty) $ be a random variable,
assume that $ Y^{ [0] } $ and $ \tau $ are independent,
let 
$ \rho \colon \Omega \to [ \nicefrac{ 1 }{ n }, \nicefrac{ 2 }{ n } ) $ 
be the random variable given by 
$
  \rho = \mathfrak{T}^2_n( \tau )
$,
and 
for every 
$ \ast \in \{ \triangle, \Box \} $ 
let 
$
  W^{ \ast } \colon \Omega \to \mathcal{C}_0
$
and 
$
  Z^{ \ast } \colon \Omega \to C( [0,\infty), \R ) 
$
be the random variables given by
\begin{equation}
  W^{ \ast } = F^{ \ast }_n( \mathfrak{T}^1_n( \tau ), Y^{ [0] } ) 
\qquad 
\text{and}
\qquad
  Z^{ \ast } = \mathcal{Z}^{ 0, \delta, b, W^{ \ast } } 
  .
\end{equation}
Then
\begin{enumerate}[(i)]

\item
\label{triv10item01}
it holds that
$  
  \tilde{W}^{ (2) }
$
and
$
  ( Z_{ \rho }^{ \triangle }, Z_{ \rho }^{ \Box } )
$ 
are independent,

\item
\label{triv10item01a}
it holds for every $ \ast \in \{ \triangle, \Box \} $ that
\begin{equation}\label{c2}
  \P\!\left(
    \forall \, t \in [0,\infty)
    \colon
    Z^{ \ast }_{ t + \rho }
    =
    \mathcal{Z}_t^{ Z_{ \rho }^{ \ast }, \delta, b, \tilde{W}^{ (2) } }
  \right)
  = 1 ,
\end{equation}
\item
\label{triv10item02}
it holds that
\begin{equation}\label{c3}
  \P\!\left(
    \big[
      Z^{ \triangle }_{ \rho } \geq Z_{ \rho }^\Box
    \big]
    \Longleftrightarrow
    \big[
      \forall \, t \in [0,\infty)
      \colon
      Z_{ t + \rho }^{ \triangle }
      \geq 
      Z^\Box_{ t + \rho }
    \big]
  \right)
  = 1,
\end{equation}
\item\label{triv10item02a}
it holds that
\begin{equation}\label{c4}
  \P\!\left(
    \big[
      Z^\Box_{ \rho } 
      \geq 
      Z_{ \rho }^{ \triangle }
    \big]
    \Longleftrightarrow
    \big[
      \forall \, t \in [0,\infty)
      \colon
      Z_{ t + \rho }^\Box
      \geq 
      Z^{ \triangle }_{ t + \rho }
    \big]
  \right)
  = 1,
\end{equation}
and 
\item
\label{triv10item03}
it holds that
\begin{multline}
\label{c5}
  \P\Big(
    \mathcal{S}_n( \tau, Y^{ [0] } )
    =
    \mathcal{T}_n( \tau, Y^{ [0] } )
    =
\\
    \inf\!\big( 
    \{ \infty \} \cup
    \big\{
      t \in [0,\infty)
      \colon 
      t \geq \rho
      \text{ and }
      \max\nolimits_{ \ast \in \{ \triangle, \Box \} }
        Z^{ \ast }_t
      = 0
    \big\}
    \big)
  \Big)
  = 1 .
\end{multline}
\end{enumerate}
\end{lem}

\begin{proof}[Proof of Lemma~\ref{triv10}]
We prove Lemma~\ref{triv10} in two steps.
In the first step we assume that there exists a real number 
$ 
  \mathfrak{t} \in [0,\infty)
$
such that
for all $ \omega \in \Omega $
it holds that
$ \tau( \omega ) = \mathfrak{t} $.
Observe that 
\begin{equation}
\label{eq:independence_W2}
  \tilde{W}^{ (2) } 
\qquad
\text{and} 
\qquad
  \big( 
    W^{ \triangle }|_{ [0, \mathfrak{T}^2_n(\mathfrak{t}) ] \times \Omega }
    , 
    W^{ \Box }|_{ [0, \mathfrak{T}^2_n(\mathfrak{t}) ] \times \Omega }
  \big)
\end{equation}
are independent. Moreover, note that
for every 
$
  \ast \in \{ \triangle, \Box \}
$,
$
  t \in [0,\infty)
$ 
it holds that
\begin{equation}
  W^{ \ast }_{ t + \mathfrak{T}^2_n(\mathfrak{t}) }
  -
  W^{ \ast }_{ \mathfrak{T}^2_n(\mathfrak{t}) }
  =
  \tilde{W}^{ (2) }_t .
\end{equation}
Combining this and \eqref{eq:independence_W2} 
proves items~\eqref{triv10item01}--\eqref{triv10item01a}.
Next note that Lemma~\ref{lemcomparison},
item~\eqref{triv10item01},
and item~\eqref{triv10item01a}
establish
item~\eqref{triv10item02} and item~\eqref{triv10item02a}.
Moreover, observe that
Lemma~\ref{lemhitzero} implies
that
\begin{equation}
\label{d1}
  \P( \mathcal{S}_n( \tau, Y^{ [0] } )
    =
    \mathcal{T}_n( \tau, Y^{ [0] } ) )
  = 1 .
\end{equation}
This, item~\eqref{triv10item02}, and item~\eqref{triv10item02a} 
establish item~\eqref{triv10item03}.
The case of a general $ \tau $ follows immediately 
from the case of a constant $ \tau $
by using the fact that $ Y^{ [0] } $ and $ \tau $ are independent.
The proof of Lemma~\ref{triv10} is thus completed.
\end{proof}

\subsubsection{Properties of the constructed random times}

\begin{lem}
\label{trivcor}
Assume the setting in Section~\ref{setproc} 
and let $ n \in \N $.
Then
\begin{enumerate}[(i)]
\item
\label{trivcorit1}
it holds 
for all $ m \in \N_0 $ 
that
\begin{equation}
\label{e2}
  0 \leq
  X_1^{ (n), [m] }
  \leq 
  X_2^{ (n), [m] } = X_1^{ (n), [m+1] }
  \leq 
  X_2^{ (n), [m+1] },
\end{equation}
\item 
it holds that
\begin{equation}
\label{e1}
  \sup_{ m \in \N_0 } X_1^{(n),[m]}
  =
  \sup_{ m \in \N_0 } X_2^{(n),[m]}
  =
  \infty ,
\end{equation}
\item
it holds for all $ m \in \N $, $ i \in \{ 5, 6 \} $ 
that
\begin{equation}\label{e3}
  \P\!\left(
    \big( X_i^{(n),[m]} \big)_{0}
    = 0
  \right)
  = 1 ,
\end{equation}
and
\item
\label{mnbbv}
it holds for all $ m \in \N_0 $, $ i \in \{ 5, 6 \} $ that
\begin{equation}
\label{impconst}
  \P\!\left(
    \big( 
      X_i^{(n),[m]}
    \big)_{ X_2^{ (n), [m] } - X_1^{ (n), [m] } }
    = 0
  \right)
  = 1 .
\end{equation}
\end{enumerate}
\end{lem}

\begin{proof}[Proof of Lemma~\ref{trivcor}]
First, observe that \eqref{e2} is a direct consequence from \eqref{eq:defX}.
Next note that for all 
$ t \in [0,\infty) $, 
$ y \in [\mathcal{C}_0]^3 \times \mathcal{C}_{00} $ 
it holds that 
\begin{equation}
  \mathcal{T}_n( t, y ) \geq 
  \tfrac{ 1 }{ n }
  .
\end{equation}
Combining this with \eqref{eq:defPhi_n} establishes \eqref{e1}.
In the next step we observe that 
for every $ m \in \N $ and every $ i \in \{ 3, 4 \} $ 
it holds that 
the stochastic process 
$ X_i^{ (n), [m] } $ is a Brownian motion.
This establishes \eqref{e3}.
It thus remains to prove \eqref{impconst}.
For this we note that Lemma~\ref{comptriv} assures that
for all $ i \in \{ 5, 6 \} $ it holds that
\begin{equation}
\label{impconst0}
  \P\!\left(
    \big( 
      X_i^{(n),[0]}
    \big)_{ X_2^{ (n), [0] } - X_1^{ (n), [0] } }
    = 0
  \right)
  = 1 .
\end{equation}
In addition, observe that
item~\eqref{triv10item03} of Lemma~\ref{triv10}
ensures that
for all $ m \in \N $, $ i \in \{ 5, 6 \} $ 
it holds that
\begin{equation}
\label{impconstm}
  \P\!\left(
    \big( 
      X_i^{(n),[m]}
    \big)_{ X_2^{ (n), [m] } - X_1^{ (n), [m] } }
    = 0
  \right)
  = 1 .
\end{equation}
Combining \eqref{impconst0} and \eqref{impconstm} establishes \eqref{impconst}.
The proof of Lemma~\ref{trivcor} is thus completed.
\end{proof}

\subsubsection{Properties of the constructed Brownian motions}

\begin{lem}
\label{eq4l}
Assume the setting in Section~\ref{setproc} and let $ n \in \N $.
Then it holds for every 
$ t \in \{ 0, \nicefrac{ 1 }{ n }, \nicefrac{ 2 }{ n }, \dots \} $ 
that
\begin{equation}
\label{claim10}
  W_t^{(n),\triangle}
  = W_t^{(n),\Box}
  .
\end{equation}
\end{lem}

\begin{proof}[Proof of Lemma~\ref{eq4l}]
First, observe that it holds for all $ t \in [0,\infty) $ that
\begin{equation}
  ( X_3^{ (n), [0] } )_t
  = 
  ( X_4^{ (n), [0] } )_t
  .
\end{equation}
Hence, we obtain for all $ k \in \N_0 $ that
\begin{equation}
\label{eqgleich1}
  ( X_3^{ (n), [0] } )_{ k / n }
  =
  ( X_4^{ (n), [0] } )_{ k / n }
\end{equation}
and
\begin{equation}
\label{eqgleich2}
  (X_3^{(n),[0]})_{X^{(n),[0]}_2}
  =
  (X_4^{(n),[0]})_{X^{(n),[0]}_2}.
\end{equation}
Next note that it holds for all
$ 
  r \in [0,\infty) 
$,
$
  y \in [ \mathcal{C}_0 ]^3 \times \mathcal{C}_{00}
$,
$
  t \in [0,r] \cup [r + \nicefrac{ 1 }{ n }, \infty ) 
$
that
\begin{equation}
\label{equkgk}
  \big( 
    F_n^{ \triangle }( r, y ) 
  \big)_t
  =
  \big( 
    F_n^\Box( r, y )
  \big)_t
  .
\end{equation}
Moreover, observe that it holds for all $ m \in \N $ that
\begin{equation}
\label{equkgk1}
\begin{split}
  X_2^{ (n), [m] } 
  -
  X_1^{ (n), [m] }
& =
  \mathcal{T}_n\big( X_2^{ (n), [m-1] }, Y^{ [m] } \big)
\\
& \geq
  \mathfrak{T}^2_n\big( X_2^{ (n), [m-1] } \big)
=
  \mathfrak{T}^1_n(X_1^{(n),[m]}) 
  + \nicefrac{ 1 }{ n } 
  .
\end{split}
\end{equation}
This and \eqref{equkgk} yield that for all 
$ m \in \N $, $ k \in \N_0 $ 
it holds that
\begin{equation}
\label{eqgleich3}
\begin{split}
  ( X_3^{(n),[m]} )_{\mathfrak{T}^1_n(X_1^{(n),[m]})+k/n}
&=\left(F^{ \triangle }_n(\mathfrak{T}^1_n(X_1^{(n),[m]}),Y^{[m]})\right)_{\mathfrak{T}^1_n(X_1^{(n),[m]})+k/n}\\
&=\left(F^\Box_n(\mathfrak{T}^1_n(X_1^{(n),[m]}),Y^{[m]})\right)_{\mathfrak{T}^1_n(X_1^{(n),[m]})+k/n}\\
&=(X_4^{(n),[m]})_{\mathfrak{T}^1_n(X_1^{(n),[m]})+k/n}
\end{split}
\end{equation}
and
\begin{equation}\label{eqgleich4}
\begin{split}
(X_3^{(n),[m]})_{X^{(n),[m]}_2-X_1^{(n),[m]}}
&=\left(F^{ \triangle }_n(\mathfrak{T}^1_n(X_1^{(n),[m]}),Y^{[m]})\right)_{X^{(n),[m]}_2-X_1^{(n),[m]}}\\
&=\left(F^\Box_n(\mathfrak{T}^1_n(X_1^{(n),[m]}),Y^{[m]})\right)_{X^{(n),[m]}_2-X_1^{(n),[m]}}\\
&=(X_4^{(n),[m]})_{X^{(n),[m]}_2-X_1^{(n),[m]}}.
\end{split}
\end{equation}
Combining
\eqref{eqgleich1},
\eqref{eqgleich2},
\eqref{eqgleich3},
and \eqref{eqgleich4}
proves \eqref{claim10}.
The proof of Lemma~\ref{eq4l} is thus completed.
\end{proof}

\subsubsection{Properties of the constructed squared Bessel processes}

\begin{lem}
\label{lembasic1}
Assume the setting in Section~\ref{setproc}
and let $n\in\N$, $\ast\in\{\triangle, \Box\}$.
Then 
\begin{enumerate}[(i)]
\item 
\label{item:lembasic1_1}
it holds that $ W^{ (n), \ast } $ is a Brownian motion,
\item 
\label{item:lembasic1_2}
it holds that
$ W^{ (n), \ast } $
and $ Z $
are independent, and 
\item
it holds that
\begin{equation}
\P\!
\left(
\forall\, t\in[0,\infty)\colon
Z_t^{(n),\ast}
=\mathcal{Z}_t^{Z,\delta,b,W^{(n),\ast}}
\right)
=1.
\end{equation}
\end{enumerate}
\end{lem}

\begin{proof}[Proof of Lemma~\ref{lembasic1}]
We present the proof of Lemma~\ref{lembasic1} in the case $ \ast = \triangle $. 
The case $ \ast = \Box $ is handled similarly.
Throughout this proof
for every $ m \in \N_0 $
let 
$
  W^{ (m) } \colon \Omega \to \mathcal C_0
$
be the Brownian motion given by
$
  W^{(m)}=X^{(n),[m]}_3
$,
for every $m\in\N_0$
let $\tau^{(m)}\colon \Omega\to [0,\infty)$
be
the random variable
given by $\tau^{(m)}=X^{(n),[m]}_2-X^{(n),[m]}_1$,
for every $m\in\N_0$ let $\mathbb{F}^{(m)}=(\mathbb{F}_t^{(m)})_{t\in [0,\infty)}$ be the normal filtration on $(\Omega,\mathfrak{F},\P)$ which satisfies for all $t\in [0,\infty)$, $m\in\N$ that
\begin{align}
\mathbb{F}_t^{(0)}
=
\sigma_\Omega
\Big(
\sigma_\Omega\big(
W^{(0)}_s\colon s\in [0,t]
\big)
\cup
\sigma_\Omega(
Z
)
\cup
\big\{A\in\mathfrak{F}\colon \P(A)=0\big\}
\Big)
\end{align}
and
\begin{multline}
\mathbb{F}_t^{(m)}
=
\sigma_\Omega
\Big(
\sigma_\Omega\big(
W^{(m)}_s\colon s\in [0,t]
\big)
\cup
\sigma_\Omega\big(
Z,\tilde{W},Y^{[m]}_4
\big)\\
\cup
\sigma_\Omega\big(
Y^{[k]}\colon k\in \N\cap[1,m-1]
\big)
\cup
\big\{A\in\mathfrak{F}\colon \P(A)=0\big\}
\Big),
\end{multline}
and for every $m\in\N$ let $\tilde{t}_1^{(m)}\colon\Omega\to [0,\nicefrac{1}{n})$
be the random variable given by
$\tilde{t}_1^{(m)}
=\mathfrak{T}^1_n(X^{(n),[m-1]}_2)$
and let
$\tilde{t}_2^{(m)}\colon\Omega\to [\nicefrac{1}{n},\nicefrac{2}{n})$
be the random variable given by
$\tilde{t}_2^{(m)}
=\mathfrak{T}^2_n(X^{(n),[m-1]}_2)$.
Note that for every $m\in\N_0$ it holds that $ W^{(m)} $ is a $\mathbb{F}^{(m)}$-Brownian motion.
Next note that
for every $m\in\N_0$ it holds that
\begin{align}\label{brauche1}
(\cup_{u\in [0,\infty)}\mathbb{F}^{(m)}_u)\subseteq \mathbb{F}^{(m+1)}_0.
\end{align}
Lemma~\ref{comptriv} implies that $\tau^{(0)}$ is a $\mathbb{F}^{(0)}$-stopping time.
Observe that
for every $m\in\N$
it holds that
$\tilde{t}_1^{(m)}$
is
$ \mathbb{F}^{(m)}_0 $/$ \mathcal{B}( [0,\nicefrac{1}{n}) ) $-measurable.
Moreover, note that
for every $m\in\N$
it holds that
$\tilde{t}_2^{(m)}$
is
$ \mathbb{F}^{(m)}_0 $/$ \mathcal{B}( [\nicefrac{1}{n},\nicefrac{2}{n}) ) $-measurable.
Item~\eqref{triv10item03} of Lemma~\ref{triv10} implies that for all $m\in\N$ it holds $\P$-a.s.\ that
\begin{equation}\label{yxcv}
\begin{split}
\tau^{(m)}
&=X^{(n),[m]}_2-X^{(n),[m]}_1
=
\mathcal{T}_n(X_2^{(n),[m-1]},Y^{[m]})
=
\mathcal{S}_n(X_2^{(n),[m-1]},Y^{[m]})
\\
&=
\max_{\ast\in \{\triangle,\Box\}}
\bigg[
\inf\!
\Big(
\Big\{
t\in[0,\infty)\colon 
t\geq \mathfrak{T}^2_n(X_2^{(n),[m-1]})
\text{ and }\\
&\qquad
\mathcal{Z}^{0,\delta,b,F^\ast_n(\mathfrak T^1_n(X_2^{(n),[m-1]}),Y^{[m]})}_t
=0
\Big\}
\cup\{\infty\}
\Big)
\bigg]
\\
&=
\max_{i\in \{5,6\}}\left[
\inf\!
\big(
\{t\in [0,\infty)\colon t\geq \tilde{t}_2^{(m)}\text{ and }X^{(n),[m]}_i=0\}
\cup
\{\infty\}
\big)
\right].
\end{split}
\end{equation}
Observe that
for every $m\in\N$, $t\in [0,\infty)$
it holds that
\begin{equation}
\begin{split}
&(X_4^{(n),[m]})_t
\\
&=
\begin{cases}
(X_3^{(n),[m]})_t &\colon 0\leq t\leq \tilde{t}_1^{(m)}
\\
n\left[(X_3^{(n),[m]})_{\tilde{t}_2^{(m)}}-(X_3^{(n),[m]})_{\tilde{t}_1^{(m)}}\right]
\cdot
[t-\tilde{t}_1^{(m)}]
\\
+\frac{1}{\sqrt{n}}(Y_4^{[m]})_{n(t-\tilde{t}_1^{(m)})}
+
(X_3^{(n),[m]})_{\tilde{t}_1^{(m)}}
&\colon \tilde{t}_1^{(m)}\leq t\leq \tilde{t}_2^{(m)}
\\
(X_3^{(n),[m]})_t &\colon \tilde{t}_2^{(m)}\leq t<\infty
\end{cases}.
\end{split}
\end{equation}
Hence, we obtain for every $m\in\N$, $t\in[0,\infty)$, $s\in [0,t]$ that
\begin{align}
\begin{split}
&(X_4^{(n),[m]})_s
\cdot
\mathbbm{1}^\Omega_{\{t\geq \tilde{t}_2^{(m)}\}}
\\
&=
W_s^{(m)}
\cdot
\mathbbm{1}^\Omega_{\{s\leq \tilde{t}_1^{(m)}\text{ or }s\geq \tilde{t}_2^{(m)}\}}
\cdot
\mathbbm{1}^\Omega_{\{t\geq \tilde{t}_2^{(m)}\}}
\\
&\quad+
\left(n(W^{(m)}_{t\wedge \tilde{t}_2^{(m)}}-W^{(m)}_{t\wedge \tilde{t}_1^{(m)}})
(s-\tilde{t}_1^{(m)})
+\frac{1}{\sqrt{n}}(Y_4^{[m]})_{1\wedge (n(s-\tilde{t}_1^{(m)}))\vee 0}
+W^{(m)}_{t\wedge\tilde{t}_1^{(m)}}
\right)
\\
&\qquad\cdot
\mathbbm{1}^\Omega_{\{\tilde{t}_1^{(m)}<s<\tilde{t}_2^{(m)}\}}
\cdot
\mathbbm{1}^\Omega_{\{t\geq \tilde{t}_2^{(m)}\}}.
\end{split}
\end{align}
This demonstrates for every $m\in\N$, $t\in[0,\infty)$, $s\in [0,t]$
that the function
$\Omega\ni \omega\mapsto (X_4^{(n),[m]})_{s}(\omega)\cdot \mathbbm{1}^\Omega_{\{t\geq \tilde{t}_2^{(m)}\}}(\omega)\in\R$
is
$ \mathbb{F}^{(m)}_t $/$ \mathcal{B}( \R ) $-measurable.
Therefore, we obtain that for every $i\in\{5,6\}$, $m\in\N$ it holds that
$((X^{(n),[m]}_i)_t\cdot \mathbbm{1}_{\{t\geq \tilde{t}_2^{(m)}\}})_{t\in [0,\infty)}$ is $\mathbb{F}^{(m)}$-adapted.
This implies that for every $i\in\{5,6\}$, $m\in\N$, $t\in [0,\infty)$ it holds that
\begin{multline}
\sup\{(X^{(n),[m]}_i)_s\colon s\in [0,t]\text{ and }s\geq \tilde{t}_2^{(m)}\}\\
=
\sup\{(X^{(n),[m]}_i)_s
\cdot
\mathbbm{1}^\Omega_{\{s\geq \tilde{t}_2^{(m)}\}}
\colon s\in [0,t]\text{ and }s\geq \tilde{t}_2^{(m)}\}
\end{multline}
and
\begin{multline}
\inf\{(X^{(n),[m]}_i)_s\colon s\in [0,t]\text{ and }s\geq \tilde{t}_2^{(m)}\}\\
=
\inf\{(X^{(n),[m]}_i)_s
\cdot
\mathbbm{1}^\Omega_{\{s\geq \tilde{t}_2^{(m)}\}}
\colon s\in [0,t]\text{ and }s\geq \tilde{t}_2^{(m)}\}
\end{multline}
are
$ \mathbb{F}^{(m)}_t $/$ \mathcal{B}( [-\infty,\infty] )$-measurable.
Hence, we get that for all $i\in\{5,6\}$, $m\in\N$ it holds that $\inf\{t\in [0,\infty)\colon t\geq \tilde{t}_2^{(m)}\text{ and }X^{(n),[m]}_i=0\}$ is a $\mathbb{F}^{(m)}$-stopping time.
This implies that for all $m\in\N$ it holds that
\begin{align}
\max_{i\in \{5,6\}}\big(\inf\{t\in [0,\infty)\colon t\geq \tilde{t}_2^{(m)}\text{ and }X^{(n),[m]}_i=0\}\big)
\end{align}
is a $\mathbb{F}^{(m)}$-stopping time.
Combining this and \eqref{yxcv} assures that for all $m\in\N$ it holds that $\tau^{(m)}$ is a $\mathbb{F}^{(m)}$-stopping time.
Next observe that it holds that
\begin{align}\label{brauche2}
\sum_{m\in\N_0}\tau^{(m)}
=\infty
\end{align}
and $Z$ is
$ \mathbb{F}_0^{(0)} $/$ \mathcal{B}( [0,\infty) )$-measurable.
Item~\eqref{mnbbv} of Lemma~\ref{trivcor} implies that for all $ m \in \N_0 $ 
it holds that 
\begin{equation}\label{brauche3}
  \P\!\left(
    ( X^{ (n), [m] }_5 )_{ \tau^{ (m) } } = 0 
  \right) = 1
  .
\end{equation}
Combining 
\eqref{brauche1},
the fact that
for every $ m \in \N_0 $ it holds that 
$ W^{(m)} $ is a $ \mathbb{F}^{(m)} $-Brownian motion,
the fact that
for every $ m \in \N_0 $ it holds that 
$
  \tau^{ (m) }
$ is a $ \mathbb{F}^{(m)} $-stopping time,
\eqref{brauche2},
the fact that
$ Z $ is
$ \mathbb{F}_0^{(0)} $/$ \mathcal{B}( [0,\infty) ) $-measurable,
\eqref{brauche3}, 
and
Lemma~\ref{lem1}
completes the proof of Lemma~\ref{lembasic1}.
\end{proof}

\subsubsection{On conditional distributions of the considered random objects}

\begin{lem}
\label{lembasic2}
Assume the setting in Section~\ref{setproc},
let $ n \in \N $,
and for every 
$ r \in [0,\infty) $ 
let 
$ 
  \P_r \colon 
  \mathcal{B}( 
    [0,\infty)^2 \times 
    [ \mathcal{C}_0 ]^2 \times 
    [ C( [0,\infty), \R ) ]^2
  )
  \to [0,1]
$ 
be the probability measure 
which satisfies for all 
$
  B \in \mathcal{B}(
    [0,\infty)^2 \times [ \mathcal{C}_0 ]^2 \times [ C( [0,\infty), \R) ]^2 
  )
$ 
that
\begin{equation}
  \P_r( B )
  =
  \P\Big(
    \big\{
      \Phi_n\big( 
        r, Y^{ [0] } 
      \big)
      \in B
    \big\}
    \, \big| \,
    \big\{ 
      \Phi_{ n, 2 }\big( 
        r, Y^{ [0] } 
      \big) > 1
    \big\}
  \Big)
  .
\end{equation}
Then it holds 
for all 
$
  B \in 
  \mathcal{B}( 
    [0,\infty)^2 \times [ \mathcal{C}_0 ]^2 \times [ C( [0,\infty), \R ) ]^2 
  ) 
$
that
\begin{equation}
  \P\!\left(
    \mathbbm{1}^{ \Omega }_{
      \{ 
        0 \leq \gamma_n \leq 1
      \}
    }
    \P\big(
      X^{ (n),[\mathcal{M}_n] } \in B 
      \, | \, 
      \sigma_\Omega(\gamma_n) 
    \big)
    =
    \mathbbm{1}^{ \Omega }_{
      \{ 
        0 \leq \gamma_n \leq 1
      \}
    }
    \P_{ \min\{\gamma_n,1\} }( B )
  \right) = 1
  .
\end{equation}
\end{lem}

\begin{proof}[Proof of Lemma~\ref{lembasic2}]
Throughout this proof 
let 
\begin{equation}
  \mathbb{B} = [0,\infty)^2 \times [ \mathcal{C}_0 ]^2 \times [ C([0,\infty), \R) ]^2
  ,
\end{equation}
let 
$
  \P_{ \infty } \colon 
  \mathcal{B}( 
    \mathbb{B}  
  ) 
  \to [0,1]
$
be the function which satisfies 
for all 
$
  B \in \mathcal{B}(
    \mathbb{B}
  )
$
that
$
  \P_{ \infty }(B)
  = 0 
$,
let 
$
  \Q_r \colon \mathcal{B}(
    \mathbb{B}
  ) \to [0,1]
$,
$ r \in [0,\infty] $,
be the functions which satisfy 
for all 
$ 
  r \in [0,\infty) 
$, 
$ 
  B \in \mathcal{B}( \mathbb{B} )
$
that
\begin{equation}
  \Q_r( B )
  =
  \P\!\left(
    \left\{ 
      \Phi_n\big( 
        r, Y^{ [0] } 
      \big)
      \in B
    \right\}
    \cap
    \left\{ 
      \Phi_{ n, 2 }\big( 
        r, Y^{ [0] } 
      \big) > 1
    \right\}
  \right)
\end{equation}
and
$
  \Q_{ \infty }( B )
  = 0 
$,
and let $ A \in \mathcal{B}( [0,1] ) $, 
$
  B \in 
  \mathcal{B}( \mathbb{B} )
$.
To establish Lemma~\ref{lembasic2},
we need to prove that 
\begin{equation}
\label{102143}
  \P\Big(
    \big\{ 
      X^{ (n),[\mathcal{M}_n] } \in B
    \big\}
    \cap
    \big\{
      \gamma_n \in A
    \big\}
  \Big)
  =
  \E\!\left[
    \,
    \P_{ \gamma_n }( B )
    \cdot
    \mathbbm{1}^{ [0,\infty] }_{A}( \gamma_n )
    \,
  \right]
  .
\end{equation}
For this we observe that 
\begin{equation}
\label{102141}
\begin{split}
& 
  \P\Big(
    \big\{ 
      X^{ (n),[\mathcal{M}_n] } \in B
    \big\}
    \cap
    \big\{
      \gamma_n \in A
    \big\}
  \Big)
  =
  \E\!\left[
    \mathbbm{1}^{ \mathbb{B} }_{B}( X^{ (n),[\mathcal{M}_n] } )
    \cdot
    \mathbbm{1}^{ [0,\infty] }_{A}(\gamma_n)
  \right]
\\
&
  =
  \sum_{ m = 0 }^{ \infty }
  \E\!\left[
    \mathbbm{1}^{ \mathbb{B} }_{B}( X^{ (n),[\mathcal{M}_n] } )
    \cdot
    \mathbbm{1}^{ [0,\infty] }_{A}( \gamma_n )
    \cdot
    \mathbbm{1}^{ \R }_{ \{ m \} }( \mathcal{M}_n )
  \right]
\\
&
  =
  \sum_{ m = 1 }^{ \infty }
  \E\!\left[
    \mathbbm{1}^{ \mathbb{B} }_{ B }( X^{ (n),[\mathcal{M}_n] } )
    \cdot
    \mathbbm{1}^{ [0,\infty] }_{ A }( \gamma_n )
    \cdot
    \mathbbm{1}^{ \R }_{ \{ m \} }( \mathcal{M}_n )
  \right]
\\
&
  =
  \sum_{ m = 1 }^{ \infty }
  \E\!\left[
    \mathbbm{1}^{ \mathbb{B} }_{B}( X^{ (n), [m] } )
    \cdot
    \mathbbm{1}^{ \R }_{A}( X^{ (n), [m] }_1 )
    \cdot
    \mathbbm{1}^{ \R }_{ [0,1] }( X_1^{ (n), [m] } )
    \cdot
    \mathbbm{1}^{ \R }_{ (1, \infty) }( X_1^{ (n), [m+1] } )
  \right]
  .
\end{split}
\end{equation}
Next we recall that for all $ m \in \N $ it holds 
\begin{enumerate}[(a)]
\item 
\label{itema}
that
\begin{equation}
  X^{ (n), [m] }
  = \Phi_n\big( 
    X_2^{ (n), [m-1] }, Y^{ [m] } 
  \big)
  ,
\end{equation}
\item
\label{itemb}
that
\begin{equation}
  Y^{ [ m ] }( \P )_{
    \mathcal{B}( 
      [ \mathcal{C}_0 ]^3 \times \mathcal{C}_{ 00 }
    )
  } 
  =  
  Y^{ [ 0 ] }( \P )_{
    \mathcal{B}( 
      [ \mathcal{C}_0 ]^3 \times \mathcal{C}_{ 00 }
    )
  } 
  ,
\end{equation}
and
\item
\label{itemc}
that
$ X_2^{ (n), [m-1] } $ and $ Y^{ [m] } $
are independent.
\end{enumerate}
Item~\eqref{itema}
ensures that for all $ m \in \N $ it holds that
\begin{equation}
\begin{split}
&
  \E\!\left[
    \mathbbm{1}^{ \mathbb{B} }_{B}( X^{ (n), [m] } )
    \cdot
    \mathbbm{1}^{ \R }_{ A }( X^{ (n), [m] }_1 )
    \cdot
    \mathbbm{1}^{ \R }_{ [0,1] }( X_1^{ (n), [m] } )
    \cdot
    \mathbbm{1}^{ \R }_{ ( 1, \infty) }( X_1^{ (n), [m+1] } )
  \right]
\\
&
  = 
  \E\!\left[
    \mathbbm{1}^{ \mathbb{B} }_B( X^{ (n), [m] } )
    \cdot
    \mathbbm{1}^{ \R }_A( 
      X^{ (n), [m-1] }_2 
    )
    \cdot
    \mathbbm{1}^{ \R }_{ [0,1] }( X_2^{ (n), [m-1] } )
    \cdot
    \mathbbm{1}^{ \R }_{ (1, \infty) }( X_2^{ (n), [m] } )
  \right]
\\
&
  = 
  \E\Big[
    \mathbbm{1}^{ \R }_A\big( 
      X^{ (n), [m-1] }_2 
    \big)
    \cdot
    \mathbbm{1}^{ \R }_{ [0,1] }( X_2^{ (n), [m-1] } )
\\
&
\quad
\cdot
    \mathbbm{1}^{ \mathbb{B} }_B\big( 
      \Phi_n( 
        X_2^{ (n), [m - 1] }, Y^{ [m] } 
      ) 
    \big)
    \cdot
    \mathbbm{1}^{ \R }_{ (1, \infty) }\big(
      \Phi_{ n, 2 }(
        X_2^{ (n), [m-1] }, Y^{ [m] } 
      )
    \big)
  \Big]
  .
\end{split}
\end{equation}
Items~\eqref{itemb} and \eqref{itemc} hence show that
for all $ m \in \N $ it holds that
\begin{equation}
\begin{split}
&
  \E\!\left[
    \mathbbm{1}^{ \mathbb{B} }_{B}( X^{ (n), [m] } )
    \cdot
    \mathbbm{1}^{ \R }_{ A }( X^{ (n), [m] }_1 )
    \cdot
    \mathbbm{1}^{ \R }_{ [0,1] }( X_1^{ (n), [m] } )
    \cdot
    \mathbbm{1}^{ \R }_{ ( 1, \infty) }( X_1^{ (n), [m+1] } )
  \right]
\\
&
  = 
  \E\!\left[
    \mathbbm{1}^{ \R }_A( X^{ (n), [m-1] }_2 )
    \cdot
    \mathbbm{1}^{ \R }_{ [0,1] }( X_2^{ (n), [m-1] } )
    \cdot
    \Q_{ X_2^{ (n), [m-1] } }( B )
  \right]
\\
&
  =
  \E\Big[
    \mathbbm{1}^{ \R }_A( X^{ (n), [m-1] }_2 )
    \cdot
    \mathbbm{1}^{ \R }_{ [0,1] }( X_2^{ (n), [m-1] } )
    \cdot
    \P_{ X_2^{(n), [m-1] } }(B)
\\
&
    \quad
    \cdot
    \mathbbm{1}^{ \R }_{ (1,\infty) }\big(
      \Phi_{ n, 2 }( X_2^{ (n), [m-1] }, Y^{ [m] } )
    \big)
  \Big]
\\
&
=
  \E\!\left[
    \mathbbm{1}^{ \R }_A( X^{ (n), [m-1] }_2 )
    \cdot
    \P_{ X_2^{ (n), [m-1] } }( B )
    \cdot
    \mathbbm{1}^{ \R }_{ [0,1] }( X_2^{ (n), [m-1] } )
    \cdot
    \mathbbm{1}^{ \R }_{ (1, \infty) }( X_2^{ (n), [m] } )
  \right]
  .
\end{split}
\end{equation}
Item~\eqref{trivcorit1} in 
Lemma~\ref{trivcor} therefore proves
that for all $ m \in \N $ it holds that
\begin{equation}
\label{102142}
\begin{split}
&
  \E\!\left[
    \mathbbm{1}^{ \mathbb{B} }_B( X^{ (n), [m] } )
    \cdot
    \mathbbm{1}^{ \R }_A( X^{ (n), [m] }_1 )
    \cdot
    \mathbbm{1}^{ \R }_{ [0,1] }( X_1^{ (n), [m] } )
    \cdot
    \mathbbm{1}^{ \R }_{ (1, \infty) }( X_1^{ (n), [m+1] } )
  \right]
\\
&
=
  \E\!\left[
    \mathbbm{1}^{ \R }_A( X^{ (n), [m-1] }_2 )
    \cdot
    \P_{ X_2^{ (n), [m - 1] } }( B )
    \cdot
    \mathbbm{1}^{ \R }_{ \{ m \} }( \mathcal{M}_n )
  \right]
\\
&
=
  \E\!\left[
    \mathbbm{1}^{ \R }_A( X^{ (n), [m] }_1 )
    \cdot
    \P_{ X_1^{ (n), [m] } }( B )
    \cdot
    \mathbbm{1}^{ \R }_{ \{ m \} }( \mathcal{M}_n )
  \right]
\\
&
=
  \E\!\left[
    \mathbbm{1}^{ [0,\infty] }_A( \gamma_n )
    \cdot
    \P_{ \gamma_n }( B )
    \cdot
    \mathbbm{1}^{ \R }_{ \{ m \} }( \mathcal{M}_n )
  \right]
  .
\end{split}
\end{equation}
Combining \eqref{102141} with \eqref{102142} yields that
\begin{equation}
\begin{split}
& 
  \E\!\left[
    \mathbbm{1}^{ \mathbb{B} }_B( X^{ (n),[\mathcal{M}_n] } )
    \cdot
    \mathbbm{1}^{ [0,\infty] }_{A}( \gamma_n )
  \right]
  =
  \sum_{ m = 1 }^{ \infty }
  \E\!\left[
    \mathbbm{1}^{ [0,\infty] }_A( \gamma_n )
    \cdot
    \P_{ \gamma_n }( B )
    \cdot
    \mathbbm{1}^{ \R }_{ \{ m \} }( \mathcal{M}_n )
  \right]
\\
&
=
  \E\!\left[
    \mathbbm{1}^{ [0,\infty] }_A( \gamma_n )
    \cdot
    \P_{ \gamma_n }( B )
  \right]
  =
  \E\!\left[
    \P_{ \gamma_n }( B )
    \cdot
    \mathbbm{1}^{ [0,\infty] }_A( \gamma_n )
  \right]
  .
\end{split}
\end{equation}
This establishes \eqref{102143}.
The proof of Lemma~\ref{lembasic2} is thus completed.
\end{proof}

\subsection{Lower bounds for strong
\texorpdfstring{$ L^1 $}{L1}-distances
between the constructed squared Bessel processes}

\subsubsection{A first very rough lower bound for strong
\texorpdfstring{$ L^1 $}{L1}-distances
between the constructed squared Bessel processes}

\begin{lem}
\label{lemdista}
Assume the setting in Section~\ref{setproc},
let $ z \in [0,\infty) $,
and
let $ \tilde{W}^{ \Box } \colon \Omega\to \mathcal{C}_0 $ 
be the Brownian motion given by
$
  \tilde{W}^{ \Box } = G^{ \Box }_1( \tilde{W}^{ \triangle }, B ) 
$.
Then the following two statements are equivalent:
\begin{enumerate}[(i)]
\item\label{lem23item1}
It holds that
\begin{equation}\label{eqshow}
  \E\!\left[
    \big|
      \mathcal{Z}_1^{ z, \delta, b, \tilde{W}^{ \triangle } }
      -
      \mathcal{Z}_1^{ z, \delta, b, \tilde{W}^\Box }
    \big|
  \right]
  = 0.
\end{equation}
\item\label{lem23item2}
There exists a
$ \mathcal{B}( \R ) $/$ \mathcal{B}( \R ) $-measurable function $ f \colon \R \to \R $ which satisfies
\begin{equation}
  \P\!\left(
    \mathcal{Z}_1^{ z, \delta, b, \tilde{W}^{ \triangle } }
    =
    f( \tilde{W}^{ \triangle }_1 )
  \right)
  = 1
  .
\end{equation}
\end{enumerate}
\end{lem}

\begin{proof}[Proof of Lemma~\ref{lemdista}]
First, we prove that $\big(\eqref{lem23item2}\implies\eqref{lem23item1}\big)$.
Item~\eqref{lem23item2}
ensures
\begin{equation}
  \P\!\left(
    \mathcal{Z}_1^{ z, \delta, b, \tilde{W}^{ \Box } }
    =
    f( \tilde{W}^{ \Box }_1 )
  \right)
  = 1
  .
\end{equation}
Combining item~\eqref{lem23item2} and the fact that $\tilde{W}^{ \triangle }_1=\tilde{W}^\Box_1$ hence ensures that
\begin{equation}
\begin{split}
\E\!
\left[
|
\mathcal{Z}_1^{z,\delta,b,\tilde{W}^{ \triangle }}
-\mathcal{Z}_1^{z,\delta,b,\tilde{W}^\Box}
|
\right]
&=\E\!
\left[
|
f(\tilde{W}^{ \triangle }_1)
-f(\tilde{W}^\Box_1)
|
\right]\\
&=\E\!
\left[
|
f(\tilde{W}^{ \triangle }_1)
-f(\tilde{W}^{ \triangle }_1)
|
\right]
=0.
\end{split}
\end{equation}
This establishes that
$\big(\eqref{lem23item2}\implies\eqref{lem23item1}\big)$.
Next we prove that $\big(\eqref{lem23item1}\implies\eqref{lem23item2}\big)$.
Combining the fact that $\tilde{W}^{ \triangle }$ and $B$ are independent and
item~\eqref{lem23item1}
assures that it holds $\P$-a.s.\ that
\begin{equation}
\begin{split}
\E \! \left[
\mathcal{Z}_1^{z,\delta,b,\tilde{W}^{ \triangle }}
\ \big|\
\sigma_\Omega(\tilde{W}^{ \triangle }_1)
\right]
&=\E \! \left[
\mathcal{Z}_1^{z,\delta,b,\tilde{W}^{ \triangle }}
\ \big|\
\sigma_\Omega(\tilde{W}^{ \triangle }_1, B)
\right]\\
&=\E \! \left[
\mathcal{Z}_1^{z,\delta,b,\tilde{W}^\Box}
\ \big|\
\sigma_\Omega(\tilde{W}^{ \triangle }_1, B)
\right]\\
&=\mathcal{Z}_1^{z,\delta,b,\tilde{W}^\Box}\\
&=\mathcal{Z}_1^{z,\delta,b,\tilde{W}^{ \triangle }}.
\end{split}
\end{equation}
This together with the factorization lemma for conditional expectations establishes item~\eqref{lem23item2}.
This demonstrates that
$\big(\eqref{lem23item1}\implies\eqref{lem23item2}\big)$.
The proof of Lemma~\ref{lemdista} is thus completed.
\end{proof}

\begin{lem}
\label{lemdistb}
Assume the setting in Section~\ref{setproc},
let $ z \in [0,\infty) $,
and
let $ \tilde{W}^{ \Box } \colon \Omega\to \mathcal{C}_0 $ 
be the Brownian motion given by
$
  \tilde{W}^{ \Box } = G^{ \Box }_1( \tilde{W}^{ \triangle }, B ) 
$.
Then 
\begin{equation}
\label{eq1}
  \E\!\left[
    \big|
      \mathcal{Z}_1^{ z, \delta, b, \tilde{W}^{ \triangle } }
      -
      \mathcal{Z}_1^{ z, \delta, b, \tilde{W}^\Box }
    \big|
  \right]
  > 0.
\end{equation}
\end{lem}

\begin{proof}[Proof of Lemma~\ref{lemdistb}]
In the case
$ (\delta,b) = (1,0) $
inequality
\eqref{eq1}
follows
from
Lemma~\ref{lemdista}
and, e.g.,
Hefter \& Herzwurm~\cite[Equation~(13)]{2016arXiv160101455H}
and
in the case $(\delta,b)\in \big((0,2)\times [0,\infty)\big)\setminus \{(1,0)\}$
inequality
\eqref{eq1}
follows
from
Lemma~\ref{lemdista}
and, e.g.,
Hefter, Herzwurm, \& M\"{u}ller-Gronbach~\cite[Theorem~6]{2017arXiv171008707H}.
The proof of Lemma~\ref{lemdistb} is thus completed.
\end{proof}

\begin{lem}
\label{lemdist}
Assume the setting in Section~\ref{setproc}
and
for every 
$ r \in [0,1] $,
$ \ast \in \{ \triangle, \Box \} $
let
$
  \mathcal{W}^{ r, \ast }
  \colon \Omega \to \mathcal{C}_0
$
be the Brownian motion given by
$
  \mathcal{W}^{ r, \ast } = F^{ \ast }_1( r, Y^{[0]} )
$.
Then
\begin{equation}
\label{eq111}
  \inf_{
    r \in [0,1]
  }
  \inf_{
    \beta \in [0,b]
  }
  \E\!\left[
    \big|
      \mathcal{Z}_{ r + 1 }^{ 0, \delta, \beta, \mathcal{W}^{ r, \triangle } }
      -
      \mathcal{Z}_{ r + 1 }^{ 0, \delta, \beta, \mathcal{W}^{ r, \Box } }
    \big|
  \right]
  > 0 .
\end{equation}
\end{lem}

\begin{proof}[Proof of Lemma~\ref{lemdist}]
Throughout this proof
let $U^\triangle,U^\Box, V\colon [0,1]\times[0,\infty)\times\Omega\to\R$
be the random fields
which satisfy for all
$\ast\in\{\triangle, \Box\}$,
$r\in [0,1]$,
$\beta \in [0,\infty)$
that
\begin{equation}
U^\ast(r,\beta)
=\mathcal{Z}_{r+1}^{0,\delta,\beta,\mathcal{W}^{ r, \ast }}
\end{equation}
and
\begin{equation}
V(r, \beta)
=\mathcal{Z}_{r}^{0,\delta,\beta,\tilde{W}^{(1)}},
\end{equation}
let $g\colon [0,1]\times[0,\infty)\to\R$ be the function which satisfies for all $r\in [0,1]$, $\beta \in [0,\infty)$ that
\begin{equation}
g(r,\beta)
=\E \! \left[\big|U^{ \triangle }(r,\beta)
-U^\Box(r,\beta)\big| \right],
\end{equation}
and let
$ \tilde{W}^{ \Box } \colon \Omega\to \mathcal{C}_0 $ 
be the Brownian motion given by
$
  \tilde{W}^{ \Box } = G^{ \Box }_1( \tilde{W}^{ \triangle }, B ) 
$.
Observe that for every
$\ast\in\{\triangle, \Box\}$,
$r,t\in [0,1]$
it holds that
\begin{equation}
\left(\mathcal{W}^{r,\ast}\right)_{t+r}-\left(\mathcal{W}^{r,\ast}\right)_{r}
=\tilde{W}^\ast_t.
\end{equation}
Moreover, note that for every
$\ast\in\{\triangle, \Box\}$,
$r\in [0,1]$,
$t\in [0,r]$
it holds that
\begin{equation}
\left(\mathcal{W}^{r,\ast}\right)_{t}
=\tilde{W}^{(1)}_t.
\end{equation}
Hence, we obtain that for every
$\ast\in\{\triangle, \Box\}$,
$r\in [0,1]$,
$\beta \in [0,\infty)$
it holds $\P$-a.s.\ that
\begin{equation}\label{eq3456}
U^\ast(r,\beta)
=\mathcal{Z}_{r+1}^{0,\delta,\beta,\mathcal{W}^{r,\ast}}
=\mathcal{Z}_1^{\mathcal{Z}_r^{0,\delta,\beta,\mathcal{W}^{r,\ast}},\delta,\beta,\tilde{W}^\ast}
=\mathcal{Z}_1^{V(r,\beta),\delta,\beta,\tilde{W}^\ast}.
\end{equation}
Combining
Lemma~\ref{lemdistb},
\eqref{eq3456},
and
the fact that
for every $r\in [0,1]$, $\beta\in [0,\infty)$ it holds that $V(r,\beta)$ and $(\tilde{W}^{ \triangle },\tilde{W}^\Box)$
are independent yields that for every $r\in [0,1]$, $\beta\in [0,\infty)$ it holds that
\begin{equation}\label{pos}
\begin{split}
g(r,\beta)
&=\E \! \left[\big|U^{ \triangle }(r,\beta)
-U^\Box(r,\beta)\big| \right]
\\
&=
\E\!
\left[
\big|
\mathcal{Z}_1^{V(r,\beta),\delta,\beta,\tilde{W}^{ \triangle }}
-\mathcal{Z}_1^{V(r,\beta),\delta,\beta,\tilde{W}^\Box}
\big|
\right]>0.
\end{split}
\end{equation}
In the next step we combine
Lemma~\ref{lipsc},
\eqref{eq3456},
and the fact that for every
$\ast\in\{\triangle, \Box\}$,
$r,t\in [0,1]$,
$\beta\in [0,\infty)$
it holds that
$\tilde{W}^\ast$ and $(V(r,\beta), V(t,\beta))$ are independent
to obtain that for every
$\ast\in\{\triangle, \Box\}$,
$r,t\in [0,1]$,
$\beta\in [0,\infty)$
it holds that
\begin{equation}\label{a}
\begin{split}
\E \! \left[ \big| U^\ast(r,\beta)
-U^\ast(t,\beta) \big|\right]
&=\E \! \left[ \big|
\mathcal{Z}_1^{V(r,\beta),\delta,\beta,\tilde{W}^\ast}
-
\mathcal{Z}_1^{V(t,\beta),\delta,\beta,\tilde{W}^\ast}
\big|\right]
\\
&=
e^{-\beta}
\cdot
\E\!
\left[
\left|
V(r,\beta)
-V(t,\beta)
\right|
\right]
\\
&=
e^{-\beta}\cdot \E \! \left[ \big|
\mathcal{Z}_{r}^{0,\delta,\beta, \tilde{W}^{(1)}}
-\mathcal{Z}_{t}^{0,\delta,\beta, \tilde{W}^{(1)}}
\big| \right].
\end{split}
\end{equation}
In addition, we note that Lemma~\ref{lemcomparison} ensures that for all
$\ast\in\{\triangle, \Box\}$,
$t\in [0,1]$,
$\beta_1,\beta_2\in [0,\infty)$
it holds that
\begin{equation}\label{b}
\E \big[ | U^\ast(t,\beta_1)
-U^\ast(t,\beta_2) | \big]
=\big|
\E 
\big[U^\ast(t,\beta_1) \big]
-
\E  \big[U^\ast(t,\beta_2) \big]
\big|.
\end{equation}
The
triangle inequality
and
\eqref{a}
thereby
imply that for all
$\ast\in\{\triangle, \Box\}$,
$r,t\in [0,1]$,
$\beta_1,\beta_2\in [0,\infty)$
it holds that
\begin{equation}\label{a1}
\begin{split}
\E& \big[ |U^\ast(r,\beta_1)
-U^\ast(t,\beta_2)|\big]\\
&\leq
\E \big[ | U^\ast(r,\beta_1)
-U^\ast(t,\beta_1) |\big]
+
\E \big[ | U^\ast(t,\beta_1)
-U^\ast(t,\beta_2) | \big]\\
&=e^{-\beta_1}\cdot \E \! \left[ \big|
\mathcal{Z}_{r}^{0,\delta,\beta_1, \tilde{W}^{(1)}}
-\mathcal{Z}_{t}^{0,\delta,\beta_1, \tilde{W}^{(1)}}
\big| \right]
+\left|
\E  \big[
U^\ast(t,\beta_1) \big]
-
\E  \big[U^\ast(t,\beta_2) \big]
\right|\\
&\leq
\E \! \left[\big|
\mathcal{Z}_{r}^{0,\delta,\beta_1, \tilde{W}^{(1)}}
-\mathcal{Z}_{t}^{0,\delta,\beta_1, \tilde{W}^{(1)}}
\big| \right]
+\left|
\E 
\big[U^\ast(t,\beta_1) \big]
-
\E  \big[U^\ast(t,\beta_2) \big]
\right|.
\end{split}
\end{equation}
Next observe that for all $r\in [0,1]$, $\beta\in [0,\infty)$ it holds that
\begin{equation}\label{a2}
\limsup_{\substack{t\to r,\\t\in [0,1]}}
\E\!
\left[
\big|
\mathcal{Z}_{t}^{0,\delta,\beta,\tilde{W}^{(1)}}
-\mathcal{Z}_{r}^{0,\delta,\beta,\tilde{W}^{(1)}}
\big|
\right]
=0
\end{equation}
(cf., e.g.,
Mao~\cite[Theorem~2.4.3]{MR1475218}).
Moreover, we note that Lemma~\ref{expcomp} ensures that for all
$\ast\in\{\triangle, \Box\}$,
$r\in [0,1]$,
$\beta_1\in [0,\infty)$
it holds that
\begin{equation}\label{a3}
\limsup_{\substack{(t,\beta_2)\to (r,\beta_1),\\(t,\beta_2)\in [0,1]\times[0,\infty)}}
\big|
\E [U^\ast(t,\beta_2)]
-\E [U^\ast(r,\beta_1)]
\big|
=0.
\end{equation}
Combining
\eqref{a1},
\eqref{a2},
and \eqref{a3}
yields
that for all
$\ast\in\{\triangle, \Box\}$,
$r\in [0,1]$,
$\beta_1\in [0,\infty)$
it holds that
\begin{equation}
\limsup_{\substack{(t,\beta_2)\to (r,\beta_1),\\(t,\beta_2)\in [0,1]\times[0,\infty)}}
\E\!
\left[
\left|
U^\ast(t,\beta_2)
-
U^\ast(r,\beta_1)
\right|
\right]
=0.
\end{equation}
This proves that $ g $ is continuous.
Combining this and \eqref{pos} establishes \eqref{eq111}.
The proof of Lemma~\ref{lemdist} is thus completed.
\end{proof}

\subsubsection{On conditional \texorpdfstring{$ L^1 $}{L1}-distances 
between the constructed squared Bessel processes}

\begin{lem}
\label{lem1100}
Assume the setting in Section~\ref{setproc},
let 
$ n \in \N \cap [5,\infty) $,
$ t_0 \in [0, \nicefrac{ 1 }{ 2 } ] $,
$ t_1 = \mathfrak{T}^1_n(t_0) \in [ 0, \nicefrac{ 1 }{ n } ) $,
$ t_2 = \mathfrak{T}^2_n(t_0) \in [ \nicefrac{ 1 }{ n } , \nicefrac{ 2 }{ n } ) $,
$ t_3 = 1 - t_0 \in [ \nicefrac{ 1 }{ 2 } , 1 ] $,
and for every 
$ \ast \in \{ \triangle, \Box \} $
let
$ 
  W^{ \ast }
  \colon \Omega \to \mathcal{C}_0
$
be the Brownian motion given by 
$
  W^{ \ast } = F^{ \ast }_n( t_1, Y^{ [0] } )
$.
Then 
\begin{enumerate}[(i)]
\item 
it holds that
$
  t_1 < t_2 < t_3
$,
\item 
it holds that
$
  \P\big(
    \inf\nolimits_{ 
      s \in [ t_2, t_3 ]
    }
    \max_{\ast \in \{ \triangle, \Box \}}
      \mathcal{Z}_s^{ 0, \delta, b, W^{ \ast } } 
     > 0
  \big) > 0
$,
and
\item 
it holds that
\begin{multline}
\label{claim1054}
  \E\bigg[
    \big|
      \mathcal{Z}_{ t_3 }^{ 0, \delta, b, W^{ \triangle } }
      -
      \mathcal{Z}_{ t_3 }^{ 0, \delta, b, W^{ \Box } }
    \big|
    \,
    \Big|
    \Big\{
      \inf\limits_{ 
        s \in [ t_2, t_3 ]
      }
      \max_{\ast \in \{ \triangle, \Box \}}        \mathcal{Z}_s^{ 0, \delta, b, W^{ \ast } } > 0
    \Big\}
  \bigg]
\\
  \geq
  \frac{
      \E\Big[
        \big|
          \mathcal{Z}_{ 
            t_1 n + 1 
          }^{ 
            0, \delta, b / n, F_1^{ \triangle }( t_1 n , Y^{ [0] } ) 
          }
          -
          \mathcal{Z}_{ t_1 n + 1 }^{ 
            0, \delta, b / n, F_1^{ \Box }( t_1 n , Y^{ [0] } ) 
          }
        \big|
      \Big]
  }{
    2 \, n \, 
    e^{ b ( t_3 - t_2 ) }
    \,
    \P\!\left(
      \forall \, s \in [ \nicefrac{ 2 }{ n } , \nicefrac{ 1 }{ 2 } ] 
      \colon 
      \mathcal{Z}_s^{ 0, \delta, b, W^{ \triangle } } > 0
    \right)
  }
  .
\end{multline}
\end{enumerate}
\end{lem}

\begin{proof}[Proof of Lemma~\ref{lem1100}]
Throughout this proof for every 
$
  \ast \in \{ \triangle, \Box \} 
$
let 
$
  U^{ \ast } \colon \Omega \to \R
$ 
be the random variable given by 
\begin{equation}
  U^{ \ast } = \mathcal{Z}_{ t_2 }^{ 0, \delta, b, W^{ \ast } }
\end{equation}
and let $\tilde Y\colon \Omega\to[\mathcal{C}_0]^3\times \mathcal{C}_{00}$
be the random variable given by
\begin{equation}
\tilde Y
=\left(
\big(\sqrt{n}\, \tilde{W}^{(1)}_{t/n}\big)_{t\in[0,\infty)},
\big(\sqrt{n}\, \tilde{W}^{\triangle}_{t/n}\big)_{t\in[0,\infty)},
\big(\sqrt{n}\, \tilde{W}^{(2)}_{t/n}\big)_{t\in[0,\infty)},
B
\right).
\end{equation}
Observe that the fact that $ n \geq 5 $
ensures that
\begin{equation}
\label{eq:bounds_for_t23}
  t_2 \leq \nicefrac{ 2 }{ n } \leq \nicefrac{ 2 }{ 5 } < \nicefrac{ 1 }{ 2 } \leq t_3
  .
\end{equation}
This and the fact that
$
  0 \leq t_1 < \nicefrac{ 1 }{ n } \leq t_2
$
prove that
\begin{equation}
\label{eq:t123}
  0 \leq t_1 < t_2 < t_3 .
\end{equation}
Moreover, note that
\begin{equation}
\begin{split}
&
  \P\big(
    \inf\nolimits_{ 
      s \in [ t_2, t_3 ]
    }
    \max\nolimits_{\ast \in \{ \triangle, \Box \}}
      \mathcal{Z}_s^{ 0, \delta, b, W^{ \ast } } > 0
  \big) 
\\
&
  \geq
  \P\big(
    \inf\nolimits_{ 
      s \in [ t_2, t_3 ]
    }
    \mathcal{Z}_s^{ 0, \delta, b, W^{ \triangle } }
    > 0
  \big) 
  =
  \P\big(
    \forall \, 
      s \in [ t_2, t_3 ]
    \colon
    \mathcal{Z}_s^{ 0, \delta, b, W^{ \triangle } }
    > 0
  \big) 
  > 0 
  .
\end{split}
\end{equation}
Next observe that
items~\eqref{triv10item01}--\eqref{triv10item01a} of Lemma~\ref{triv10} 
imply that 
\begin{equation}
\begin{split}
&
  \E\!\left[
    \big|
      \mathcal{Z}_{ t_3 }^{ 0, \delta, b, W^{ \triangle } }
      -
      \mathcal{Z}_{ t_3 }^{ 0, \delta, b, W^{ \Box } }
    \big|
    \,
    \mathbbm{1}^{ \Omega }_{
      \{ 
        \exists \, s \in [ t_2, t_3 ]
        \colon 
        \max\{
          \mathcal{Z}_s^{ 0, \delta, b, W^{ \triangle } }
          ,
          \mathcal{Z}_s^{ 0, \delta, b, W^\Box}
        \} = 0
      \}
    }
  \right]
\\
&
  =
  \E\!\left[
    \big|
      \mathcal{Z}_{ t_3 - t_2 }^{
        U^{ \triangle }, \delta, b, \tilde{W}^{(2)} 
      }
      -
      \mathcal{Z}_{ 
        t_3 - t_2 
      }^{ U^{ \Box }, \delta, b, \tilde{W}^{(2)} 
      }
    \big|
    \,
    \mathbbm{1}^{ \Omega }_{
      \{ 
        \exists \, s \in [0, t_3 - t_2 ]
        \colon
        \mathcal{Z}_s^{ 
          U^{ \triangle }, \delta, b, \tilde{W}^{(2)} 
        }
        =
        \mathcal{Z}_s^{
          U^\Box, \delta, b, \tilde{W}^{(2)}
        } 
        = 0
      \}
    }
  \right]
\\
&
  = 0 .
\end{split}
\end{equation}
Hence, we obtain that
\begin{equation}
\label{eqs102}
\begin{split}
&
  \E\bigg[
    \big|
      \mathcal{Z}_{ t_3 }^{ 0, \delta, b, W^{ \triangle } }
      -
      \mathcal{Z}_{ t_3 }^{ 0, \delta, b, W^{ \Box } }
    \big|
    \,
    \Big|
    \Big\{
      \inf\nolimits_{ 
        s \in [ t_2, t_3 ]
      }
      \max_{\ast \in \{ \triangle, \Box \}}
        \mathcal{Z}_s^{ 0, \delta, b, W^{ \ast } } > 0
    \Big\}
  \bigg]
\\
&
  =
  \frac{
    \E\bigg[
      \big|
        \mathcal{Z}_{ t_3 }^{ 0, \delta, b, W^{ \triangle } }
        -
        \mathcal{Z}_{ t_3 }^{ 0, \delta, b, W^\Box }
      \big|
      \,
      \mathbbm{1}^{ \Omega }_{
        \{ 
          \inf\nolimits_{ 
            s \in [ t_2, t_3 ]
          }
          \max_{\ast \in \{ \triangle, \Box \}}
            \mathcal{Z}_{s}^{0,\delta,b,W^{ \ast }} > 0
        \}
      }
    \bigg]
  }{
    \P\!\left(
      \forall \, s \in [ t_2, t_3]
      \colon 
      \max_{\ast \in \{ \triangle, \Box \}}
        \mathcal{Z}_{s}^{0,\delta,b,W^{ \ast }} > 0
    \right)
  }
\\
&
  =
  \frac{
    \E\!\left[
      \big|
        \mathcal{Z}_{ t_3 }^{ 0, \delta, b, W^{ \triangle } }
        -
        \mathcal{Z}_{ t_3 }^{ 0, \delta, b, W^{ \Box } }
      \big|
    \right]
  }{
    \P\!\left(
      \forall \, s \in [ t_2, t_3 ]
      \colon 
      \max_{\ast \in \{ \triangle, \Box \}}
      \mathcal{Z}_s^{ 0, \delta, b, W^{ \ast }}      > 0
    \right)
  }
  .
\end{split}
\end{equation}
In the next step we note that
items~\eqref{triv10item02} and \eqref{triv10item02a} of Lemma~\ref{triv10} 
and \eqref{eq:bounds_for_t23}
imply that
\begin{equation}
\label{eqs103}
\begin{split}
&\P \! \left(
\forall \, s\in[ t_2, t_3]\colon 
\max_{\ast \in \{ \triangle, \Box \}}
\mathcal{Z}_{s}^{0,\delta,b,W^{ \ast }}>0
\right)
\\
&\leq
2\cdot \P \!\left(
\forall \, s\in[ t_2, t_3]\colon 
\mathcal{Z}_{s}^{0,\delta,b,W^{ \triangle }}>0
\right)
\\
&\leq
2\cdot \P \!\left(
\forall \, s\in[\tfrac{2}{n},\tfrac{1}{2}]\colon 
\mathcal{Z}_{s}^{0,\delta,b,W^{ \triangle }}>0
\right)
  .
\end{split}
\end{equation}
In addition, we observe that
item~\eqref{triv10item01a} of Lemma~\ref{triv10}
ensures that
\begin{align}\label{nr1}
\E \! \left[
\big|
\mathcal{Z}_{ t_3}^{0,\delta,b,W^{ \triangle }}
-\mathcal{Z}_{ t_3}^{0,\delta,b,W^\Box}
\big|
\right]
=
\E \! \left[
\big|
\mathcal{Z}_{ t_3- t_2}^{U^{ \triangle },\delta,b,\tilde{W}^{(2)}}
-\mathcal{Z}_{ t_3- t_2}^{U^\Box,\delta,b,\tilde{W}^{(2)}}
\big|
\right].
\end{align}
Furthermore, we note that
item~\eqref{triv10item01} of Lemma~\ref{triv10} implies that $(U^{ \triangle },U^\Box)$ and $\tilde{W}^{(2)}$
are independent.
Combining this with \eqref{nr1}
and Lemma~\ref{lipsc}
assures that
\begin{align}
\begin{split}
\E \! \left[
\big|
\mathcal{Z}_{ t_3}^{0,\delta,b,W^{ \triangle }}
-\mathcal{Z}_{ t_3}^{0,\delta,b,W^\Box}
\big|
\right]
&=
e^{-b( t_3- t_2)}\cdot
\E \! \left[
\big|
U^{ \triangle }
-U^\Box
\big|
\right]
\\
&=
e^{-b( t_3- t_2)}\cdot
\E \! \left[
\big|
\mathcal{Z}_{ t_2}^{0,\delta,b,W^{ \triangle }}
-\mathcal{Z}_{ t_2}^{0,\delta,b,W^\Box}
\big|
\right].
\end{split}
\end{align}
Combing this with Lemma~\ref{lem101} assures that 
\begin{multline}
\E \! \left[
\big|
\mathcal{Z}_{ t_3}^{0,\delta,b,W^{ \triangle }}
-\mathcal{Z}_{ t_3}^{0,\delta,b,W^\Box}
\big|
\right]
\\
=
e^{-b( t_3- t_2)}\cdot
\tfrac{1}{n}\cdot {\E \! \left[
\Big|
\mathcal{Z}_{n\cdot  t_2}^{0,\delta,b/n,(\sqrt{n}W^{ \triangle }_{t/n})_{t\in [0,\infty)}}
-\mathcal{Z}_{n\cdot  t_2}^{0,\delta,b/n,(\sqrt{n}W^\Box_{t/n})_{t\in [0,\infty)}}
\Big|
\right]}.
\end{multline}
The fact that
\begin{equation}
t_2 n
=t_1 n+1
\end{equation}
and the fact that
for every
$\ast\in\{\triangle, \Box\}$,
$t\in [0,\infty)$
it holds that
\begin{equation}
\sqrt{n}W^\ast_{t/n}
=
\sqrt{n}(F_n^\ast( t_1,Y^{[0]}))_{t/n}
=(F_1^\ast(t_1 n,\tilde Y))_{t}
\end{equation}
therefore demonstrate that
\begin{equation}\label{eqs104}
\begin{split}
&\E \! \left[
\big|
\mathcal{Z}_{ t_3}^{0,\delta,b,W^{ \triangle }}
-\mathcal{Z}_{ t_3}^{0,\delta,b,W^\Box}
\big|
\right]
\\
&=
e^{-b( t_3- t_2)}\cdot
\tfrac{1}{n}\cdot {\E \! \left[
\Big|
\mathcal{Z}_{t_1 n+1}^{0,\delta,b/n,F_1^{ \triangle }(t_1 n,\tilde Y)}
-\mathcal{Z}_{t_1 n+1}^{0,\delta,b/n,F_1^\Box(t_1 n,\tilde Y)}
\Big|
\right]}
\\
&=
e^{-b( t_3- t_2)}\cdot
\tfrac{1}{n}\cdot {\E \! \left[
\Big|
\mathcal{Z}_{t_1 n+1}^{0,\delta,b/n,F_1^{ \triangle }(t_1 n,Y^{[0]})}
-\mathcal{Z}_{t_1 n+1}^{0,\delta,b/n,F_1^\Box(t_1 n,Y^{[0]})}
\Big|
\right]}.
\end{split}
\end{equation}
Combining \eqref{eqs102}, \eqref{eqs103}, and \eqref{eqs104} yields \eqref{claim1054}.
The proof of Lemma~\ref{lem1100} is thus completed.
\end{proof}

\begin{lem}
\label{lem11}
Assume the setting in Section~\ref{setproc}.
Then
\begin{multline}
\label{eq:rate_of_convergence_Zdifference}
  \inf_{ n \in \N \cap [5,\infty) }
  \Bigg(
  n^{ \delta / 2 }
  \cdot
  \inf_{ 
    r \in [ 0, \nicefrac{ 1 }{ 2 } ] 
  }
  \E\bigg[
    \big|
      \mathcal{Z}_{ 1 - r }^{ 
        0, 
        \delta, 
        b, 
        F^{ \triangle }_n( \mathfrak{T}^1_n( r ) , Y^{ [0] } ) 
      }
      -
      \mathcal{Z}_{ 1 - r }^{ 
        0,
        \delta,
        b,
        F^\Box_n( \mathfrak{T}^1_n(r), Y^{ [0] } ) 
      }
    \big|
    \,
    \Big|
    \\
    \Big\{
      \forall \, 
      s \in [ \mathfrak{T}^2_n(r), 1 - r ]
      \colon
      \max\nolimits_{
        \ast \in \{ \triangle, \Box \}
      }
      \mathcal{Z}_s^{ 
        0, \delta, b, F^{ \ast }_n( \mathfrak{T}^1_n(r), Y^{ [0] } ) 
      }
      > 0
    \Big\}
  \bigg]
  \Bigg) > 0
  .
\end{multline}
\end{lem}

\begin{proof}[Proof of Lemma~\ref{lem11}]
Inequality~\eqref{eq:rate_of_convergence_Zdifference} is an immediate consequence of Lemma~\ref{hitnotzero},
Lemma~\ref{lemdist}, and Lemma~\ref{lem1100}.
The proof of Lemma~\ref{lem11} is thus completed.
\end{proof}

\begin{lem}
\label{lemmain}
Assume the setting in Section~\ref{setproc},
let $ n \in \N $,
for every 
$ \ast \in \{ \triangle, \Box \} $,
$ r \in [0,1] $
let 
$
  \mathcal{W}^{ *, r }
  \colon \Omega \to \mathcal{C}_0
$
be the Brownian motion given by 
$
  \mathcal{W}^{ *, r }
  = F^{ \ast }_n( \mathfrak{T}^1_n( r ), Y^{[0]} ) 
$,
and 
for every $ r \in [0,1] $ 
let $ E_r \in \R $
be the real number given by
\begin{align}
&
  E_r
  =
\\ &
\nonumber
  \E\bigg[
    \big|
      \mathcal{Z}_{ 1 - r }^{ 0, \delta, b, \mathcal{W}^{ \triangle, r } }
      -
      \mathcal{Z}_{ 1 - r }^{ 0, \delta, b, \mathcal{W}^{ \Box, r } }
    \big|
    \, 
    \Big| 
    \,
    \Big\{
      \forall \, s \in [ \mathfrak{T}^n_2(r) , 1 - r ] 
      \colon
      \max_{
        \ast \in \{ \triangle, \Box \}
      }
      \mathcal{Z}_s^{ 0, \delta, b, \mathcal{W}^{ \ast, r } }
      > 0 
    \Big\}
  \bigg]
  .
\end{align}
Then it holds $\P$-a.s.\ that
\begin{align}
\label{calim1234}
   \E\Big[
      \big|
        \mathcal{Z}_1^{ Z, \delta, b, W^{ (n), \triangle } }
        -
        \mathcal{Z}_1^{ Z, \delta, b, W^{ (n), \Box } }
      \big|
      \,
      \Big|
      \,
      \sigma_\Omega(\gamma_n)
    \Big]
    \mathbbm{1}^{ \Omega }_{
      \{ 
        0 \leq \gamma_n \leq 1
      \} 
    }
    =
    E_{ \min\{\gamma_n,1\} }
    \,
    \mathbbm{1}^{ \Omega }_{
      \{ 
        0 \leq \gamma_n \leq 1
      \} 
    }.
\end{align}
\end{lem}

\begin{proof}[Proof of Lemma~\ref{lemmain}]
Throughout this proof let $ E_{ \infty } $ be the real number given by
$
  E_{ \infty } = 0 
$,
let 
$
  \mathbb{B} =
  [0,\infty)^2 \times [ \mathcal{C}_0 ]^2 \times [ C( [0,\infty), \R ) ]^2 
$,
for every $ r \in [0,\infty] $
let
$ 
  \P_r \colon 
  \mathcal{B}( 
    \mathbb{B}
  ) 
  \to [0,1]
$
be the probability measures 
which satisfy
for all 
$
  B \in \mathcal{B}( \mathbb{B} )
$ 
that
\begin{equation}
  \P_r( B )
  =
  \begin{cases}
    \P\!\left(
      \Phi_n( r, Y^{ [0] } )
      \in B
      \, \big| \,
      \{ 
        \Phi_{ n, 2 }( r, Y^{ [0] } ) > 1
      \}
    \right)
  &
    \colon
    r < \infty
  \\
  \P\!\left(
    \Phi_n( 0, Y^{ [0] } )
    \in B
    \, \big| \,
    \{
      \Phi_{ n, 2 }( 0, Y^{ [0] } ) 
      > 1
    \}
  \right)
  &
    \colon
    r = \infty
  \end{cases}
  ,
\end{equation}
and let 
$
  G \colon \mathbb{B} \to [0,\infty)
$ 
be the function which satisfies 
for all 
$ 
  x = ( x_1, \dots, x_6 ) 
  \in \mathbb{B}
$ 
that
\begin{equation}
  G( x )
  =
  \left|
    x_5( | 1 - x_1 | )
    -
    x_6( | 1 - x_1 | )
  \right|.
\end{equation}
Observe that for all $ \ast \in \{ \triangle, \Box \} $
it holds that
\begin{equation}
  \P\!\left(
    Z_1^{ (n), \ast }
    =
    ( 
      X_{ \mathfrak{v}( \ast ) + 2 }^{ (n), [ \mathcal{M}_n ] } 
    )_{ 1 - X_1^{ (n),[\mathcal{M}_n] } }
  \right)
  = 1
  .
\end{equation}
Lemma~\ref{lembasic1} hence implies that
for all $ \ast \in \{ \triangle, \Box \} $
it holds that
\begin{equation}
\label{99}
  \P\!\left(
    \mathcal{Z}_1^{ Z, \delta, b, W^{ (n), \ast } } 
    =
    ( X_{ \mathfrak{v}( \ast ) + 2 }^{ (n), [ \mathcal{M}_n ] } )_{ 1 - X_1^{ (n), [\mathcal{M}_n] } }
  \right)
  = 1
  .
\end{equation}
Next observe for every $ r \in [0,1] $ that
it holds that
$  
  \Phi_{ n, 2 }( r, Y^{ [0] } ) > 1
$
if and only if it holds that
$
  \mathcal{T}_n( r, Y^{ [0] } ) > 1 - r
$.
Item~\eqref{triv10item03} of Lemma~\ref{triv10} 
therefore assures that for all 
$ r \in [0,1] $ 
it holds that
\begin{equation}
\label{102}
  \P\!\left(
  \left[
    \Phi_{ n, 2 }( r, Y^{ [0] } ) > 1
  \right]
  \Leftrightarrow
  \Big[
    \forall \, s \in [ \mathfrak{T}^n_2(r) , 1 - r ]
    \colon 
    \max_{
      \ast \in \{ \triangle, \Box \}
    }
    \mathcal{Z}_s^{
      0, \delta, b, \mathcal{W}^{ r, \ast } 
    }
    > 0 
  \Big]
  \right) = 1
  .
\end{equation}
Hence, we obtain that for all $ r \in [0,1] $ it holds that
\begin{equation}\label{asdfg}
\begin{split}
  &\int G( x ) \, \P_r( \mathrm{d}x )\\
&
  =
  \E\Big[
    \,
    G\big( \Phi_n( r, Y^{ [0] } ) \big)
    \, \big| \,
    \big\{ 
      \Phi_{ n, 2 }( r, Y^{ [0] } ) > 1
    \big\}
  \Big]
\\
&=
  \E\bigg[
    \big|
      \mathcal{Z}_{ 1 - r }^{ 0, \delta, b, \mathcal{W}^{ \triangle, r } }
      -
      \mathcal{Z}_{ 1 - r }^{ 0, \delta, b, \mathcal{W}^{ \Box, r } }
    \big|
    \, \Big| \,
    \cap_{
      s \in [ \mathfrak{T}^n_2(r), 1 - r ] 
    }
    \Big\{
      \max_{  
        \ast \in \{ \triangle, \Box \}
      }
      \mathcal{Z}_s^{ 0, \delta, b, \mathcal{W}^{ \ast, r } } 
      > 0
    \Big\}
  \bigg]
\\
&=
  E_r .
\end{split}
\end{equation}
Next observe that \eqref{99} yields that 
\begin{equation}
  \P\!\left(
    \big|
      \mathcal{Z}_1^{ Z, \delta, b, W^{ (n), \triangle } }
      -
      \mathcal{Z}_1^{ Z, \delta, b, W^{ (n), \Box } }
    \big|
    = 
    G( X^{ (n), [ \mathcal{M}_n ] } )
  \right)
  = 1 .
\end{equation}
Hence, we obtain that it holds $ \P $-a.s.\ that
\begin{equation}
\begin{split}
&
  \mathbbm{1}^{ [0,\infty] }_{ [0,1] }( \gamma_n )
  \cdot
  \E\!\left[
    \big|
      \mathcal{Z}_1^{ Z, \delta, b, W^{ (n), \triangle } }
      -
      \mathcal{Z}_1^{ Z, \delta, b, W^{ (n), \Box } }
    \big|
    \,
    \Big|
    \,
   \sigma_\Omega(\gamma_n)
  \right]
\\
&
=
  \mathbbm{1}^{[0,\infty]}_{ [0,1] }( \gamma_n )
\cdot
\E\!
\left[
G(X^{(n),[ \mathcal{M}_n ]})
\, \big|\,
\sigma_\Omega(\gamma_n)
\right]\\
&=
\mathbbm{1}^{ [0,\infty] }_{[0,1]}(\gamma_n)
\cdot
\int G(x)\,\P\big(X^{(n),[ \mathcal{M}_n ]}
\in\mathrm dx\, \big|\,
\sigma_\Omega(\gamma_n)\big).
\end{split}
\end{equation}
Lemma~\ref{lembasic2} therefore assures that it holds $ \P $-a.s.\ that
\begin{equation}
\begin{split}
  &\mathbbm{1}^{ [0,\infty] }_{ [0,1] }( \gamma_n )
  \cdot
  \E\!\left[
    \left|
      \mathcal{Z}_1^{ Z, \delta, b, W^{ (n), \triangle } }
      -
      \mathcal{Z}_1^{ Z, \delta, b, W^{ (n), \Box } }
    \right|
    \,
    \bigg|
    \,
  \sigma_\Omega(\gamma_n)
  \right]\\
&=
  \left[
    \int G \, \mathrm{d}\P_{ \min\{\gamma_n,1\} }
  \right]
  \mathbbm{1}^{ [0,\infty] }_{ [0,1] }( \gamma_n )
  .
\end{split}
\end{equation}
Equation \eqref{asdfg} hence demonstrates that it holds $ \P $-a.s.\ that
\begin{multline}
  \mathbbm{1}^{ [0,\infty] }_{ [0,1] }( \gamma_n )
  \cdot
  \E\!\left[
    \left|
      \mathcal{Z}_1^{ Z, \delta, b, W^{ (n), \triangle } }
      -
      \mathcal{Z}_1^{ Z, \delta, b, W^{ (n), \Box } }
    \right|
    \,
    \bigg|
    \,
    \sigma_\Omega(\gamma_n)
  \right]
  \\
=
  E_{ \min\{\gamma_n,1\} }
  \cdot
  \mathbbm{1}^{ [0,\infty] }_{ [0,1] }( \gamma_n )
  .
\end{multline}
This establishes \eqref{calim1234}. 
The proof of Lemma~\ref{lemmain} is thus completed.
\end{proof}

\subsubsection{A lower bound for hitting time probabilities}

\begin{lem}
\label{lem5}
Assume the setting in Section~\ref{setproc}.
Then 
\begin{equation}
  \inf_{ n \in \N } 
  \Big[
  \P\big(
    0 \leq \gamma_n \leq 
    \nicefrac{ 1 }{ 2 } 
  \big) 
  \Big] 
  > 0 .
\end{equation}
\end{lem}

\begin{proof}[Proof of Lemma~\ref{lem5}]
First, observe that for all
$n\in\N$
it holds that
\begin{equation}
\big(
\gamma_n\in [0,1]
\big)
\iff
\big(
\mathcal{M}_n\neq 0
\big)
\iff
\big(
X_1^{(n),[1]}\leq 1
\big)
\iff
\big(
X_2^{(n),[0]}\leq 1
\big).
\end{equation}
Hence, we obtain that for all $n\in\N$ it holds that
\begin{equation}
\P\!
\left(
\big(
\gamma_n\in [0,1]
\big)
\iff
\big(
\exists\, t\in [0,1]\colon
\mathcal{Z}_t^{Z,\delta,b,\tilde{W}}
=0
\big)
\right)
=1.
\end{equation}
This ensures that for all
$n\in\N$, $\ast\in\{\triangle,\Box\}$
it holds that
\begin{equation}
\P\!
\left(
\big(
\gamma_n\in [0,1]
\big)
\iff
\big(
\exists\, t\in [0,1]\colon
Z_t^{(n),\ast}
=0
\big)
\right)
=1.
\end{equation}
This and Lemma~\ref{lembasic1} demonstrate that for all
$n\in\N$, $\ast\in\{\triangle,\Box\}$
it holds that
\begin{equation}\label{109}
\P\!
\left(
\big(\gamma_n\in [0,1]\big)
\iff
\big(
\exists\, t\in [0,1]\colon
\mathcal{Z}_t^{Z,\delta,b,W^{(n),\ast}}
=0
\big)
\right)
=1.
\end{equation}
Next note that for all
$n\in\N$, $\ast\in\{\triangle,\Box\}$
it holds that
\begin{equation}
\P\!
\left(
\forall\, m\in\N\colon
Z^{(n),\ast}_{X^{(n),[m]}_1}
=0
\right)
=1.
\end{equation}
This shows that for all
$n\in\N$, $\ast\in\{\triangle,\Box\}$
it holds that
\begin{equation}
\P\!
\left(
\big(\gamma_n\in [0,1]\big)
\implies
\big(
Z^{(n),\ast}_{\min\{\gamma_n,1\}}
=0
\big)
\right)
=1.
\end{equation}
Lemma~\ref{lembasic1} hence proves
that for all
$n\in\N$, $\ast\in\{\triangle,\Box\}$
it holds that
\begin{equation}\label{201}
\P\!
\left(
\big(\gamma_n\in [0,1]\big)
\implies
\big(\mathcal{Z}_{\min\{\gamma_n,1\}}^{Z,\delta,b,W^{(n),\ast}}
=0
\big)
\right)
=1.
\end{equation}
Combining \eqref{109} and \eqref{201} demonstrates that for all
$n\in\N$, $\ast\in\{\triangle,\Box\}$
it holds that
\begin{equation}
\begin{split}
&\P \big(0\leq \gamma_n\leq \nicefrac{1}{2}\big)
\\
&=
\P \big(\big\{0\leq \gamma_n\leq \nicefrac{1}{2}\big\}
\\
&\quad
\cap
\big\{
\big(\gamma_n\in [0,1]\big)
\iff
\big(
\exists\, t\in [0,1]\colon
\mathcal{Z}_t^{Z,\delta,b,W^{(n),\ast}}
=0
\big)
\big\}
\\
&\quad
\cap
\big\{
\big(\gamma_n\in [0,1]\big)
\implies
\big(\mathcal{Z}_{\min\{\gamma_n,1\}}^{Z,\delta,b,W^{(n),\ast}}
=0
\big)
\big\}
\big)
\\
&\geq
\P \big(
\big\{
\exists \, t\in[0,\nicefrac{1}{2}]\colon 
\mathcal{Z}_t^{Z,\delta,b,W^{(n),\ast}}
=0
\big\}
\cap
\big\{
\forall \, t\in[1/2,1]\colon 
\mathcal{Z}_t^{Z,\delta,b,W^{(n),\ast}}
>0
\big\}
\\
&\quad
\cap
\big\{
\big(\gamma_n\in [0,1]\big)
\iff
\big(
\exists\, t\in [0,1]\colon
\mathcal{Z}_t^{Z,\delta,b,W^{(n),\ast}}
=0
\big)
\big\}
\\
&\quad
\cap
\big\{
\big(\gamma_n\in [0,1]\big)
\implies
\big(\mathcal{Z}_{\min\{\gamma_n,1\}}^{Z,\delta,b,W^{(n),\ast}}
=0
\big)
\big\}
\big)
\\
&=\P \big(
\big\{
\exists \, t\in[0,\nicefrac{1}{2}]\colon 
\mathcal{Z}_t^{Z,\delta,b,W^{(n),\ast}}
=0
\big\}
\cap
\big\{
\forall \, t\in[1/2,1]\colon 
\mathcal{Z}_t^{Z,\delta,b,W^{(n),\ast}}
>0
\big\}
\big).
\end{split}
\end{equation}
This and Lemma~\ref{lembasic1} show that for all $n\in\N$ it holds that
\begin{multline}
\P \big(0\leq \gamma_n\leq \nicefrac{1}{2}\big)
\\
\geq
\P \big(
\big\{
\exists \, t\in[0,\nicefrac{1}{2}]\colon 
\mathcal{Z}_t^{Z,\delta,b,\tilde W}
=0
\big\}
\cap
\big\{
\forall \, t\in[1/2,1]\colon 
\mathcal{Z}_t^{Z,\delta,b,\tilde W}
>0
\big\}
\big)
>0.
\end{multline}
The proof of Lemma~\ref{lem5} is thus completed.
\end{proof}

\subsubsection{A refined lower bound for strong
\texorpdfstring{$ L^1 $}{L1}-distances
between the constructed squared Bessel processes}

\begin{lem}
\label{lem12}
Assume the setting in Section~\ref{setproc}.
Then 
\begin{equation}
\label{eq:lem12}
  \inf_{ 
    n \in \N \cap [5,\infty)
  }
  \left(
    n^{ \delta / 2 }
    \cdot 
    \E\!\left[
      \big|
        \mathcal{Z}_1^{ Z, \delta, b, W^{ (n), \triangle } }
        -
        \mathcal{Z}_1^{ Z, \delta, b, W^{ (n), \Box } }
      \big|
    \right]
  \right)
  > 0
  .
\end{equation}
\end{lem}

\begin{proof}[Proof of Lemma~\ref{lem12}]
Throughout this proof 
for every 
$ n \in \N \cap [5,\infty) $,
$ \ast \in \{ \triangle, \Box \} $,
$ r \in [0,1] $
let 
$
  \mathcal{W}^{ n, \ast, r } \colon \Omega \to \mathcal{C}_0
$
be the Brownian motion given by
$
  W^{ n, \ast, r } = F^{ \ast }_n( \mathfrak{T}^1_n(r), Y^{ [0] } )
$
and
for every 
$ n \in \N \cap [5,\infty) $,
$ r \in [0,1] $
let 
$ E_{ n, r } \in \R $
be the real number given by
\begin{align}
&
  E_{ n, r }
  =
\\ &
\nonumber
  \E\bigg[
    \big|
      \mathcal{Z}_{ 1 - r }^{ 0, \delta, b, \mathcal{W}^{ n, \triangle, r } }
      \!\!
      -
      \mathcal{Z}_{ 1 - r }^{ 0, \delta, b, \mathcal{W}^{ n, \Box, r } }
    \big|
    \, 
    \Big| 
    \,
    \cap_{
      s \in [ \mathfrak{T}^n_2(r) , 1 - r ] 
    }
    \Big\{
      \max_{
        \ast \in \{ \triangle, \Box \}
      }
      \mathcal{Z}_s^{ 0, \delta, b, \mathcal{W}^{ n, \ast, r } }
      > 0 
    \Big\}
  \bigg]
  .
\end{align}
Next observe that the tower property for conditional expectations ensures that
for all $ n \in \N \cap [5,\infty) $ it holds that
\begin{equation}
\begin{split}
&
  \E\!\left[
    \big|
      \mathcal{Z}_1^{ Z, \delta, b, W^{ (n), \triangle } }
      -
      \mathcal{Z}_1^{ Z, \delta, b, W^{ (n), \Box } }
    \big| 
  \right]
\\
&
  =
  \E\bigg[
    \,
    \E\!\left[
      \big|
        \mathcal{Z}_1^{ Z, \delta, b, W^{ (n), \triangle } }
        -
        \mathcal{Z}_1^{ Z, \delta, b, W^{ (n), \Box } }
      \big|
      \, \Big| \,
      \sigma_\Omega(\gamma_n)
    \right]
  \bigg]
  .
\end{split}
\end{equation}
Combining this with Lemma~\ref{lemmain} implies that for all $ n \in \N \cap [5,\infty) $ 
it holds that
\begin{equation}
\begin{split}
&
  \E\!\left[
    \big|
      \mathcal{Z}_1^{ Z, \delta, b, W^{ (n), \triangle } }
      -
      \mathcal{Z}_1^{ Z, \delta, b, W^{ (n), \Box } }
    \big| 
  \right]
\\
&
\geq
  \E\bigg[
    \E\!\left[
      \big|
        \mathcal{Z}_1^{ Z, \delta, b, W^{ (n), \triangle } }
        -
        \mathcal{Z}_1^{ Z, \delta, b, W^{ (n), \Box } }
      \big|
      \, \Big| \,
      \sigma_\Omega(\gamma_n)
    \right]
    \mathbbm{1}^{ \Omega }_{ \{ 0 \leq \gamma_n \leq 1 \} }
  \bigg]
\\
&
=
  \E\Big[
    E_{ n, \gamma_n }
    \mathbbm{1}^{ \Omega }_{ \{ 0 \leq \gamma_n \leq 1 \} }
  \Big]
\geq
  \E\Big[
    E_{ n, \gamma_n }
    \mathbbm{1}^{ \Omega }_{ \{ 0 \leq \gamma_n \leq \nicefrac{ 1 }{ 2 } \} }
  \Big]
  .
\end{split}
\end{equation}
Hence, we obtain that for all $ n \in \N \cap [5,\infty) $ 
it holds that
\begin{equation}
\begin{split}
&
  \E\!\left[
    \big|
      \mathcal{Z}_1^{ Z, \delta, b, W^{ (n), \triangle } }
      -
      \mathcal{Z}_1^{ Z, \delta, b, W^{ (n), \Box } }
    \big| 
  \right]
\\
&
\geq
  \E\Big[
    \mathbbm{1}^{ \Omega }_{ \{ 0 \leq \gamma_n \leq \nicefrac{ 1 }{ 2 } \} }
  \Big]
  \bigg[
    \inf_{ r \in [0, \nicefrac{ 1 }{ 2 }] }
    E_{ n, r }
  \bigg]
=
  \P\Big(
    \gamma_n \in [0, \nicefrac{ 1 }{ 2 }] 
  \Big) 
  \bigg[
    \inf_{ r \in [0, \nicefrac{ 1 }{ 2 }] }
    E_{ n, r }
  \bigg]
  .
\end{split}
\end{equation}
This assures that
\begin{equation}
\begin{split}
&
  \inf_{ n \in \N \cap [5,\infty) }
  \left(
  n^{ \delta / 2 }
  \cdot
  \E\!\left[
    \big|
      \mathcal{Z}_1^{ Z, \delta, b, W^{ (n), \triangle } }
      -
      \mathcal{Z}_1^{ Z, \delta, b, W^{ (n), \Box } }
    \big| 
  \right]
  \right)
\\
&
\geq
  \inf_{ n \in \N \cap [5,\infty) }
  \left(
  \P\Big(
    \gamma_n \in [0, \nicefrac{ 1 }{ 2 }] 
  \Big) 
  \cdot
  \left(
  n^{ \delta / 2 }
  \cdot
    \inf_{ r \in [0, \nicefrac{ 1 }{ 2 }] }
    E_{ n, r }
  \right)
  \right)
\\
&
\geq
  \left[
    \inf_{ n \in \N \cap [5,\infty) }
    \P\Big(
      \gamma_n \in [0, \nicefrac{ 1 }{ 2 }] 
    \Big) 
  \right]
  \cdot
  \left[
    \inf_{ n \in \N \cap [5,\infty) }
    \left(
    n^{ \delta / 2 }
    \cdot
      \inf_{ r \in [0, \nicefrac{ 1 }{ 2 }] }
      E_{ n, r }
    \right)
  \right]
  .
\end{split}
\end{equation}
Combining this with Lemma~\ref{lem5}
and
Lemma~\ref{lem11}
establishes \eqref{eq:lem12}.
The proof of Lemma~\ref{lem12} is thus completed.
\end{proof}

\subsection{Proofs for the lower error bounds}
\label{proofofmaintheo}

\begin{lem}
\label{proplowerbound}
Assume the setting in Section~\ref{setproc}.
Then there exists a real number $ c \in (0,\infty) $ 
such that 
for all $ n \in \N $ 
it holds that
\begin{equation}
  \inf_{ 
    \substack{
      \varphi \colon \R^n \to \R
    \\
      \text{Borel-measurable}
    }
  }
  \E\!\left[
    \big|
      \mathcal{Z}_1^{ Z, \delta, b, \tilde{W} }
      -
      \varphi( 
        \tilde{W}_{ 1 / n }, 
        \tilde{W}_{ 2 / n }, 
        \dots, 
        \tilde{W}_1 
      )
    \big|
  \right]
  \geq 
  c \cdot n^{ - \delta / 2 } 
  .
\end{equation}
\end{lem}

\begin{proof}[Proof of Lemma~\ref{proplowerbound}]
Throughout this proof let 
$ e = ( e_n )_{ n \in \N } \colon \N \to [0,\infty] $ be the function which 
satisfies for all $ n \in \N $ that
\begin{equation}
  e_n
  =
  \inf_{ 
    \substack{
      \varphi \colon \R^n \to \R
    \\
      \text{Borel-measurable}
    }
  }
  \E\!\left[
    \big|
      \mathcal{Z}_1^{ Z, \delta, b, \tilde{W} }
      -
      \varphi( 
        \tilde{W}_{ 1 / n }, 
        \tilde{W}_{ 2 / n }, 
        \dots, 
        \tilde{W}_1 
      )
    \big|
  \right]
\end{equation}
and let $ c, C \in [0,\infty] $ be the real numbers given by
\begin{equation}
  C 
  =
  \inf_{ n \in \N \cap [5,\infty) }
  \left(
    n^{ \delta / 2 }
    \cdot
      \E\!\left[
        \big|
          \mathcal{Z}_1^{ Z, \delta, b, W^{ (n), \triangle } }
          -
          \mathcal{Z}_1^{ Z, \delta, b, W^{ (n), \Box } }
        \big| 
      \right]
  \right)
\end{equation}
and 
$ 
  c = \frac{ C }{ 24 }
$.
Note that 
items~\eqref{item:lembasic1_1} and \eqref{item:lembasic1_2} of Lemma~\ref{lembasic1} 
ensure that 
for all $ n \in \N $ 
and all Borel-measurable functions 
$ \varphi \colon \R^n \to \R $ 
it holds that
\begin{equation}
\label{eq:main_lemma_0}
\begin{split}
&
  \E\!\left[
    \big|
      \mathcal{Z}^{ Z, \delta, b, \tilde{W} }
      -
      \varphi( \tilde{W}_{ 1 / n }, \tilde{W}_{ 1 / n }, \dots, \tilde{W}_1 )
    \big| 
  \right]
\\
&
  =
  \tfrac{ 1 }{ 2 } 
  \left( 
    2 \cdot 
    \E\!\left[
      \big|
        \mathcal{Z}_1^{ Z, \delta, b, \tilde{W} }
        -
        \varphi( \tilde{W}_{ 1 / n }, \tilde{W}_{ 1 / n }, \dots, \tilde{W}_1 )
      \big| 
    \right]
  \right)
\\
&=
  \tfrac{ 1 }{ 2 } 
  \,
  \Big( 
    \E\!\left[
      \big|
        \mathcal{Z}_1^{ Z, \delta, b, W^{ (n), \triangle } }
        -
        \varphi\big( 
          W_{ 1 / n }^{ (n), \triangle },
          W_{ 2 / n }^{ (n), \triangle },
          \dots,
          W_1^{ (n), \triangle }
        \big)
      \big|
    \right]
\\
&
\quad
    +
    \E\!\left[
      \big|
        \mathcal{Z}_1^{ Z, \delta, b, W^{ (n), \Box } }
        -
        \varphi\big( 
          W_{ 1 / n }^{ (n), \Box },
          W_{ 2 / n }^{ (n), \Box },
          \dots,
          W_1^{ (n), \Box }
        \big)
      \big| 
    \right]
  \Big)
  \,
  .
\end{split}
\end{equation}
Next observe that Lemma~\ref{eq4l} ensures that
for all $ n \in \N $ 
and all Borel-measurable functions 
$ \varphi \colon \R^n \to \R $ 
it holds that
\begin{equation}
        \varphi\big( 
          W_{ 1 / n }^{ (n), \triangle },
          W_{ 2 / n }^{ (n), \triangle },
          \dots,
          W_1^{ (n), \triangle }
        \big)
        =
        \varphi\big( 
          W_{ 1 / n }^{ (n), \Box },
          W_{ 2 / n }^{ (n), \Box },
          \dots,
          W_1^{ (n), \Box }
        \big).
\end{equation}
Combining this with \eqref{eq:main_lemma_0} proves that
for all $ n \in \N $ 
and all Borel-measurable functions 
$ \varphi \colon \R^n \to \R $ 
it holds that
\begin{equation}
\begin{split}
&
  \E\!\left[
    \big|
      \mathcal{Z}_1^{ Z, \delta, b, \tilde{W} }
      -
      \varphi( \tilde{W}_{ 1 / n }, \tilde{W}_{ 1 / n }, \dots, \tilde{W}_1 )
    \big| 
  \right]
\\
&
=
  \tfrac{ 1 }{ 2 } 
  \,
  \Big( 
    \E\!\left[
      \big|
        \mathcal{Z}_1^{ Z, \delta, b, W^{ (n), \triangle } }
        -
        \varphi\big( 
          W_{ 1 / n }^{ (n), \triangle },
          W_{ 2 / n }^{ (n), \triangle },
          \dots,
          W_1^{ (n), \triangle }
        \big)
      \big|
    \right]
\\
&
\quad
    +
    \E\!\left[
      \big|
        \mathcal{Z}_1^{ Z, \delta, b, W^{ (n), \Box } }
        -
        \varphi\big( 
          W_{ 1 / n }^{ (n), \triangle },
          W_{ 2 / n }^{ (n), \triangle },
          \dots,
          W_1^{ (n), \triangle }
        \big)
      \big| 
    \right]
  \Big)
  \,
  .
\end{split}
\end{equation}
The triangle inequality hence implies that 
for all $ n \in \N $ 
and all Borel-measurable functions 
$ \varphi \colon \R^n \to \R $ 
it holds that
\begin{equation}
\begin{split}
&
  \E\!\left[
    \big|
      \mathcal{Z}_1^{ Z, \delta, b, \tilde{W} }
      -
      \varphi( \tilde{W}_{ 1 / n }, \tilde{W}_{ 1 / n }, \dots, \tilde{W}_1 )
    \big| 
  \right]
\\
&
  \geq 
  \tfrac{ 1 }{ 2 } 
  \, 
  \E\!\left[
    \big|
      \mathcal{Z}_1^{ Z, \delta, b, W^{ (n), \triangle } }
      -
      \mathcal{Z}_1^{ Z, \delta, b, W^{ (n), \Box } }
    \big| 
  \right]
  .
\end{split}
\end{equation}
This establishes that for all $ n \in \N \cap [5,\infty) $
it holds that
\begin{equation}
  e_n
  \geq 
  \tfrac{ 1 }{ 2 } 
  \, 
  \E\!\left[
    \big|
      \mathcal{Z}_1^{ Z, \delta, b, W^{ (n), \triangle } }
      -
      \mathcal{Z}_1^{ Z, \delta, b, W^{ (n), \Box } }
    \big| 
  \right]
  \geq 
  \left[
    \tfrac{ C }{ 2 }
  \right]
  \cdot n^{ - \delta / 2 }
  .
\end{equation}
The fact that
$
  \forall \, n \in \{ 1, 2, 3, 4 \} \colon
  e_n \geq e_{ 12 }
$
hence assures that for all $ n \in \N $ 
it holds that
\begin{equation}
\label{eq:main_lemma_end}
\begin{split}
    e_n 
&
  \geq 
    \min\!\left\{ 
      e_1, e_2, e_3, e_4, 
      \left[
        \tfrac{ C }{ 2 }
      \right]
      \cdot n^{ - \delta / 2 }
    \right\}
  \geq
    \min\!\left\{ 
      e_{ 12 }, 
      \left[
        \tfrac{ C }{ 2 }
      \right]
      \cdot n^{ - \delta / 2 }
    \right\}
\\
&
  \geq
    \min\!\left\{ 
      \left[
        \tfrac{ C }{ 2 }
      \right]
      \cdot 12^{ - \delta / 2 }
      ,
      \left[
        \tfrac{ C }{ 2 }
      \right]
      \cdot n^{ - \delta / 2 }
    \right\}
  \geq
      \left[
        \tfrac{ C }{ 2 }
      \right]
      \cdot 12^{ - \delta / 2 }
      \cdot n^{ - \delta / 2 }
\\ &
  \geq
      \left[
        \tfrac{ C }{ 24 }
      \right]
      \cdot n^{ - \delta / 2 }
  =
      c \cdot n^{ - \delta / 2 }
      .
\end{split}
\end{equation}
In the next step we observe that Lemma~\ref{lem12} 
proves that $ C > 0 $.
Hence, we obtain that $ c \in (0,\infty) $.
This and \eqref{eq:main_lemma_end} complete the proof
of Lemma~\ref{proplowerbound}.
\end{proof}

\section{Lower error bounds for CIR processes and squared Bessel processes in the general case}
\label{last}

\begin{lem}\label{lemlower1}
let $ \delta \in (0,2) $, $ b,x \in [0,\infty) $,
let 
$ \mathcal{C}_0 $ and $ \mathcal{C}_{ 00 } $ 
be the sets 
given by
$
  \mathcal{C}_0
  =
  \{
    f\in C([0,\infty),\R)\colon f(0)=0
  \}
$
and
$
  \mathcal{C}_{00}
=
  \{
    f \in C([0,1],\R)
  \colon
    f(0)=f(1)=0
  \}
$,
let $(\Omega,\mathfrak{F},\P)$ be a complete probability space,
let 
$ \tilde{W}, \tilde{W}^{(1)},\tilde{W}^{ \triangle },\tilde{W}^{(2)} \colon \Omega\to\mathcal{C}_0$ be Brownian motions,
let 
$ 
  B \colon \Omega \to \mathcal{C}_{00}
$ 
be a Brownian bridge,
let
$
  Y^{ [n] } 
  \colon \Omega \to 
  [ \mathcal{C}_0 ]^3 \times \mathcal{C}_{00}
$,
$ n \in \N_0 $,
be i.i.d.\ random variables 
with
$
  Y^{ [0] } = 
  ( \tilde{W}^{ (1) } , \tilde{W}^{ \triangle }, \tilde{W}^{ (2) }, B )
$,
assume that 
$
  \tilde{W}
$,
$
  \tilde{W}^{ (1) } 
$, 
$  
  \tilde{W}^{ \triangle }
$, 
$
  \tilde{W}^{(2)}
$, 
$ B $, 
$ Y^{ [1] } $,
$ Y^{ [2] } $,
$ \dots $ 
are independent,
let $ X \colon [0,\infty) \times \Omega \to [0,\infty) $
be a
$(\sigma_{ \Omega }(\{\{\tilde W_s\leq a\}\colon a\in\R, s\in [0,t]\}\cup\{A\in\mathfrak{F}\colon \P(A)=0\}))_{t\in [0,\infty)}$-adapted
stochastic process
with continuous sample paths
which satisfies that for all $ t \in [0,\infty) $ it holds $ \P $-a.s.\ that
\begin{equation}
  X_t
  =
  x
  + 
  \int_0^t
  \left(
    \delta - b X_s
  \right)\mathrm{d}s
  +
  \int_0^t
  2
  \sqrt{ X_s } \,\mathrm{d}\tilde W_s
  .
\end{equation}
Then there exists a real number $ c\in (0,\infty) $ such 
that for all $ N \in \N $ it holds that
\begin{equation}\label{finaleq}
  \inf_{
    \substack{
      \varphi \colon \R^N \to \R
    \\
      \text{Borel-measurable}
    }
  }
  \E\!\left[ 
    \big|
      X_1
      -
      \varphi( 
        \tilde W_{ \frac{ 1 }{ N } }, \tilde W_{ \frac{2}{N} }, \dots, 
        \tilde W_1
      )
    \big|
  \right]
  \geq 
  c
  \cdot 
  N^{ 
    - 
    {  \delta  }/{ 2 }
  }.
\end{equation}
\end{lem}

\begin{proof}[Proof of Lemma~\ref{lemlower1}]
Inequality \eqref{finaleq} is a consequence of Kallenberg~\cite[Theorem~21.14]{MR1876169}
and Lemma~\ref{proplowerbound}
(observe that all objects in Section~\ref{setproc} exist,
cf.\ Lemma~\ref{trivcor}
for the existence of
$
  W^{ (n), \ast }
  \colon \Omega \to \mathcal{C}_{ 0 }
$
and
$
  Z^{ (n), \ast }
  \colon [0,\infty) \times \Omega \to \R
$
for
$ n \in \N $,
$ \ast \in \{ \triangle, \Box \} $).
The proof of Lemma~\ref{lemlower1} is thus completed.
\end{proof}

\begin{cor}\label{cor1}
Let $ \delta \in (0,2) $, $ b,x \in [0,\infty) $,
let
$ 
  ( \Omega, \mathfrak{F}, \P ) 
$
be a probability space 
with a normal filtration $ ( \mathbb{F}_t )_{ t \in [0,1] } $,
let 
$ 
  W \colon [0,1] \times \Omega \to \R 
$
be a
$ 
  ( \mathbb{F}_t )_{ t \in [0,1] }
$-Brownian motion, 
let $ X \colon [0,1] \times \Omega \to [0,\infty) $
be a $ ( \mathbb{F}_t )_{ t \in [0,1] } $-adapted stochastic process
with continuous sample paths
which satisfies that for all $ t \in [0,1] $ it holds $ \P $-a.s.\ that
\begin{equation}
  X_t
  =
  x
  + 
  \int_0^t
  \left(
    \delta - b X_s
  \right)\mathrm{d}s
  +
  \int_0^t
  2
  \sqrt{ X_s } \,\mathrm{d}W_s
  .
\end{equation}
Then there exists a real number $ c\in (0,\infty) $ such 
that for all $ N \in \N $ it holds that
\begin{equation}
  \inf_{
    \substack{
      \varphi \colon \R^N \to \R
    \\
      \text{Borel-measurable}
    }
  }
  \E\!\left[ 
    \big|
      X_1
      -
      \varphi( 
        W_{ \frac{ 1 }{ N } }, W_{ \frac{2}{N} }, \dots, 
        W_1
      )
    \big|
  \right]
  \geq 
  c
  \cdot 
  N^{ 
    - 
    {  \delta  }/{ 2 }
  }.
\end{equation}
\end{cor}

\begin{proof}[Proof of Corollary~\ref{cor1}]
The claim follows directly from Lemma~\ref{lemlower1}.
\end{proof}

\begin{thm}[Cox-Ingersoll-Ross processes]
\label{thmgenpar}
Let 
$ T, a, \sigma \in (0,\infty) $, $ b, x \in [0,\infty) $
satisfy $ 2 a < \sigma^2 $,
let
$ 
  ( \Omega, \mathfrak{F}, \P ) 
$
be a probability space 
with a normal filtration $ ( \mathbb{F}_t )_{ t \in [0,T] } $,
let 
$ 
  W \colon [0,T] \times \Omega \to \R 
$
be a
$ 
  ( \mathbb{F}_t )_{ t \in [0,T] }
$-Brownian motion, 
let $ X \colon [0,T] \times \Omega \to [0,\infty) $
be a $ ( \mathbb{F}_t )_{ t \in [0,T] } $-adapted stochastic process
with continuous sample paths
which satisfies that for all $ t \in [0,T] $ it holds $ \P $-a.s.\ that
\begin{equation}
  X_t
  =
  x
  + 
  \int_0^t
  \left(
    a - b X_s
  \right)\mathrm{d}s
  +
  \int_0^t
  \sigma
  \sqrt{ X_s } \,\mathrm{d}W_s
  .
\end{equation}
Then there exists a real number $ c\in (0,\infty) $ such 
that for all $ N \in \N $ it holds that
\begin{equation}\label{finalfinaleq}
  \inf_{
    \substack{
      \varphi \colon \R^N \to \R
    \\
      \text{Borel-measurable}
    }
  }
  \E\!\left[ 
    \big|
      X_T
      -
      \varphi( 
        W_{ \frac{ T }{ N } }, W_{ \frac{ 2 T }{ N } }, \dots, 
        W_T
      )
    \big|
  \right]
  \geq 
  c
  \cdot 
  N^{ 
    - 
    { ( 2 a ) }/{ \sigma^2 }
  }.
\end{equation}
\end{thm}

\begin{proof}[Proof of Theorem~\ref{thmgenpar}]
Throughout this proof let 
$ 
  ( {\bf F}_t )_{ t \in [0,1] }
$ 
be the normal filtration on 
$
  ( \Omega, \mathfrak{F}, \P) 
$ 
which satisfies for all 
$ t \in [0,1] $ that 
$
  {\bf F}_t = \mathbb{F}_{ t T } 
$,
let 
$
  {\bf W} \colon [0,1] \times \Omega \to \R
$ 
be the 
$
	( {\bf F}_t )_{ t \in [0,1] }
$-Brownian motion 
which satisfies for all $ t \in [0,1] $ 
that 
$
  {\bf W}_t = \frac{ 1 }{ \sqrt{T} } W_{ t T }
$,
let 
$
  \delta = \nicefrac{ 4a }{ \sigma^2 } 
$,
$
  {\bf b} = T b 
$,
$
  \rho = \nicefrac{ 4 }{ ( T \sigma^2 ) } \in (0,\infty)
$,
$
  {\bf x} = \rho x 
$,
let 
$
  {\bf X} \colon [0,1] \times \Omega \to [0,\infty)
$ 
be the
$ ( {\bf F}_t )_{ t \in [0,1] } $-adapted stochastic process
with continuous sample paths
which satisfies 
for all $ t \in [0,1] $ that 
$
  {\bf X}_t = \rho X_{ t T }
$.
Observe that it holds that
\begin{equation}\label{needthe1}
  \delta \in (0,2),
  \qquad 
  {\bf b} \in [0,\infty),
  \
  \text{and}
  \qquad
  {\bf x} \in [0,\infty).
\end{equation}
Moreover, note that for all $t\in [0,1]$ it holds $\P$-a.s.\ that
\begin{equation}\label{needthe2}
\begin{split}
  {\bf X}_t
&
=
  \rho X_{ t T }
\\
&
  =
  \rho x
  +
  \rho
  \int_0^{ t T }
  \left( a - b X_s \right)
  \mathrm{d}s
  +
  \rho 
  \int_0^{ t T }
    \sigma \sqrt{ X_s } 
    \,
  \mathrm{d}W_s
\\
&
  =
  \rho x
  +
  \rho T 
  \int_0^t
  \left( a - b X_{ s T } \right)
  \mathrm{d}s
  +
  \rho \sqrt{T}
  \int_0^t
    \sigma \sqrt{ X_{ s T } }
    \,
  \mathrm{d}{\bf W}_s
\\
&
=
  \rho x
  +
  \rho T \int_0^t
    \left(a-b {\bf X}_{s}/\rho\right)
\mathrm{d}s
+\rho\sqrt{T}\int_0^{t}
\sigma\sqrt{{\bf X}_{s}/\rho}\,\mathrm{d}{\bf W}_s
\\
&=
{\bf x}
+\int_0^{t}
\left(\delta-{\bf b} {\bf X}_{s}\right)
\mathrm{d}s
+2\int_0^{t}
\sqrt{{\bf X}_{s}}\,\mathrm{d}{\bf W}_s.
\end{split}
\end{equation}
Next observe that for all $N\in\N$ it holds that
\begin{equation}\label{needthe3}
\begin{split}
&\inf_{
    \substack{
      \varphi \colon \R^N \to \R
    \\
      \text{Borel-measurable}
    }
  }
  \E\!\left[ 
    \big|
      X_T
      -
      \varphi\big( 
        W_{ 
          \frac{ T }{ N } 
        }, 
        W_{ 
          \frac{ 2 T }{ N } 
        }, 
        \dots, 
        W_T
      \big)
    \big|
  \right]
\\
&
  = \frac{ 1 }{ \rho }
\cdot
\inf_{
    \substack{
      \varphi \colon \R^N \to \R
    \\
      \text{Borel-measurable}
    }
  }
  \E\!\left[ 
    \big|
      {\bf X}_1
      -
      \varphi( 
        {\bf W}_{ \frac{ 1 }{ N } }, {\bf W}_{ \frac{ 2 }{ N } }, \dots, 
        {\bf W}_1
      )
    \big|
  \right].
\end{split}
\end{equation}
Combining \eqref{needthe1}, \eqref{needthe2}, and \eqref{needthe3} with Corollary~\ref{cor1} establishes \eqref{finalfinaleq}.
The proof of Theorem~\ref{thmgenpar} is thus completed.
\end{proof}

\section*{Acknowledgement}

Special thanks are due to Andr\'{e} Herzwurm for a series of fruitful discussions on this work.
This project has been supported through the
SNSF-Research project 200021\_156603
``Numerical approximations of nonlinear stochastic ordinary and partial differential equations''.
We gratefully acknowledge the Institute for Mathematical Research (FIM) at ETH Zurich which provided office space and partially organized the short visit of the first author to ETH Zurich in 2016 when part of this work was done.

\bibliographystyle{plain}
\bibliography{bib}

\begin{thebibliography}{10}

\bibitem{MR2186814}
Aur{\'e}lien Alfonsi.
\newblock On the discretization schemes for the {CIR} (and {B}essel squared)
  processes.
\newblock {\em Monte Carlo Methods Appl.}, 11(4):355--384, 2005.

\bibitem{MR3006996}
Aur{\'e}lien Alfonsi.
\newblock Strong order one convergence of a drift implicit {E}uler scheme:
  application to the {CIR} process.
\newblock {\em Statist. Probab. Lett.}, 83(2):602--607, 2013.

\bibitem{MR2367990}
Abdel Berkaoui, Mireille Bossy, and Awa Diop.
\newblock Euler scheme for {SDE}s with non-{L}ipschitz diffusion coefficient:
  strong convergence.
\newblock {\em ESAIM Probab. Stat.}, 12:1--11 (electronic), 2008.

\bibitem{MR1912205}
Andrei~N. Borodin and Paavo Salminen.
\newblock {\em Handbook of {B}rownian motion---facts and formulae}.
\newblock Probability and its Applications. Birkh\"auser Verlag, second
  edition, 2015.

\bibitem{2015arXiv150804581B}
M.~{Bossy} and H.~{Olivero Quinteros}.
\newblock {Strong convergence of the symmetrized Milstein scheme for some
  CEV-like SDEs}.
\newblock {\em ArXiv e-prints}, August 2015.

\bibitem{MR3582409}
Jean-Fran{\c{c}}ois Chassagneux, Antoine Jacquier, and Ivo Mihaylov.
\newblock An {E}xplicit {E}uler {S}cheme with {S}trong {R}ate of {C}onvergence
  for {F}inancial {SDE}s with {N}on-{L}ipschitz {C}oefficients.
\newblock {\em SIAM J. Financial Math.}, 7(1):993--1021, 2016.

\bibitem{MR785475}
John~C. Cox, Jonathan~E. Ingersoll, Jr., and Stephen~A. Ross.
\newblock A theory of the term structure of interest rates.
\newblock {\em Econometrica}, 53(2):385--407, 1985.

\bibitem{2015arXiv150901479C}
A.~{Cozma} and C.~{Reisinger}.
\newblock {A mixed Monte Carlo and PDE variance reduction method for foreign
  exchange options under the Heston-CIR model}.
\newblock {\em ArXiv e-prints}, September 2015.

\bibitem{2016arXiv160100919C}
A.~{Cozma} and C.~{Reisinger}.
\newblock {Exponential integrability properties of Euler discretization schemes
  for the Cox-Ingersoll-Ross process}.
\newblock {\em ArXiv e-prints}, December 2016.

\bibitem{MR1641781}
G.~Deelstra and F.~Delbaen.
\newblock Convergence of discretized stochastic (interest rate) processes with
  stochastic drift term.
\newblock {\em Appl. Stochastic Models Data Anal.}, 14(1):77--84, 1998.

\bibitem{MR2898556}
Steffen Dereich, Andreas Neuenkirch, and Lukasz Szpruch.
\newblock An {E}uler-type method for the strong approximation of the
  {C}ox-{I}ngersoll-{R}oss process.
\newblock {\em Proc. R. Soc. Lond. Ser. A Math. Phys. Eng. Sci.},
  468(2140):1105--1115, 2012.

\bibitem{2017arXiv170203229G}
M.~{Gerencs{\'e}r}, A.~{Jentzen}, and D.~{Salimova}.
\newblock {On stochastic differential equations with arbitrarily slow
  convergence rates for strong approximation in two space dimensions}.
\newblock {\em ArXiv e-prints}, February 2017.
\newblock To appear in Proc. R. Soc. Lond. Ser. A Math. Phys. Eng. Sci.

\bibitem{MR1997032}
Anja G\"oing-Jaeschke and Marc Yor.
\newblock A survey and some generalizations of {B}essel processes.
\newblock {\em Bernoulli}, 9(2):313--349, 2003.

\bibitem{MR1625576}
Istv{\'a}n Gy{\"o}ngy.
\newblock A note on {E}uler's approximations.
\newblock {\em Potential Anal.}, 8(3):205--216, 1998.

\bibitem{MR2822773}
Istv{\'a}n Gy{\"o}ngy and Mikl{\'o}s R{\'a}sonyi.
\newblock A note on {E}uler approximations for {SDE}s with {H}\"older
  continuous diffusion coefficients.
\newblock {\em Stochastic Process. Appl.}, 121(10):2189--2200, 2011.

\bibitem{MR3305998}
Martin Hairer, Martin Hutzenthaler, and Arnulf Jentzen.
\newblock Loss of regularity for {K}olmogorov equations.
\newblock {\em Ann. Probab.}, 43(2):468--527, 2015.

\bibitem{2016arXiv160101455H}
M.~{Hefter} and A.~{Herzwurm}.
\newblock {Optimal Strong Approximation of the One-dimensional Squared Bessel
  Process}.
\newblock {\em ArXiv e-prints}, January 2016.
\newblock To appear in Commun. Math. Sci.

\bibitem{2016arXiv160800410H}
M.~{Hefter} and A.~{Herzwurm}.
\newblock {Strong Convergence Rates for Cox-Ingersoll-Ross Processes - Full
  Parameter Range}.
\newblock {\em ArXiv e-prints}, August 2016.
\newblock To appear in J. Math. Anal. Appl.

\bibitem{2017arXiv171008707H}
M.~{Hefter}, A.~{Herzwurm}, and T.~{M{\"u}ller-Gronbach}.
\newblock {Lower Error Bounds for Strong Approximation of Scalar SDEs with
  non-Lipschitzian Coefficients}.
\newblock {\em ArXiv e-prints}, October 2017.

\bibitem{heston1993closed}
Steven~L Heston.
\newblock A closed-form solution for options with stochastic volatility with
  applications to bond and currency options.
\newblock {\em Review of financial studies}, 6(2):327--343, 1993.

\bibitem{MR1949404}
Desmond~J. Higham, Xuerong Mao, and Andrew~M. Stuart.
\newblock Strong convergence of {E}uler-type methods for nonlinear stochastic
  differential equations.
\newblock {\em SIAM J. Numer. Anal.}, 40(3):1041--1063 (electronic), 2002.

\bibitem{strathprints160}
D.J. Higham and X.~Mao.
\newblock Convergence of monte carlo simulations involving the mean-reverting
  square root process.
\newblock {\em Journal of Computational Finance}, 8(3):35--61, 2005.

\bibitem{MR1817611}
Norbert Hofmann, Thomas M\"uller-Gronbach, and Klaus Ritter.
\newblock The optimal discretization of stochastic differential equations.
\newblock {\em J. Complexity}, 17(1):117--153, 2001.

\bibitem{MR1396331}
Yaozhong Hu.
\newblock Semi-implicit {E}uler-{M}aruyama scheme for stiff stochastic
  equations.
\newblock In {\em Stochastic analysis and related topics, {V} ({S}ilivri,
  1994)}, volume~38 of {\em Progr. Probab.}, pages 183--202. Birkh\"auser
  Boston, Boston, MA, 1996.

\bibitem{2014arXiv1403.6385H}
M.~{Hutzenthaler}, A.~{Jentzen}, and M.~{Noll}.
\newblock {Strong convergence rates and temporal regularity for
  Cox-Ingersoll-Ross processes and Bessel processes with accessible
  boundaries}.
\newblock {\em ArXiv e-prints}, March 2014.

\bibitem{MR3364862}
Martin Hutzenthaler and Arnulf Jentzen.
\newblock Numerical approximations of stochastic differential equations with
  non-globally {L}ipschitz continuous coefficients.
\newblock {\em Mem. Amer. Math. Soc.}, 236(1112):v+99, 2015.

\bibitem{MR2985171}
Martin Hutzenthaler, Arnulf Jentzen, and Peter~E. Kloeden.
\newblock Strong convergence of an explicit numerical method for {SDE}s with
  nonglobally {L}ipschitz continuous coefficients.
\newblock {\em Ann. Appl. Probab.}, 22(4):1611--1641, 2012.

\bibitem{MR3538358}
Arnulf Jentzen, Thomas M{\"u}ller-Gronbach, and Larisa Yaroslavtseva.
\newblock On stochastic differential equations with arbitrary slow convergence
  rates for strong approximation.
\newblock {\em Commun. Math. Sci.}, 14(6):1477--1500, 2016.

\bibitem{MR1876169}
Olav Kallenberg.
\newblock {\em Foundations of modern probability}.
\newblock Probability and its Applications (New York). Springer-Verlag, New
  York, second edition, 2002.

\bibitem{MR1121940}
Ioannis Karatzas and Steven~E. Shreve.
\newblock {\em Brownian motion and stochastic calculus}, volume 113 of {\em
  Graduate Texts in Mathematics}.
\newblock Springer-Verlag, New York, second edition, 1991.

\bibitem{2016arXiv161004003K}
C.~{Kelly} and G.~J. {Lord}.
\newblock {Adaptive timestepping strategies for nonlinear stochastic systems}.
\newblock {\em ArXiv e-prints}, October 2016.

\bibitem{MR1475218}
Xuerong Mao.
\newblock {\em Stochastic differential equations and their applications}.
\newblock Horwood Publishing Series in Mathematics \& Applications. Horwood
  Publishing Limited, Chichester, 1997.

\bibitem{MR3433299}
Grigori~N. Milstein and John Schoenmakers.
\newblock Uniform approximation of the {C}ox-{I}ngersoll-{R}oss process.
\newblock {\em Adv. in Appl. Probab.}, 47(4):1132--1156, 2015.

\bibitem{MR2099646}
Thomas M\"uller-Gronbach.
\newblock Optimal pointwise approximation of {SDE}s based on {B}rownian motion
  at discrete points.
\newblock {\em Ann. Appl. Probab.}, 14(4):1605--1642, 2004.

\bibitem{MR3248050}
Andreas Neuenkirch and Lukasz Szpruch.
\newblock First order strong approximations of scalar {SDE}s defined in a
  domain.
\newblock {\em Numer. Math.}, 128(1):103--136, 2014.

\bibitem{MR1725357}
Daniel Revuz and Marc Yor.
\newblock {\em Continuous martingales and {B}rownian motion}, volume 293 of
  {\em Grundlehren der Mathematischen Wissenschaften [Fundamental Principles of
  Mathematical Sciences]}.
\newblock Springer-Verlag, Berlin, third edition, 1999.

\bibitem{MR3070913}
Sotirios Sabanis.
\newblock A note on tamed {E}uler approximations.
\newblock {\em Electron. Commun. Probab.}, 18:no. 47, 10, 2013.

\bibitem{MR3543890}
Sotirios Sabanis.
\newblock Euler approximations with varying coefficients: the case of
  superlinearly growing diffusion coefficients.
\newblock {\em Ann. Appl. Probab.}, 26(4):2083--2105, 2016.

\bibitem{2016arXiv160908073Y}
Larisa Yaroslavtseva.
\newblock On non-polynomial lower error bounds for adaptive strong
  approximation of {SDEs}.
\newblock {\em J. Complexity}, 42:1 -- 18, 2017.

\bibitem{2016arXiv160308686M}
Larisa Yaroslavtseva and Thomas M\"uller-Gronbach.
\newblock On sub-polynomial lower error bounds for quadrature of {SDE}s with
  bounded smooth coefficients.
\newblock {\em Stoch. Anal. Appl.}, 35(3):423--451, 2017.

\end{thebibliography}

\end{document}